\documentclass[english,american,dvipsnames]{article}
\usepackage[T1]{fontenc}
\usepackage{textcomp}
\usepackage[utf8]{inputenc}
\usepackage{color}
\usepackage{babel}
\usepackage{units}
\usepackage{url}
\usepackage{amsmath}
\usepackage{amsthm}
\usepackage{amssymb}
\usepackage{graphicx}
\usepackage{xargs}[2008/03/08]
\usepackage[all]{xy}
\usepackage[pdfusetitle,
 bookmarks=true,bookmarksnumbered=false,bookmarksopen=false,
 breaklinks=true,pdfborder={0 0 0},pdfborderstyle={},backref=false,colorlinks=true]
 {hyperref}
\hypersetup{
 linkcolor=BlueViolet, urlcolor=MidnightBlue, citecolor=RoyalPurple, unicode=true}

\makeatletter

\newcommand{\noun}[1]{\textsc{#1}}

\newcommand\thmsname{\protect\theoremname}
\newcommand\nm@thmtype{theorem}
\theoremstyle{plain}

\newenvironment{namedthm}[1][Undefined Theorem Name]{
  \ifx{#1}{Undefined Theorem Name}\renewcommand\nm@thmtype{theorem*}
  \else\renewcommand\thmsname{#1}\renewcommand\nm@thmtype{namedtheorem}
  \fi
  \begin{\nm@thmtype}}
  {\end{\nm@thmtype}}
\theoremstyle{plain}
\newtheorem{thm}{\protect\theoremname}[section]
\theoremstyle{remark}
\newtheorem{rem}[thm]{\protect\remarkname}
\newenvironment{lyxlist}[1]
	{\begin{list}{}
		{\settowidth{\labelwidth}{#1}
		 \setlength{\leftmargin}{\labelwidth}
		 \addtolength{\leftmargin}{\labelsep}
		 }}
	{\end{list}}
\theoremstyle{plain}
\newtheorem{prop}[thm]{\protect\propositionname}
\theoremstyle{plain}
\newtheorem{lem}[thm]{\protect\lemmaname}
\theoremstyle{definition}
\newtheorem{defn}[thm]{\protect\definitionname}
\theoremstyle{plain}
\newtheorem{cor}[thm]{\protect\corollaryname}
\theoremstyle{definition}
\newtheorem{example}[thm]{\protect\examplename}

\usepackage{tikz}
\usepackage{tikz-cd}

\makeatother

\addto\captionsamerican{\renewcommand{\corollaryname}{Corollary}}
\addto\captionsamerican{\renewcommand{\definitionname}{Definition}}
\addto\captionsamerican{\renewcommand{\examplename}{Example}}
\addto\captionsamerican{\renewcommand{\lemmaname}{Lemma}}
\addto\captionsamerican{\renewcommand{\propositionname}{Proposition}}
\addto\captionsamerican{\renewcommand{\remarkname}{Remark}}
\addto\captionsamerican{\renewcommand{\theoremname}{Theorem}}
\addto\captionsenglish{\renewcommand{\corollaryname}{Corollary}}
\addto\captionsenglish{\renewcommand{\definitionname}{Definition}}
\addto\captionsenglish{\renewcommand{\examplename}{Example}}
\addto\captionsenglish{\renewcommand{\lemmaname}{Lemma}}
\addto\captionsenglish{\renewcommand{\propositionname}{Proposition}}
\addto\captionsenglish{\renewcommand{\remarkname}{Remark}}
\addto\captionsenglish{\renewcommand{\theoremname}{Theorem}}
\providecommand{\corollaryname}{Corollary}
\providecommand{\definitionname}{Definition}
\providecommand{\examplename}{Example}
\providecommand{\lemmaname}{Lemma}
\providecommand{\propositionname}{Proposition}
\providecommand{\remarkname}{Remark}
\providecommand{\theoremname}{Theorem}

\begin{document}
\selectlanguage{english}%
\global\long\def\setminus{\smallsetminus}%

\global\long\def\tx#1{\mathrm{#1}}%
\global\long\def\dd#1{\tx d#1}%
\global\long\def\nf#1#2{\nicefrac{#1}{#2}}%

\global\long\def\ww#1{\mathbb{#1}}%
\global\long\def\group#1{{#1}}%
\global\long\def\bigslant#1#2{{\raisebox{.3em}{$#1$}/\raisebox{-.3em}{$#2$}}}%
 
\global\long\def\quot#1#2{\bigslant{#1}{#2}}%
\global\long\def\rr{\mathbb{R}}%
\global\long\def\cc{\mathbb{C}}%
\global\long\def\disc{\mathbb{D}}%
\global\long\def\zz{\mathbb{Z}}%
\global\long\def\zp{\mathbb{Z}_{\geq0}}%
\global\long\def\qq{\mathbb{Q}}%
\global\long\def\pol#1{\cc\left[#1\right] }%
\global\long\def\id{\tx{Id}}%
\global\long\def\GL#1#2{\tx{GL}_{#1}\left(#2\right)}%
\global\long\def\fol#1{\mathcal{F}_{#1}}%
\global\long\def\pp#1{\frac{\partial}{\partial#1}}%
\global\long\def\ii{\tx i}%
\global\long\def\ee{\tx e}%
\global\long\def\re#1{\text{Re}\left(#1\right)}%
\global\long\def\im#1{\text{Im}\left(#1\right)}%
\global\long\def\floor#1{\left\lfloor #1\right\rfloor }%
\global\long\def\ceiling#1{\left\lceil #1\right\rceil }%
\newcommandx\neigh[2][usedefault, addprefix=\global, 1=, 2=0]{\left(\cc^{#1},\group{#2}\right)}%
\global\long\def\longto{\longrightarrow}%
\global\long\def\longmaps{\longmapsto}%
\global\long\def\surj{\twoheadrightarrow}%
\global\long\def\cst{\tx{cst}}%
\global\long\def\oo#1{\tx o\left(#1\right)}%
\global\long\def\OO#1{\tx O\left(#1\right)}%
\newcommandx\diff[1][usedefault, addprefix=\global, 1={\cc,0}]{\text{Diff}\left(#1\right)}%

\global\long\def\areg{\mathcal{AR}}%
\global\long\def\dulac{\mathcal{D}}%
\global\long\def\unramif{\mathcal{U}}%
 
\global\long\def\RingQSD{\mathcal{Q}}%
 
\global\long\def\rDulac{\areg}%
 
\global\long\def\rDulacf{\widehat{\rDulac}}%
 
\global\long\def\Dulac{\dulac}%
 
\global\long\def\Dulacf{\widehat{\Dulac}}%
 
\global\long\def\DerDulacf{\mathrm{N}\left(\Dulacf\right)}%

\global\long\def\modul#1{\mathfrak{m}\left(#1\right)}%
\newcommandx\model[3][usedefault, addprefix=\global, 1=k, 2=\mu]{#3_{#1,#2}}%
\global\long\def\var#1{\tx{var}\left(#1\right)}%
\global\long\def\Var#1{\Pi\left(\var{#1}\right)}%
\global\long\def\floc{\fol{}}%
\global\long\def\cloc{f_{\star}}%

\global\long\def\tmop#1{\ensuremath{\operatorname{#1}}}%

\global\long\def\cR{\ensuremath{\mathbb{R}}}%
 
\global\long\def\cC{\ensuremath{\mathbb{C}}}%
 
\global\long\def\cZ{\ensuremath{\mathbb{Z}}}%
 
\global\long\def\cQ{\ensuremath{\mathbb{Q}}}%

\global\long\def\bL{\ensuremath{\mathbb{L}}}%

\global\long\def\vphi{\varphi}%

\global\long\def\eps{\varepsilon}%

\global\long\def\supp{\mathrm{supp}}%

\global\long\def\Exp#1{\exp\,#1}%
 
\global\long\def\Log#1{\mathrm{log}\,{#1}}%
 
\global\long\def\pP{\mathbf{P}}%

\global\long\def\hol{\mathrm{hol}}%

\global\long\def\idD{\id}%
 
\global\long\def\Diff{\mathrm{Diff}}%
 
\global\long\def\Difff{\widehat{\Diff}}%
 
\global\long\def\DerDifff{\mathrm{N}\left(\Difff(\cC,0)\right)}%
 
\global\long\def\Differ{\Delta}%

\global\long\def\varD{\mathrm{var}}%
 
\global\long\def\lvar{\mathrm{lvar}}%
 
\global\long\def\DVar#1{\mathrm{Dvar}({#1})}%
 
\global\long\def\Center{\mathrm{Center}}%
 
\global\long\def\Cent{\mathfrak{C}}%

\global\long\def\transl{\mathbf{t}}%
 
\global\long\def\scale{\mathbf{s}}%
 
\global\long\def\pow{\mathbf{p}}%
 
\global\long\def\Fatou{\mathfrak{F}}%

\global\long\def\U{\unramif}%
 
\global\long\def\MR{\mathcal{M}}%
 
\global\long\def\Uf{\widehat{\U}}%
 
\global\long\def\MRf{\widehat{{\MR}}}%
 
\global\long\def\varf{\widehat{\mathrm{var}}}%

\global\long\def\Orbit{\mathrm{Orbit}}%
 
\global\long\def\QSD{\mathrm{QSD}}%

\global\long\def\Tay{\mathrm{T}}%

\global\long\def\vF{\fol{}}%

\global\long\def\vG{\mathcal{G}}%

\global\long\def\bD{\bar{\Delta}}%

\global\long\def\Fib{\mathrm{Fib}}%

\global\long\def\vR{\mathcal{R}}%

\global\long\def\cut{\mathrm{cut}}%
 
\global\long\def\hhol{{\holt{}}}%

\global\long\def\nobracket{\{\}}%
 
\global\long\def\textdots{...}%
 
\global\long\def\tmem#1{{\em#1\/}}%
 
\global\long\def\tmstrong#1{\textbf{#1}}%
 
\global\long\def\tmtextit#1{\text{{\itshape{#1}}}}%

\global\long\def\Corner{\textrm{\ensuremath{\tmop{Corner}}}}%
 
\global\long\def\Poinc{\tmop{Poinc}}%
 
\global\long\def\GlueM{G}%

\global\long\def\holt#1{\mathrm{hol}_{#1}}%
\newcommandx\holo[1][usedefault, addprefix=\global, 1=\gamma]{\mathfrak{h}_{#1}}%
\selectlanguage{american}%

\title{Rigidity of saddle loops}
\author{Daniel \noun{Panazzolo}\thanks{The first author is partially supported by the ANR/Foliage project
and the bilateral Hubert-Curien Cogito grant 2021-22-23-24 and CSF
IP-2022-10-9820 grant and the project NonSperANR-23-CE40-0028. \protect \\
\protect\url{daniel.panazzolo@uha.fr}\protect \\
 University of Haute-Alsace}, Maja \noun{Resman}\thanks{The second author is partially supported by the Croatian Science Foundation
(HRZZ) grants UIP-2017-05-1020 and PZS-2019-02-30, and the bilateral
Hubert-Curien Cogito grant 2021-22-23-24 and CSF IP-2022-10-9820 grant
and Horizon grant 101183111-DSYREKI-HORIZON-MSCA-2023-SE-01.\protect \\
\protect\url{maja.resman@math.hr}\protect \\
 University of Zagreb}, Loïc \noun{Teyssier}\thanks{The third author is partially supported by the bilateral Hubert-Curien
Cogito grant 2021-22-23-24 and CSF IP-2022-10-9820 grant.\protect \\
\protect\url{teyssier@math.unistra.fr}\protect \\
 University of Strasbourg}}
\date{Spring 2025}
\maketitle
\begin{abstract}
A saddle loop is a germ of a holomorphic foliation near a homoclinic
saddle connection. We prove that they are classified by their Poincaré
first-return map. We also prove that they are formally rigid when
the Poincaré map is multivalued. Finally, we provide a list of all
analytic classes of Liouville-integrable saddle loops.
\end{abstract}
\begin{namedthm}[2010~MSC Classification]
\emph{32S65,34C07,37F75}
\end{namedthm}
\tableofcontents{}

\section{Introduction}

Let $X$ be a real analytic vector field defined on an open subset
$U$ of $\mathbb{R}^{2}$ and let $\Gamma\subset U$ be an invariant
subset formed by a saddle point and a solution curve connecting two
local branches of the stable and unstable separatrices. We will say
that the triple $(U,X,\Gamma)$ is a {\tmem{real planar saddle loop}}.

The qualitative behavior of the solution curves in the vicinity of
$\Gamma$ is encoded by fixing an analytic transverse section $\Sigma$
through a regular point $\sigma\in\Gamma$ and considering the Poincaré
return map, 
\[
P:(\Sigma_{\geqslant0},\sigma)\rightarrow(\Sigma_{\geqslant0},\sigma).
\]
We denote by $\Sigma_{\geqslant0}$ the positive part of $\Sigma$,
with respect to a conveniently chosen parameterization (see figure
below) and the notation $(V,p)$ indicates some open neighborhood
of a point $p$ in a topological vector space $V$.

\begin{figure}[htb]
\centering{}\includegraphics[width=7.05905cm,height=2.958022cm]{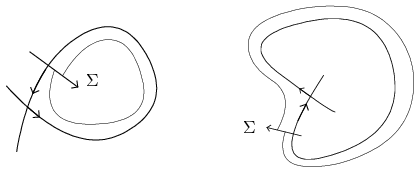}
\end{figure}

The two main goals of this article are to study the analytic classification
of such maps and relate it to the classification of the germ of $X$
along $\Gamma$. More precisely, consider two analytic planar saddle
loops $(U,X,\Gamma)$, $(\tilde{U},\tilde{X},\tilde{\Gamma})$, and
let $P,\tilde{P}$ denote the underlying Poincaré maps. We introduce
the following equivalence relations.
\begin{enumerate}
\item $(U,X,\Gamma)$ and $(\tilde{U},\tilde{X},\tilde{\Gamma})$ are {\tmem{analytically
equivalent}} if, up to changing $U,\tilde{U}$ for smaller neighborhoods
of $\Gamma,\tilde{\Gamma}$, there exists a real analytic diffeomorphism
\[
\Phi:\tilde{U}\rightarrow U
\]
with $\Phi(\tilde{\Gamma})=\Gamma$, mapping the solution curves of
$\tilde{X}$ into the solution curves of $X$ and preserving the orientation
(but not necessarily the natural parameterization by the time).
\item $P,\tilde{P}$ are {\tmem{analytically conjugate}} if there exists
a germ of a real analytic diffeomorphism 
\[
\varphi:(\tilde{\Sigma},\tilde{\sigma})\rightarrow(\Sigma,\sigma)
\]
preserving the orientation (\emph{i.e.}\ mapping $\tilde{\Sigma}_{\geqslant0}$
to $\Sigma_{\geqslant0}$) such that $\tilde{P}=\varphi^{-1}P\varphi$.
\end{enumerate}
The two main problems that we want to address are the following.
\begin{namedthm}[Problem~1]
\begin{flushleft}
\emph{Describe the analytic conjugacy classes of Poincaré maps for
saddle loops.}
\par\end{flushleft}

\end{namedthm}
\begin{namedthm}[Problem~2]
\emph{Show that $(U,X,\Gamma)\thicksim(\tilde{U},\tilde{X},\tilde{\Gamma})$
if and only if $P\thicksim\tilde{P}$.}
\end{namedthm}
Concerning Problem~1, we observe that related problems have been
studied by P.~Mardešić, D.~Peran \emph{et~al.}~\cite{MardeResReal,MardeResClassif,PerResRolSerForm,PerResRolSerAn,Per}
but with respect to the action of a much larger group of germs admitting
logarithmic asymptotic expansions (generalizations of the so-called
{\tmem{Dulac germs}}). As we shall see, the results are strikingly
different with respect to the present classification.

To put Problem~2 in some context, notice that, if we assume $\Gamma,\tilde{\Gamma}$
to be {\tmem{periodic orbits}}, the above statement is a classical
result, as both $(U,X,\Gamma)$ and $(\tilde{U},\tilde{X},\tilde{\Gamma})$
are analytically equivalent to suspensions of $P$ and $\tilde{P}$
(see \emph{e.g.} \cite{Smale1963}).

In our case, due to the presence of singular points, there is no clear
analogue of this construction by suspension. We take an indirect approach
by complexifying the problem and working with the associated complex
holomorphic foliations.

In Sections~\ref{sub:first} and~\ref{sub:second}, we discuss separately
Problem~2 and Problem~1, respectively.

\subsection{\protect\label{sub:first}Real and complex saddle loops}

Since our approach involves considering complex holomorphic foliations,
one is naturally led to consider the following more general set-up.
Let $(S,\mathcal{G})$ be a smooth holomorphic surface equipped with
a singular foliation. A {\tmem{saddle loop}} for $(S,\mathcal{G})$
is defined by a saddle singularity $s\in S$ and an oriented $C^{1}$-path
$\Gamma:[-1,1]\rightarrow S$ such that:
\begin{enumerate}
\item $\Gamma$ is everywhere tangent to $\mathcal{G}$ and $\lim_{t\rightarrow\pm1}\Gamma(t)=\{s\}$;
\item there exists some $\varepsilon>0$ such that $\Gamma|_{]-1,-1+\varepsilon[}$
and $\Gamma|_{]1-\varepsilon,1[}$ lie on distinct local separatrices.
\end{enumerate}
In other words, we assume that the two local separatrices at $s$
lie on a common global leaf $\mathcal{L}_{s}$ of $\mathcal{G}$ and
fix a path $\Gamma\subset\mathcal{L}_{s}$ which accumulates on $s$
along distinct separatrices as $t\rightarrow\pm1$. As above, we fix
a local holomorphic transverse section $(\Sigma,\sigma)$ at a point
$\sigma\in\Gamma$. However, contrary to the real setting, there is
no natural choice of a Poincaré first return map $P$, since the leaves
are not 1-dimensional and there is no natural order on the multiple
crossing points between leaves and transversals.

Since difficulties concentrate in the vicinity of the saddle point,
\ it is convenient to place the transverse section $(\Sigma,\sigma)$
near $s$, with base-point $\sigma$ on one of the local separatrices,
and introduce an auxiliary transverse section ($\Omega,\omega$),
with base-point $\omega$ lying on the other separatrix.

\begin{figure}[htb]
\centering{}\raisebox{-0.00321198471797896\height}{\includegraphics[width=2.745884cm,height=2.61385cm]{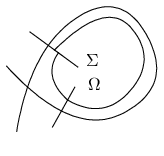}}
\end{figure}

In this setting, we write the factorization $P=RD$, where the {\tmem{regular
transition}} 
\[
R:(\Omega,\omega)\rightarrow(\Sigma,\sigma)
\]
is a well-defined, injective holomorphic map obtained by lifting the
sub-path of $\Gamma$ linking $\Omega$ to $\Sigma$ to the nearby
leaves, and the so-called {\tmem{corner transition}}\footnote{In many references, $D$ is called a Dulac map, but we prefer to keep
this name for a much larger class of maps which will appear very frequently
in the present paper.} $D$ is a multivalued map establishing a point-wise correspondence
between $\Sigma$ and $\Omega$. The motivation for this construction
is that we can study $D$ by using the classical Poincaré-Dulac local
normal form theory.

For instance, to fix a determination of such corner transition we
proceed as follows. Choose an oriented path lying on a sufficiently
close regular leaf $\mathcal{L}$ connecting a point $p\in\Sigma\cap\mathcal{L}$
to a point $q\in\Omega\cap\mathcal{L}$ (we call it a {\tmem{guiding
path}}). By holonomy transport, this path uniquely determines a
map 
\[
D:(\Sigma,p)\rightarrow(\Omega,q)
\]
which (by a very elegant construction of Ilyashenko~\cite{IlYaRus})
has an holomorphic extension to a map between domains lying in the
universal coverings $\widetilde{\Sigma\setminus\{\sigma\}}$ and $\widetilde{\Omega\setminus\{\omega\}}$.

It is important to emphasize that, except for some rather special
cases, the Poincaré map \ cannot be extended as a holomorphic map
on $\Sigma$ itself, and the passage to the universal covering is
necessary due to the intrinsic multivaluedness of $P$.

Of course, different choices of guiding path can lead to different
determinations of $D$ (and hence of $P$), and one needs to take
this choice into account in the definition of equivalence between
complex saddle loops. We refer to Section~\ref{sub:three} for the
details. \medskip{}

The above construction suggests to work with a more abstract model
for a saddle loop, where one considers a pair of a germ of a complex
saddle foliation $\vF$ in $(\mathbb{C}^{2},0)$, namely a foliation
defined by a differential 1-form 
\begin{equation}
\omega=x\dd y+y(\lambda-K(x,y))\dd x,\quad\lambda\in\cR_{>0},~K\in\cC\{x,y\},\label{eq:local_saddle_and_gluing}
\end{equation}
equipped with two transversals $\Omega,\Sigma$ through the separatrices
$\{y=0\}$ and $\{x=0\}$, and a germ of a holomorphic map $R:(\Omega,\omega)\rightarrow(\Sigma,\sigma)$,
seen intuitively as a recipe for {\tmem{gluing}} these transversals.
The scalar $-\lambda$ is called the \emph{eigenratio} of $\fol{}$.

One of the important points is that, while we can always fix the position
of one of these transversals, assuming for instance that $\Omega=\{x=1\}$,
the position of the other transversal $\Sigma$ should be allowed
to vary, while still staying inside the same equivalence class. More
precisely, we identify two such gluing maps 
\[
R:(\Omega,1)\rightarrow(\Sigma,\sigma),\quad\text{and}\quad\widetilde{R}:(\Omega,1)\rightarrow(\widetilde{\Sigma},\widetilde{\sigma})
\]
if there exists a germ of a $\vF$-holonomy map, $h:(\Sigma,\sigma)\rightarrow(\widetilde{\Sigma},\widetilde{\sigma})$
such that $\widetilde{R}=h\,R$.

\begin{figure}[htb]
\centering{}\raisebox{-0.00321198471797896\height}{\includegraphics[height=4.2cm]{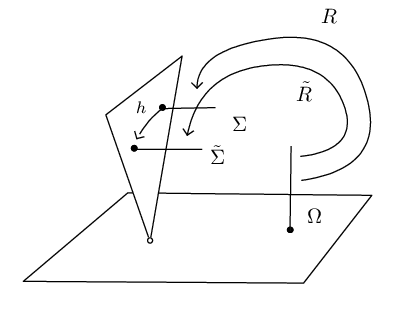}}
\end{figure}

In other words, we define a {\tmem{germ of a saddle loop}} (or,
shortly, a \emph{loop germ}) as a pair $\mathbb{L}=(\mathcal{F},\vR)$,
where:
\begin{enumerate}
\item $\vF$ is a prepared germ of a complex saddle foliation (see Section
\ref{subsect:preparedsaddle});
\item $\vR$ is an {\em equivalence class} of germs of a diffeomorphism
between the fixed vertical transversal $\Omega=\{x=1\}$ and a varying
horizontal transversal $\Sigma$, with respect to the equivalence
relation defined above.
\end{enumerate}
We will say that each germ $R\in\vR$ is a {\em determination}
of the regular transition. We refer to Section~\ref{sec:gcc} for
the detailed definitions and to Section~\ref{sec:realization} for
some examples. \medskip{}

To a loop germ $\mathbb{L}$, we associate the {\em class of Poincaré
first return maps} as the set, denoted by $\Poinc({\vF},{\vR})$,
of all possible compositions 
\[
P=R\,D,
\]
where $R:(\Omega,1)\rightarrow(\Sigma,\sigma)$ and $D:(\Sigma,\sigma)\rightarrow(\Omega,1)$
are arbitrary determinations of the regular transition and corner
transition, respectively.

In this new setting, the following generalization of the Problem~2
seems quite natural.
\begin{namedthm}[Problem~2']
\emph{Suppose that two germs of a saddle loop $\mathbb{L},\tilde{\mathbb{L}}$
have (some determinations of) their Poincaré maps which are analytically
conjugate. Does this imply an analytic equivalence between $\mathbb{L}$
and $\tilde{\mathbb{L}}$?}
\end{namedthm}
More precisely, we consider the action of the group $\tmop{Diff}_{\tmop{fib}}(\mathbb{C}^{2},\bar{\Delta})$
of germs of biholomorphisms defined in a neighborhood of the closed
unit disk $\bar{\Delta}\times\{0\}$ in $\mathbb{C}^{2}$, preserving
the fibration $\{x=\cst\}$, and we say that two loop germs $\bL=(\mathcal{F},\vR)$
and $\widetilde{\bL}=(\widetilde{\mathcal{F}},\widetilde{\vR})$ are
{\em $\tmop{Diff}_{\tmop{fib}}(\mathbb{C}^{2},\bar{\Delta})$-equivalent}
if there exists a germ of a $x$-fibered biholomorphism $\Phi\in\tmop{Diff}_{\tmop{fib}}(\mathbb{C}^{2},\bar{\Delta})$
such that:
\begin{enumerate}
\item $\Phi(\vF)=\widetilde{\mathcal{F}}$;
\item $\Phi$ maps the class $\vR$ to the class $\widetilde{\vR}$.
\end{enumerate}
The last requirement means that there exist suitable representatives
$R\in\vR$ and $\tilde{R}\in\widetilde{\vR}$ which are conjugate
under the restriction of $\Phi$ to the appropriate transversals.
\medskip{}

We are now ready to enunciate the two main equivalence results for
loop germs.
\begin{namedthm}[Theorem A]
Consider two $\tmop{Diff}_{\tmop{fib}}(\mathbb{C}^{2},\bar{\Delta})$-equivalent
loop germs $\bL=\left(\vF,\vR\right)$ and $\widetilde{\bL}=(\widetilde{\vF},\widetilde{\vR})$.
Then, there exist representatives of such germs in suitable neighborhoods
of the origin in $\cC^{2}$ such that each Poincaré map $P\in\Poinc(\vF,\vR)$
is analytically conjugate to a Poincaré map $\tilde{P}\in\Poinc(\widetilde{\vF},\widetilde{\vR})$.
\end{namedthm}
In fact, a more refined version of Theorem~A, stated in Section \ref{sec:gcc},
explicitly describes the correspondence between $P$ and its ``associate''
$\tilde{P}$. \smallskip{}

We also establish a converse to this statement.
\begin{namedthm}[Theorem~B]
Let $(\widetilde{\vF},\widetilde{\vR})$ and $\left(\vF,\vR\right)$
be two loop germs, and suppose that there exist two Poincaré maps,
\[
P\in\Poinc(\vF,\vR)~~\text{and}~~\tilde{P}\in\Poinc(\widetilde{\vF},\widetilde{\vR}),
\]
which are analytically conjugate. Then, $(\widetilde{\vF},\widetilde{\vR})$
and $\left(\vF,\vR\right)$ are $\tmop{Diff}_{\tmop{fib}}(\mathbb{C}^{2},\bar{\Delta})$-equivalent.
\end{namedthm}
The basic idea to prove the latter result is quite simple. As we shall
see, the conjugacy of the Poincaré maps implies the conjugacy of the
local holonomies at the saddle points. Hence, one can use the Mattei-Moussu
theorem~\cite{MaMou} to construct a fibered local equivalence between
the foliations near the saddle points by path lifting.

The subtle point is that such equivalence does not necessarily preserve
both transversals, and this is one of the reasons for having introduced
the more general concept of a loop germ, thus allowing one of the
transversals to {\em move}. \smallskip{}

Fortunately, in the case of {\em real} saddle loops, the structure
is more rigid. In particular, the local equivalence obtained by the
Mattei-Moussu theorem respects the real structure, so one can better
control how the other transversal moves under such an equivalence
map. As a consequence, we obtain the following positive answer to
Problem~2. Here, by the Poincaré map of a real saddle loop we mean
the canonically chosen determination which preserves the real line.
\begin{namedthm}[Theorem~C]
Two real saddle loops $(U,X,\Gamma)$, $(\tilde{U},\tilde{X},\tilde{\Gamma})$
are analytically equivalent if and only if the associated Poincaré
maps $P,\tilde{P}$ are analytically conjugate.
\end{namedthm}

\subsection{\protect\label{sub:second}Rigidity of Poincaré maps}

We now turn to the problem of classifying Poincaré first return maps.
For a loop germ $\mathbb{L}$ as above, consider an associated Poincaré
map 
\[
P=RD.
\]
Following~\cite{IlyaDu}, we observe that the lift of $P$ to the
logarithmic chart $x=\ee^{-z}$ defines a {\tmem{Dulac germ}}:\ a
holomorphic function $p$ on a quadratic standard domain and having
a formal asymptotic expansion in the pol-exp scale, 
\[
p(z)\thicksim az+b+\sum_{k\geqslant1}P_{k}(z)\ee^{-\lambda_{k}z},
\]
where $a>0$, $b\in\mathbb{C}$, $P_{k}\in\mathbb{C}[z]$ and $\{\lambda_{k}\}_{k}$
is an increasing sequence of positive real numbers. We denote by $\Dulac$
the group of Dulac germs and by $\widehat{\Dulac}$ is its {\tmem{formal}}
counterpart, consisting in all formal pol-exp series as above. One
fundamental fact, also due to Y.~Ilyashenko, is that the {\tmem{Taylor
map}} $\Tay:\Dulac\rightarrow\widehat{\Dulac}$, which associates
to each Dulac germ its asymptotic expansion, is injective. We refer
to Section \ref{subsec:ramified-var} for detailed statements.
\begin{rem}
We note that M.~Yeung~\cite{MYeung} recently identified a gap in\linebreak{}
Ilyashenko's proof of Dulac's problem. However, this does not affect
the results used in the present paper, since we only deal with purely
hyperbolic polycycles.
\end{rem}

In order to characterize the Dulac germs originating from Poincaré
maps of saddle loops, we consider the so-called {\tmem{functional
variation operator}} $\tmop{var}:\Dulac\rightarrow\Dulac$, 
\[
\tmop{var}(d)=[\tau,d],
\]
where $[a,b]=a^{-1}b^{-1}ab$ is the commutator operator and $\tau$
is the translation map $z\mapsto z+2\pi\ii$ (also seen as an element
of $\Dulac$). We say that a Dulac germ $d\in\Dulac$ is {\tmem{unramified}}
(noted $d\in\U$) if 
\begin{equation}
\tmop{var}(d)=\id,\label{eq:var}
\end{equation}
and we say that $d$ is {\tmem{mildly ramified}} (noted $d\in\MR$)
if 
\begin{equation}
\tmop{var}(\tmop{var}(d))=\id.\label{eq:varvar}
\end{equation}
The following result establishes a correspondence between these notions
and the maps considered above.
\begin{namedthm}[Proposition. (Dictionary)]
~
\begin{enumerate}
\item A Dulac germ is the lift of an analytic germ from $\mathrm{Diff}(\mathbb{C},0)$
if and only if it is unramified.
\item A Dulac germ is the lift of a Poincaré map of a germ of a complex
saddle loop if and only if it is mildly ramified.
\end{enumerate}
\end{namedthm}
Item~1. is an immediate consequence of the definition of unramified
germs (Proposition~\ref{prop:unramified}). Item~2. uses the nontrivial
realization theorems proved in~\cite{MaRa-Res} and~\cite{PY}:
every mildly ramified germ $f$ can be realized as a corner map of
a saddle foliation, because its variation $\var f$ can be realized
as the holonomy of a saddle foliation $\fol{}$ computed on a horizontal
transversal. By using properties of the variation operator on the
spaces $\Dulac$ and $\MR$, stated in Remark~\ref{rem:added}, one
can establish that the holonomy on the vertical transversal belongs
to the same conjugacy class as $\var{f^{-1}}$, providing a gluing
germ $R$ such that the loop germ $\left(\fol{},\mathcal{R}\right)$
admits $f$ as Poincaré map.

\medskip{}

We observe that equations~\eqref{eq:var} and~\eqref{eq:varvar}
also make perfect sense at the formal level, and define two subsets
$\widehat{\U}$ and $\widehat{\MR}$ of $\widehat{\Dulac}$, so-called
{\tmem{unramified}} and {\tmem{mildly ramified}} formal Dulac
series. Moreover, since the Taylor map induces a group morphism between
$\Dulac$ and $\Dulacf$, one has
\[
\Tay(\U)\subset\widehat{\U}\quad\tmop{and}\quad\Tay(\MR)\subset\widehat{\MR}.
\]
Based on the above dictionary, we can formulate our contribution to
the Problem~1 stated above purely in terms of properties of $\U$,
$\MR$ and their formal counterparts.
\begin{namedthm}[Theorem~D]
 Let $d_{1},d_{2}\in\MR$ be two mildly ramified Dulac germs which
are $\widehat{\U}$-conjugate. Then, one of the following alternatives
holds:
\begin{enumerate}
\item $d_{1}$ and $d_{2}$ belong to $\U$;
\item $d_{1},d_{2}$ belong to $\MR\setminus\U$ and they are $\U$-conjugate.
\end{enumerate}
\end{namedthm}
More precisely, we prove the following \emph{strong rigidity} property.
Suppose that a formal series $\varphi\in\widehat{\U}$ satisfies the
relation 
\[
d_{1}=\varphi^{-1}\,d_{2}\,\varphi
\]
for some $d_{1},d_{2}\in\MR\setminus\U$. Then $\varphi\in\U$.

We immediately obtain the following consequence.
\begin{namedthm}[Corollary]
Consider two Poincaré maps of saddle loop germs $P_{1},P_{2}$ which
are ramified, and suppose that there exists a formal diffeomorphism
$\phi\in\widehat{\tmop{Diff}(\mathbb{C},0)}$ conjugating $P_{1}$
to $P_{2}$. Then $\phi$ converges.
\end{namedthm}
We observe that any germ of a diffeomorphism $P\in\tmop{Diff}(\mathbb{C},0)$
can be realized as an {\tmem{unramified}} Poincaré first return
map of a germ of a saddle loop. This is for instance the case of a
loop germ obtained by gluing a linear $1:1$ saddle (whose canonical
corner transition map is the identity) by $P$. Therefore, according
to the well-known theories of Birkhoff--Écalle--Voronin (resonant
diffeomorphisms) and Yoccoz (quasi-resonant diffeomorphisms), one
cannot expect to have a similar rigidity result if we do not assume
$P_{1},P_{2}$ to be ramified.
\begin{rem}
As kindly requested by a referee, Subsection~\ref{subsec:novo} addresses
the problem of topological rigidity for complex Dulac germs. Our results
suggest that mildly ramified Poincaré maps are generically strongly
topologically rigid. Actually, it might even be true that formal and
topological rigidity coincide for saddle loop Poincaré maps, in the
sense that a Dulac map $d_{1}\in\mathcal{M}\backslash\mathcal{U}$
topologically conjugate to some other $d_{2}$, is analytically conjugate
to the latter (albeit by a different mapping).
\end{rem}

\subsection{Complex saddle loops in birational models and symplectic foliations}

Complex saddle loops appear quite frequently in the problem of birational
classification of holomorphic foliations. For instance, suppose that
a foliated surface $(S,\mathcal{G})$ contains a singular invariant
set which is a nodal rational curve and that the nodal point is a
saddle singularity for $\mathcal{G}$. Then $(S,\mathcal{G})$ contains
a saddle loop.

A saddle loop also appears in the ``very special foliation'' described
by Brunella in~\cite[Section 4.2]{Brunella2015} as one of the class
of foliations without a minimal model. We remark in passing that the
existence of an invariant set of the above type imposes strong restrictions
if we assume $(S,\mathcal{G})$ to be {\tmem{algebraic}}. We refer
to~\cite[Example 7.1]{Brunella2015} for more details.

Another motivation to consider complex saddle loops comes from the
related problem of classifying singular Liouville foliations of focus-focus
type in $4$-dimensional symplectic manifolds $(M,\omega)$ (see \emph{e.g.}~\cite{VNgoc2003},~\cite{Bolsinov2019},~\cite{Duistermaat1980},
or~\cite{Smirnov-focus}). More precisely, one considers germs of
a completely integrable singular Liouville foliation $\fol{}$, defined
in the vicinity of a singular compact leaf $\mathcal{L}\subset M$
which is homeomorphic to a {\tmem{pinched torus}}.
\begin{center}
{\includegraphics[width=5.19328cm,height=5.42959cm]{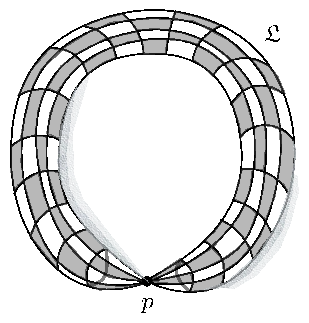}}
\par\end{center}

Under generic hypothesis, it follows from a theorem of Eliasson~\cite{Elia}
that one can find canonical local complex coordinates $(z,w)\in\mathbb{C}^{2}$
in a neighborhood of the pinch point $p$ such that the leaves of
$\fol{}$ are given as the level sets $\left\{ zw=\cst\right\} $.
In other words, we locally obtain a saddle type singularity for a
holomorphic foliation.

In this context, we can define a {\tmem{symplectic}} saddle loop
by considering any path $\gamma\subset\mathcal{L}$ which connects
$p$ to itself and is not homotopic to a the trivial path. \ Notice
that the global foliation $\fol{}$ is not necessarily complex holomorphic
(as $(M,\omega)$ is not necessarily equipped with a compatible complex
structure) but we expect that our results will give new invariants
for the above classification, notably if one considers non-integrable
perturbations of the completely integrable case.

\subsection{Integrability of loop germs in the Liouvillian class}

The discussion about symplectic saddle loops naturally prompts to
consider integrable settings. The usual framework to speak about integrable
holomorphic foliations on complex surfaces is that of Liouvillian
first integrals. A \emph{Liouvillian function} lies in a finite tower
of extensions of differential fields, starting from meromorphic germs,
of the following three types: algebraic, integral or exponential.
A classical reference is~\cite{SingLiou}.

A loop germ $\left(\fol{},\mathcal{R}\right)$ is said to be \emph{integrable
}whenever there exists a (non-constant) Liouvillian first integral
$H$ of $\fol{}$, meaning that the leaves of $\fol{}$ are included
in level sets $\left\{ H=\cst\right\} $, that is compatible with
a gluing map $R\in\mathcal{R}$. When we say that $H$ and $R$ are
\emph{compatible,} we require that $R$ preserve the ``nice'' transverse
structure provided by $H$. We refer to Section~\ref{sec:integrability}
for a precise definition.

The theory developed by M.~Berthier and F.~Touzet in~\cite{BerTouze}
classifies the saddle foliations $\fol{}$ admitting Liouvillian first
integrals $H$. We carry out here the final step that allows us to
list all compatible gluing maps $R\in\diff$ leading to integrable
loop germs.
\begin{namedthm}[Theorem~E]
Any integrable loop germ $\ww L=\left(\fol{},\mathcal{R}\right)$
is $\tmop{Diff}_{\tmop{fib}}(\mathbb{C}^{2},\bar{\Delta})$-equivalent
to one appearing in the following list.
\begin{enumerate}
\item $\mathcal{F}$ is a linear $1:\lambda$ saddle foliation, $\lambda>0$,
and $R\in\mathcal{R}$ is a linear map.
\item $\fol{}$ is a linear $1:1$ saddle foliation and $R\in\mathcal{R}$
is a Bernoulli diffeomorphism.
\item $\mathcal{F}$ is a non-linear $1:1$ saddle foliation in Poincaré-Dulac
normal form and $R\in\mathcal{R}$ embeds in a holomorphic flow.
\end{enumerate}
\end{namedthm}
Section~\ref{sec:integrability} contains the detailed list of the
model loop germs and the precise definitions of the terms involved
above. For the sake of readability, let us simply say that a Bernoulli
diffeomorphism is a resonant diffeomorphism whose Écalle--Voronin
cocycles~\cite{VoroParaboEng,EcalParab} are ramified homographies
(with same ramification order), whereas $R$ embeds in a holomorphic
flow if its cocycles are all trivial.

Because integrable foliations have solvable holonomy groups, and since
solvable finitely-generated subgroups of $\diff$ are very rare, one
is not surprised to encounter only so few integrable foliations. What
may seem more surprising is that we only encounter abelian holonomy
groups in the list. Let us try to give an intuitive explanation of
the fact. It is clear that a linear foliation can be glued by a linear
map $R$ without destroying the transverse structure (a linear map
commutes with the other linear data), and in that case the holonomy
group is itself linear and commutative. In contrast, the nonlinear
terms of $R$ have to conciliate the different behaviors near distinct
local branches of the separatrices passing through the saddle point.

On the one hand, when $\lambda\neq1$, the roles played by the branches
of the local separatrices are asymmetric, which hints at why only
linear gluing maps are compatible with the existence of a Liouvillian
first integral. Nonlinearities in the foliation pose the same problems,
and indeed a nonlinear $p:q$ saddle in Poincaré-Dulac normal form,
given as the solutions to the differential equation
\begin{align*}
\frac{\dd y}{y} & =\frac{1+\mu\left(x^{q}y^{p}\right)^{k}}{\left(x^{q}y^{p}\right)^{k+1}}\dd{\left(x^{q}y^{p}\right)},~~k\in\zz_{>0},~\mu\in\cc,
\end{align*}
even when glued with $R=\id$, is \emph{not} an integrable loop germ.

On the other hand, when $\lambda=1$ the additional symmetry gives
rise to richer outcomes. Since the holonomy group is generated by
tangent to the identity mappings, it is well known that solvability
is equivalent to commutativity. It is worth noticing, though, that
not every abelian holonomy group that arise in loop germs leads to
an integrable situation. Proposition~\ref{prop:GH} gives the complete
list of loop germs with abelian holonomy around a saddle $1:1$ foliation.
In addition to the cases described by Theorem~E, one finds the exceptional
case of foliations that are formally conjugate to the Poincaré-Dulac
foliation with $\mu=\frac{1}{2}$, and having both generators of the
holonomy group being equal. The holonomy group is thus commutative
for trivial reasons having nothing to do with the potential integrability
of the underlying foliations (the realization theorem of~\cite{MaRa-Res}
can be used again to prove that every candidate diffeomorphism is
actually realizable in a loop germ), and it turns out that none of
them correspond to an actual Liouville-integrable loop germ, except
those that embed in a holomorphic flow.

\subsection{Saddle quantities, separatrix values and real classification}

Coefficients of the Poincaré-Dulac normal form for $1:\lambda$ saddle
singularities, called saddle quantities, and Taylor coefficients of
regular transitions of the homoclinic loop, called separatrix values,
have long been conjectured by E.~Leontovi\v{c} to be important for
the cyclicity of homoclinic loops. The trace of the linear part of
the vector field at the saddle point is known as the first saddle
quantity. As shown in~\cite{AnLeon} (see also~\cite{ChoWang}),
if $\lambda\neq1$ the sign of the first saddle quantity determines
the stability of saddle loop (\emph{i.e.} stability of the limit cycle
that unfolds from it). Furthermore, in such a case, the cyclicity
of a generic saddle loop bifurcation is shown by E.~Leontovi\v{c}~\cite{AnLeon,Leon}
to be exactly 1. If $\lambda=1$, further saddle quantities and separatrix
values are needed to estimate the cyclicity, as conjectured in~\cite{Leon}
and proved later in~\cite{Roussarie1998}. The question of effectively
computing saddle quantities and separatrix values becomes of importance.
The approach was continued by many authors who developed effective
algorithms to compute those (or equivalent quantities) and to relate
them to cyclicity in unfoldings of saddle loops and of more complicated
graphics, see \emph{e.g.} C.~Rousseau, P.~Joyal, M.~Han, H.~Zhu~\cite{Jojo},~\cite{HanZ}
and~\cite{RouZhu}.
\begin{rem}
As shown in~\cite{RouLimCyc} and~\cite{Ilya-Kov}, the saddle quantities
and separatrix values fully determine the asymptotic expansion of
the Poincaré first return map near a saddle loop. However, these quantities
are intricately intertwined in the expansion, and only the initial
terms can be computed effectively. In this paper we have proven that
the unramified formal class is indeed the unramified analytic class,
but we do not give the complete description of formal invariants in
terms of the coefficients of the Dulac maps (\emph{i.e.} in terms
of saddle quantities and separatrix values), or whatsoever. Characterization
of formal invariants for mildly ramified Dulac germs remains a valuable
question for future research.
\end{rem}

To the best of our knowledge, there are few results concerning finitely
smooth classification of Poincaré first return maps of smooth real
saddle loops. In~\cite{DumRou}, it is shown that, if $\lambda\neq1$,
the Poincaré map is always $\text{C}^{1}\left(\rr_{\geq0},0\right)$-conjugate
to $x\mapsto x^{\lambda}$. The finitely smooth normal form results
from~\cite{Ilya-Kov} can probably be applied to obtain a $\text{C}^{r}\left(\rr_{\geq0},0\right)$-classification.

However, we observe that, as a consequence of our results, a $\text{C}^{\infty}\left(\rr_{\geq0},0\right)$-conjugacy
between two ramified Poincaré maps of real-analytic saddle loops implies
their analytic conjugacy. Indeed, if two such germs are conjugate
\emph{via} a smooth germ $\varphi$, then the Taylor series of $\varphi$
yields an unramified formal conjugacy between their corresponding
Dulac series: we then use Theorem~D to infer the analytic conjugacy
of the Poincaré maps.

\subsection{Plan of the article}

The main results of the article, stated in the introduction, are proved
in the sections that are listed below.
\begin{lyxlist}{00.00.0000}
\item [{Theorem~A}] is proved in Section~\ref{sec:gcc} as Theorem~\ref{thm:A}.
\item [{Theorem~B}] is proved in Section~\ref{sec:gcc} as Theorem~\ref{thm:B}.
\item [{Theorem~C}] is proved in Section~\ref{subsec:real_loops} as
Theorem~\ref{thm:realequiv-poincare}.
\item [{Theorem~D}] is proved in Section~\ref{sec:Rigidity} as Theorem~\ref{thm:rigidity1}
(attracting/repelling case) and Theorem~\ref{thm:indifferent} (indifferent
case).
\item [{Theorem~E}] is proved in Section~\ref{sec:integrability} as
Theorem~\ref{thm:regluing_integrability}.
\end{lyxlist}

\section{\protect\label{sec:Rigidity}Rigidity in the Dulac group}

\subsection{\protect\label{subsec:ramified-var}The Dulac group $\protect\Dulac$}

Let $\Delta_{c}\in\mathbb{C},\ c\ge0,$ denote the open disk of radius
$c$ centered at the origin. Following~\cite{IlyaDu}, we say that
a subset $\Omega\subset\cC$ is a {\em quadratic standard domain}
if there exist 2 constants $c\ge0$ and $d>0$ such that $\Omega=\varphi_{d}\left(\cC\setminus\Delta_{c}\right)$,
where 
\[
\varphi_{d}(z)=z+d(z+1)^{1/2}.
\]

Let $A(\Omega)$ denote the ring of holomorphic functions on $\Omega$.
The collection $(\QSD,\infty)$ of all quadratic standard domains
form a (direct) partially ordered set for the inclusion relation.
The direct limit $\RingQSD=\varinjlim A(\Omega)$ will be called the
{\em ring of QSD-germs}.

We define an asymptotic partial order relation in this ring by writing
$g\succ f$ (or $f=\oo g$) for two elements $f,g\in\RingQSD$ if,
for each $\eps>0$, there exist representatives (also denoted by $f,g$)
defined in a common domain $\Omega$ such that $|f(z)|\le\eps|g(z)|$
for all $z\in\Omega$.

We will be mostly interested in QSD-germs having a particular type
of asymptotic expansion. The {\em pol-exp asymptotic scale} is
the collection of QSD-germs defined by 
\[
f_{k,\lambda}=z^{k}\ee^{-\lambda z}
\]
with $(k,\lambda)\in\cZ_{\ge0}\times\cR_{\ge0}$. Notice that this
collection is totally ordered with respect to the asymptotic relation
defined above, namely $f_{k,\lambda}\succ f_{l,\mu}$ if and only
if $(\lambda,-k)<(\mu,-l)$ for the usual lexicographical order.

The {\em Dulac formal ring} $\rDulacf$ (with the letters ${\cal AR}$
standing for \emph{almost regular}) is the set of formal sums in the
pol-exp scale 
\[
f=\sum_{(k,\lambda)}{a_{k,\lambda}f_{k,\lambda}},\qquad a_{k,\lambda}\in\cC,
\]
whose {\em support} $\supp(f)=\{(k,\lambda):a_{k,\lambda}\ne0\}$
satisfy the following conditions:
\begin{itemize}
\item The {\em exponential coefficient set} 
\[
L=\{\lambda\,|\,\exists k:(k,\lambda)\in\supp(f)\}
\]
forms a discrete subset of $\cR_{\ge0}$.
\item For each $\lambda\in L$, the set $\{k\,|\,(k,\lambda)\in\supp(f)\}$
is finite.
\end{itemize}
In other words, the elements of $\rDulacf$ are formal sums 
\begin{equation}
\sum_{\lambda\in L}P_{\lambda}(z)\ee^{-\lambda z},\label{asympt-Dulac}
\end{equation}
with $L\subset\cR_{\ge0}$ a discrete subset and $\{P_{\lambda}\}_{\lambda\in L}$
a collection of complex polynomials.

The {\em Dulac ring} is the sub-ring $\rDulac\subset\RingQSD$
of QSD-germs having an asymptotic expansion in $\rDulacf$. Consider
the Taylor map 
\[
\Tay:\rDulac\longrightarrow\rDulacf,
\]
which is the ring morphism mapping each Dulac germ to its asymptotic
expansion. The following result of quasi-analyticity is crucial to
the theory.
\begin{prop}[{see \cite[§0.3,Theorem 1]{IlyaDu}}]
\label{prop:ilyashenko} The morphism $\Tay$ is injective.
\end{prop}

The {\em formal Dulac group} is the subset $\Dulacf\subset\rDulacf$
of elements $f\in\rDulacf$ whose initial part has the form 
\[
f=az+b+\oo 1
\]
for some {\em real} coefficient $a>0$ (the so-called {\em multiplier}
of $f$) and some $b\in\cC$. Notice that $\Dulacf$ forms a group
with respect to the composition. To see this, it suffices to remark
that the substitution $z\rightarrow az+b+z^{l}\ee^{-\mu z}$ into
$\ee^{-\lambda z}$ can be re-expanded as: 
\[
\ee^{-\lambda(az+b+z^{l}\ee^{-\mu z})}=\ee^{-b\lambda}\ee^{-a\lambda z}\,\sum_{k\ge0}{\frac{1}{k!}z^{kl}\ee^{-k\mu z}},
\]
and leads to an expansion in $\rDulacf$.

The subset $\Dulac\subset\rDulac$ of germs $f$ such that $T(f)$
belongs to $\Dulacf$ forms a group under composition (see {\cite{IlyaDu}},
section 0.3B), which we will call the {\em Dulac group}. In what
follows, we will refer to the elements of $\Dulac$ (\emph{resp.}
$\Dulacf$) simply as Dulac germs (\emph{resp.} formal Dulac series).

Notice that the Taylor map $\Tay$ restricts to an injective group
morphism $\Tay:\Dulac\rightarrow\Dulacf$.

\subsection{\protect\label{subsec:ramifmildramif}Two subgroups of $\protect\Dulac$}

Consider the $2\pi\ii$-translation map $\tau(z)=z+2\pi\ii$, seen
as an element of $\Dulac$. The {\em functional variation} of a
Dulac germ $f\in\Dulac$ is defined by the commutator: 
\[
\varD(f)=[\tau,f]=\tau^{-1}f^{-1}\tau f.
\]
Appendix \ref{sec:commut} contains some general identities involving
the commutators which will be useful in the sequel. \smallskip{}

A germ $f\in\Dulac$ is said to be \emph{unramified} (which we write
$f\in\U$) if $\varD(f)=\id$ (\emph{i.e.} if it commutes with $\tau$).
It is said to be \emph{mildly ramified} (which we write $f\in\MR$)
if $\varD^{2}(f)=\id$ (\emph{i.e.} if $\varD(f)$ is unramified).
Appendix \ref{sec:commut} contains some general identities involving
the commutators which will be useful in the sequel. For instance,
based on the definitions given in (\ref{eq:definitions-Ck}), we can
equivalently define:
\begin{itemize}
\item $\U=\Cent_{1}(\tau)$ is the subgroup of {\em unramified germs},
\item $\MR=\Cent_{2}(\tau)$ is the set of {\em mildly ramified germs}.
\end{itemize}
The nomenclature is justified by the following immediate result.
\begin{prop}
\label{prop:unramified} A Dulac germ $f\in\Dulac$ is unramified
if and only if there exists a germ $F\in\Diff(\cC,0)$ such that the
following diagram commutes 
\[
\xymatrix{\left(\QSD,\infty\right)\ar[d]^{\Pi}\ar[r]_{f} & \left(\QSD,\infty\right)\ar[d]_{\Pi}\\
\neigh\ar[r]^{F} & \neigh
}
\]
where $\Pi(z)=\ee^{-z}$. In other words, $f$ is a lift of a holomorphic
germ under the universal covering $\Pi:\cC_{\ge0}\rightarrow\mathbb{D}^{\star}$
where $\cC_{\ge0}=\left\{ z\in\cC:\mathrm{Re}(z)\ge0\right\} $ and
$\mathbb{D}^{\star}$ is the pointed unit disk.
\end{prop}

In fact, by the uniqueness of the lift \emph{modulo} post composition
by the deck transformation $\tau$, it follows that there exists a
group morphism from $\U$ to $\Diff(\cC,0)$ whose kernel is precisely
the cyclic group $\langle\tau\rangle$ generated by $\tau$. By an
abuse of notation, we will also denote this morphism 
\begin{equation}
\Pi_{*}:\U\longrightarrow\Diff(\cC,0).\label{isomUDiff}
\end{equation}
 On the formal side, we define, in exactly the same way, the sets
\[
\Uf\subset\MRf\subset\Dulacf
\]
of unramified and mildly ramified formal Dulac series. The following
characterization of $\Uf$ will be quite useful.
\begin{prop}
\label{prop:unramified2} A formal Dulac series $f$ is in $\Uf$
if and only if it has a multiplier $a=1$, its exponential coefficient
set $L$ is contained in $\cZ_{\ge0}$ and, for all $\lambda\ge1$,
the polynomial $P_{\lambda}$ is of degree zero.
\end{prop}

\begin{proof}
It suffices to write the identity $\tau f=f\tau$ and compare the
asymptotic expansions.
\end{proof}
It follows that an unramified formal series $f$ can always be written
in the form 
\[
f=\Pi^{-1}\circ\Big(\lambda z\Big(1+\sum_{k\ge1}a_{k}z^{k}\Big)\Big)\circ\Pi=z-\ln(\lambda)-\ln\Big(1+\sum_{k_{\ge}1}a_{k}\ee^{-kz}\Big)
\]
where $\lambda z\big(1+\sum a_{k}z^{k}\big)$ is an arbitrary invertible
formal series (\emph{i.e.}\ an element of $\Difff(\cC,0)$) and $\ln(\lambda)$
is chosen \emph{modulo} $2\pi\ii\cZ$.

Recall that the Taylor expansion map $\Tay$ defines an embedding
of $\Dulac$ into $\Dulacf$. Notice that the $\varD$ operator commutes
with the Taylor map $\Tay$, and, therefore 
\[
\Tay(\U)\subset\Uf,\qquad\Tay(\MR)\subset\MRf.
\]
By an abuse of notation, we will keep on writing $\U$, $\MR$, $\Dulac$
to refer to their images in $\Dulacf$ under the Taylor map.
\begin{rem}
\label{rem:added}It follows from Proposition~\ref{prop:cosetinj}
in the Appendix that the operator $\varD$ establishes one-to-one
correspondences 
\[
\varD:\U\backslash\MR\longrightarrow\U,\quad\text{and}\quad\varD:\Uf\backslash\MRf\longrightarrow\Uf
\]
where $H\backslash G=\{Hg:g\in G\}$ denotes the set of right cosets
of a subgroup $H$ in a group $G$. Using a formal iterative procedure
(see Proposition~\ref{prop:inverselvar}), it is not difficult to
prove that the rightmost map is indeed a bijection. A much deeper
fact, which is an immediate consequence of the results in~\cite{MaRa-Res}
and~\cite{PY}, is that the leftmost map is also a bijection.
\end{rem}

\subsection{Ramified classification of attracting/repelling Dulac germs}

Let us briefly discuss the ramified classification (\emph{i.e.} up
to $\Dulac$-conjugation) of attracting/repelling Dulac germs established
in \cite{PerResRolSerForm}, \cite{PerResRolSerAn} and \cite{Per}.
Consider a Dulac germ $f\in\Dulac$ of the form 
\[
f=az+b+\oo 1.
\]
We will say that $f$ is
\begin{itemize}
\item {\em super-attracting} if $a>1$;
\item {\em super-repelling} if $a<1$;
\item {\em hyperbolically attracting} if $a=1$ and $\re b>0$;
\item {\em hyperbolically repelling} if $a=1$ and $\re b<0$.
\end{itemize}
For shortness, we will simply say that $f$ is attracting/repelling
if one of the above four conditions holds, and in the complementary
case 
\[
a=1\quad\text{ and }\quad\re b=0
\]
we will say that $f$ is {\em indifferent}.
\begin{prop}[Ramified classification~\cite{Per}]
\label{prop:conjugatingramif} Suppose that $f\in\Dulac$ is attracting/repelling.
Then, $f$ is respectively $\Dulac$-conjugate to 
\[
g=az,\quad\text{ or }\quad g=z+b
\]
in the super attracting/repelling and hyperbolically attracting/repelling
cases.
\end{prop}

We remark that, in contrast to Proposition~\ref{prop:conjugatingramif},
the ramified classification of indifferent germs is much more involved.
Indeed, up to a conjugation by a scaling map $s(z)=cz$ for some $c\in\cR_{>0}$,
we can assume that 
\[
b\in2\pi\ii\cZ,
\]
and the latter case is studied in \cite{MardeResClassif} and \cite{MardeResReal}.
More precisely, the authors embed $\Dulac$ in a larger group of germs
defined on some QSD which admit {\em transserial} asymptotic expansions
in power-logarithm scale. It is shown that there exist functional
moduli similar to \emph{Birkhoff--Écalle--Voronin cocycles} for
the classification up to conjugacy.

\subsection{\protect\label{subsec:unramif-rigid}Unramified rigidity of attracting/repelling
Dulac germs}

We say that a Dulac germ $f\in\Dulac$ is {\em strongly formally
$\U$-rigid} if, whenever a formal unramified series $\psi\in\Uf$
satisfies $\psi^{-1}f\psi\in\Dulac$, then in fact 
\[
\psi\in\U.
\]
In particular, denoting by $\Orbit(x,G)$ the orbit of an element
$x$ in a group $G$ under the $G$-action by conjugacy, we have 
\[
\Orbit(f,\Uf)\cap\Dulac=\Orbit(f,\U).
\]
Our first rigidity result is the following:
\begin{thm}
\label{thm:rigidity1} Suppose that $f\in\Dulac$ is attracting/repelling.
Then $f$ is strongly formally $\U$-rigid.
\end{thm}

To prove Theorem~\ref{thm:rigidity1}, we start by considering the
following diagram of strict group inclusions (see Section~\ref{subsec:ramifmildramif}):\[\begin{tikzcd}              & \Dulacf                      &               \\ \Dulac \arrow[ru, hook] &                         & \Uf \arrow[lu, hook'] \\   & \U \arrow[ru, hook] \arrow[lu, hook'] &               \end{tikzcd}\]
\begin{lem}
\label{lem:intersect} $\Dulac\cap\Uf=\U$.
\end{lem}

\begin{proof}
It follows from Propositions~\ref{prop:unramified} and~\ref{prop:unramified2}
that, if a Dulac germ $f\in\Dulac$ is formally unramified (\emph{i.e.}\ lies
in $\Uf$), then it is unramified.
\end{proof}
Let $\Center(x,G)$ denote the centralizer of an element $x$ in a
group $G$. An immediate consequence of Proposition~\ref{prop:conjugatingramif}
is the following lemma.
\begin{lem}
\label{lem:center}Let $f\in\Dulac$ be attracting/repelling. Then
\[
\Center(f,\Dulacf)\subset\Dulac.
\]
\end{lem}

\begin{proof}
Suppose firstly that $f$ is super attracting/repelling. By Proposition~\ref{prop:conjugatingramif},
up to conjugation inside $\Dulac$, we can assume that $f(z)=az$.
Since the conjugation with a germ $az$, $a>0$, preserves the order
of the monomials in the Dulac expansion, it suffices to remark that
the monomials of the form $\mu z^{k}\ee^{-\lambda z}$ commuting with
$az$ must satisfy: 
\[
\mu a^{k}z^{k}\ee^{-\lambda az}=\mu az^{k}\ee^{-\lambda z},
\]
which gives $\mu a^{k}=\mu a$ and $\lambda az\equiv\lambda z$ \emph{modulo}
$2\pi\ii\cZ$. We conclude that either $\mu=0$ or $(k,\lambda)=(1,0)$.
Therefore, 
\[
\Center(az,\Dulacf)=\{\mu z:\mu\in\cR_{>0}\}\subset\Dulac.
\]
The case where $f$ is hyperbolically attracting/repelling is similar.
Using again Proposition~\ref{prop:conjugatingramif} we assume, up
to a $\Dulac$-conjugation, that $f(z)=z+b$, with $\re b$ non-zero.
Now, if a formal Dulac series 
\[
g=cz+d+\sum_{L\subset\cR_{>0}}P_{\lambda}(z)\ee^{-\lambda z}
\]
commutes with $f$, it necessarily follows that $c=1$, and that each
term $P_{\lambda}(z)\ee^{-\lambda z}$ satisfies 
\[
\ee^{-\lambda b}P_{\lambda}(z+b)\ee^{-\lambda z}=P_{\lambda}(z)\ee^{-\lambda z}.
\]
If $P_{\lambda}\neq0$, the leading terms on both sides are equal.
We get a contradiction since $\re b$ is non-zero. Therefore, all
$P_{\lambda}$ vanish and we get: 
\[
\Center(z,\Dulacf)=\{z+d:d\in\cC\}\subset\Dulac.
\]
\end{proof}
\begin{proof}[Proof of Theorem~\ref{thm:rigidity1}.]
Let $\psi\in\Uf$ be a formal unramified series conjugating $f$
to another Dulac germ $g\in\Dulac$, and let $\vphi\in\Dulac$ be
a Dulac germ conjugating $f$ and $g$, whose existence is proved
in Proposition~\ref{prop:conjugatingramif}.

Then, the element $\psi^{-1}\vphi$ lies in $\Center(f,\Dulacf)$.
By Lemma~\ref{lem:center}, this implies that 
\[
\psi^{-1}\vphi\in\Dulac.
\]
Therefore, $\psi$ lies in the intersection $\Dulac\cap\Uf$. We conclude
using Lemma~\ref{lem:intersect}.
\end{proof}

\subsection{Unramified rigidity of indifferent Dulac germs}

We now consider indifferent Dulac germs, namely, germs $f\in\Dulac$
of the form 
\begin{equation}
f(z)=z+b+\oo 1\label{eq:f-unitymult}
\end{equation}
with $\re b=0$.

Here, the rigidity issue is more subtle. In particular, assuming that
$f$ is {\em itself} an unramified germ, it follows from Proposition \ref{prop:unramified}
that the classification up to $\U$-conjugation is precisely the same
as the holomorphic classification of the projected diffeomorphism
$F=\Pi_{*}f\in\Diff(\cC,0)$ (see (\ref{isomUDiff})), given by 
\[
F(Z)=\ee^{b}Z(1+\oo 1).
\]
In this case, the formal and the holomorphic classification can differ
spectacularly when $\re b=0$. We refer the reader to the vast literature
on the subject (see \emph{e.g.} \cite{MaRa-SN}, \cite{VoroParaboEng},
\cite{Yoccoz1992}).

For the remaining of this section, we prove that indifferent Dulac
germs \emph{that are not itself unramified} are strongly rigid.
\begin{thm}
\label{thm:indifferent} Let $f\in\Dulac$, $f\left(z\right)=z+2\pi\ii\beta+\oo 1$,
$\beta\in\mathbb{R}$, be an indifferent Dulac germ which is mildly
ramified but not unramified. Then, $f$ is strongly formally $\U$-rigid.
\end{thm}

The main ingredient of the proof is the interplay between the unramified
conjugation and the variation operator. We first state and prove Lemma~\ref{lem:variationgroup}
and Proposition~\ref{prop:GH} that will be used in the proof.

We define the {\em $\Dulac$-variation group} of a Dulac germ $f\in\Dulac$
as the group 
\[
\DVar f=\langle\varD(f),\varD(f^{-1})\rangle\subset\Dulac.
\]
The following result shows that $\DVar f$ behaves ``nicely'' under
unramified conjugation.
\begin{lem}
\label{lem:variationgroup} Let $g$ be an unramified germ. Then,
for $f\in\Dulac$, we have
\[
\varD(g^{-1}fg)=g^{-1}\varD(f)g\quad\text{ and }\quad\varD(g^{-1}f^{-1}g)=g^{-1}\varD(f^{-1})g.\ 
\]
In particular, $g$ conjugates $\DVar{g^{-1}fg}$ to $\DVar f$. The
same relations hold if $g$ is an unramified formal series.
\end{lem}

\begin{proof}
These relations are a direct consequence of formulas (\ref{eq:varxz})
in Appendix~\ref{sec:commut}. More precisely, we apply the formulas
with $x=f$, $z=g$ and reason in the group $G=\Dulac$ (resp. $G=\Dulacf$)
whenever $g$ is an unramified Dulac germ (\emph{resp}. an unramified
formal series).
\end{proof}
In this subsection, we assume that $f$ is an indifferent germ lying
in $\MR\setminus\U$, and write 
\begin{equation}
f=z+2\pi\ii\beta+\oo 1\label{eq:ind}
\end{equation}
with $\beta\in\cR$. We consider the germs 
\[
G=\Pi_{*}\varD(f),\quad H=\Pi_{*}\varD(f^{-1})
\]
of diffeomorphisms in $\Diff(\cC,0)$ obtained by the projection morphism
$\Pi$ given in (\ref{isomUDiff}).
\begin{prop}
\label{prop:GH} Let $f\in\MR\setminus\U$ be indifferent, as in \eqref{eq:ind},
and let $G,\,H\in\Diff(\cC,0)$ be defined as above. Then, the germs
$G,H$ are $k$-tangent to the identity\footnote{A tangent to identity germ $G\in\Diff(\cC,0)$ is said to be \emph{$k$-tangent
to identity} if $G(x)=x+ax^{k+1}+\oo{x^{k+1}}$, $a\in\mathbb{C}^{\times}$,
as $x\to0$, $k\in\mathbb{Z}_{\geqslant1}$.} at the same order $k\in\cZ_{\ge1}$. Moreover, $G$ and $H$ commute
if and only if one of the following two situations occurs.
\begin{itemize}
\item \textbf{Embedded into a flow case}: $\beta\not\in\frac{1}{2k}\cZ$
and 
\[
G=\Exp{\partial}\quad\text{and}\quad H=\Exp{\left(-\frac{1}{\nu}\,\partial\right)},
\]
where $\nu=\ee^{-2\pi\ii k\beta}$ and, up to a conjugation in $\Diff(\cC,0)$,
\begin{equation}
\partial=2\pi\ii\frac{x^{k}}{1+\frac{1}{1-\nu}x^{k}}\ x\frac{\partial}{\partial x}.\label{partial-expr}
\end{equation}
\item \textbf{Identical variations case}: $\beta\in\frac{1}{2k}+\frac{1}{k}\cZ$,
$G=H$ and, up to a {\em formal conjugation} in $\Difff(\cC,0)$,
$G=H=\Exp{\partial}$, where 
\[
\partial=2\pi\ii\frac{x^{k}}{1+\frac{1}{2}x^{k}}\ x\frac{\partial}{\partial x}.
\]
\end{itemize}
\end{prop}

\begin{proof}
Suppose first that $\beta=0$. Then $f\in\Dulacf_{>0}$ (see the notation
in Section~\ref{subsec:appB} of Appendix~B). By Lemma \ref{lem:exponentiallog}
in the Appendix, $f=\Exp X$ for some (possibly formal) nilpotent
derivation $X\in\DerDulacf$. By Corollary~\ref{cor:finalformvars},
up to a simultaneous conjugation of both $G$ and $H$ inside $\Diff(\cC,0)$,
there exist constants $k\in\cZ_{\ge1}$ and $\mu\in\cC$ such that
\[
G=\Exp{\partial}\quad\text{and}\quad H=\Exp{\eta},
\]
where the formal derivations $\partial,\eta\in\DerDifff$ are given
by 
\[
\partial=\Big(2\pi\ii\frac{x^{k}}{1+\mu x^{k}}+\oo{x^{2k}}\Big)\ x\frac{\partial}{\partial x},
\]
and 
\[
\eta=\Big(-2\pi\ii\frac{x^{k}}{1+\Big(\mu-1\Big)\ x^{k}}+\oo{x^{2k}}\Big)\ x\frac{\partial}{\partial x}.
\]
Note that $\partial,\eta$ are respectively the images of the derivations
$Z,W$ given in (\ref{first-vf}) and (\ref{second-vf}) by the map
$x=\ee^{-z}$. Furthermore, it follows from that same Corollary~\ref{cor:finalformvars}
that the commutator $[H,G]$ has the form 
\[
[H,G]=\Exp{\Big(4\pi^{2}k\ x^{3k}+\oo{x^{3k}}\Big)x\frac{\partial}{\partial x}},
\]
and hence $G$ and $H$ do not commute. \smallskip{}

Suppose now that $\beta\neq0$. We write $f(z)=\transl_{\beta}^{-1}\,f_{0}(z)$,
where $f_{0}=z+\oo 1$ is an element of $\Dulacf_{>0}$, and 
\[
\transl_{\beta}(z)=z-2\pi\ii\beta.
\]
It follows from identities (\ref{eq:varxz}) that 
\[
\varD(f)=\varD(f_{0}),\quad\text{and}\quad\varD(f^{-1})=\varD(f_{0}^{-1})^{\transl_{\beta}}.
\]
Therefore, if we denote by $G,H$ and $G_{0},H_{0}$ the respective
images in $\Diff(\cC,0)$ of the $\varD(f),\varD(f^{-1})$ and $\varD(f_{0}),\varD(f_{0}^{-1})$
by the projection morphism $\Pi$, we obtain 
\[
G=G_{0},\qquad H=\scale_{B}^{-1}\,H_{0}\,\scale_{B},
\]
where $\scale_{B}(x)=Bx$ is the scaling map with ratio $B=\ee^{2\pi\ii\beta}$.
By the previous paragraph, up to a conjugacy inside $\Diff(\cC,0)$,
we can assume $G=\Exp{\partial}$ and $H=\Exp{\eta^{B}}$, where 
\[
\eta^{B}=\Big(-2\pi\ii\frac{B^{k}x^{k}}{1+\Big(\mu-1\Big)\ B^{k}x^{k}}+\oo{x^{2k}}\Big)\ x\frac{\partial}{\partial x}
\]
is the conjugate of $\eta$ under $\scale_{B}$. \smallskip{}

Let us now suppose that $G$ and $H$ commute, or, equivalently, that
\[
[\partial,\eta^{B}]=0.
\]
It follows immediately that $B^{k}\ne1$, (otherwise, $\eta^{B}$
coincides with $\eta$ and does not commute with $\partial$ according
to the previous paragraph). Moreover, it follows from the explicit
characterization of the centralizer of $\partial$ (see \emph{e.g.}
\cite{CerMou}, 1.3) that there exists some constant $t\in\cC$ such
that 
\[
\eta^{B}=t\,\partial.
\]
By comparing coefficients on both sides, we get 
\[
t=-B^{k},\quad\text{and}\quad B^{k}\left(\mu-1\right)=\mu.
\]
Since $|B|=1$ and $B^{k}\neq1$, there are two cases to consider:
\begin{enumerate}
\item $B^{k}$ is non-real.
\item $B^{k}=-1$ (\emph{i.e.} $\beta\in\frac{1}{2k}+\frac{1}{k}\cZ$).
\end{enumerate}
We now use Écalle's theory on the existence of iterative roots in
$\Diff(\cC,0)$ (see \emph{e.g.} \cite{EcalParab}).

In case 1., since $G$ has a non-real iterative root $H$, it follows
from Écalle's theory that $\partial$ is a {\em convergent} derivation
and that both $G$ and $H$ can be embedded in a common holomorphic
flow. Moreover, up to conjugation inside $\Diff(\cC,0)$, we can assume
that $G=\Exp{\partial}$ and $H=\Exp ({-B^{k}\partial})$, where 
\[
\partial=2\pi\ii\frac{x^{k}}{1+\mu x^{k}}\ x\frac{\partial}{\partial x},
\]
for $\mu=\frac{B^{k}}{B^{k}-1}$.

In case 2., $B$ is a $k^{th}$-root of $-1$, and we necessarily
have $\eta^{B}=\partial$ and 
\[
G=H.
\]
Moreover, in this case the residue of the (possibly formal) derivation
$\partial$ is $\mu=\frac{1}{2}$.

Finally, we see that in both cases (1. and 2.) $G$ and $H$ indeed
commute, which then proves also the converse direction.
\end{proof}
\begin{proof}[Proof of Theorem~\ref{thm:indifferent}]
Let $\psi\in\Uf$ be an unramified formal series conjugating $f$
to another indifferent Dulac germ $f_{0}$. If follows from Lemma~\ref{lem:variationgroup}
that 
\[
\psi\,\varD(f)\,\psi^{-1}=\varD(f_{0})\quad\text{and}\quad\psi\,\varD(f^{-1})\,\psi^{-1}=\varD(f_{0}{}^{-1}).
\]
Consider the respective images of these variations in $\Diff(\cC,0)$
under the projection morphism (\ref{isomUDiff}), 
\[
G=\Pi_{*}\varD(f),\;H=\Pi_{*}\varD(f^{-1})\text{ and }G_{0}=\Pi_{*}\varD(f_{0}),\;H_{0}=\Pi_{*}\varD(f_{0}^{-1}).
\]
It follows that $\Psi=\Pi_{*}\psi\in\Difff(\cC,0)$ defines a conjugation
between the subgroups 
\[
\langle G,H\rangle\quad\text{and}\quad\langle G_{0},H_{0}\rangle,
\]
which are subgroups of the group $\Diff_{1}(\cC,0)$ of tangent to
the identity diffeomorphisms.

We now consider different cases. If either $G,\,H$ do not commute
or $G,\,H$ commute and case $(1)$ from Proposition \ref{prop:GH}
occurs, the group $\langle G,H\rangle$ is not cyclic. Hence, it follows
from \cite[Proposition 1]{CerMou}, that $\Psi$ converges. Therefore,
$\psi$ lies in $\U$ and we conclude that $f$ is strongly formally
$\U$-rigid.

\noindent Finally, if $G,\,H$ commute and case $(2)$ from Proposition~\ref{prop:GH}
occurs, then $G=H$, that is, $\varD(f)=\varD(f^{-1})$. By Item~3.
of Lemma~\ref{lem:properties-var} (with $x=y=f$), it follows that
$f^{2}\in\U$ is unramified. Let us denote by $F^{2}=\Pi_{*}f^{2}$
the corresponding germ of a diffeomorphism in $\Diff(\cC,0)$, and
similarly $F_{0}{}^{2}=\Pi_{*}f_{0}{}^{2}$.

It follows that the formal germ $\Psi$ conjugates the group $\langle G,F^{2}\rangle$
to $\langle G_{0},F_{0}{}^{2}\rangle$. According to \cite[Proposition 1]{CerMou},
to prove that $\Psi$ converges, it suffices to show that the subgroup
$\langle G,F^{2}\rangle\cap\Diff_{1}(\mathbb{C},0)$ of the group
of tangent to identity germs is not cyclic. Suppose the contrary.
Then, there exists a tangent to identity germ $C\in\Diff_{1}(\mathbb{C},0)$
such that 
\[
G=C^{m}\quad\text{and}\quad F^{2k}=C^{n},
\]
for some $m,n\in\mathbb{Z}$. Note that $F^{2k}$ is tangent to identity
(while $F^{2}$ is not necessarily), since $F(x)=\ee^{2\pi\ii\beta}x+\oo x$
and $\beta\in\frac{1}{2k}+\frac{1}{k}\mathbb{Z}$. Therefore, in the
covering coordinate $x=\ee^{-z}$, there exists a unramified germ
$c\in\U$ such that 
\[
\varD(f)=c^{m}\quad\text{and}\quad f^{2k}=c^{n}.
\]
Using Lemma \ref{lem:x2} in the Appendix and the fact that $\varD(f)=\varD(f^{-1})$,
we get $f\,\varD(f)=\varD(f)^{-1}f$. Therefore, $f\,\varD(f)^{n}=\varD(f)^{-n}f$,
which finally gives: 
\[
f\,c^{mn}=c^{-mn}\,f.
\]
In other words, $ff^{2km}=f^{-2km}f$, \emph{i.e.}, $f^{4km}$ is
the identity. This implies that $c=\idD$, and, consequently, that
$\varD(f)$ is the identity. This is a contradiction with the assumption
in the enunciate that $f$ is not unramified.
\end{proof}
\begin{rem}
In the case 1. of the Proposition \ref{prop:GH}, up to unramified
conjugation, $f$ is the solution of the following implicit equation
\[
\ee^{kf}+\frac{k}{1-\nu}f=\frac{1}{\nu}\left(\ee^{kz}+\frac{k}{1-\nu}z\right)
\]
where $\nu=\ee^{-2\pi\ii k\beta}$ and $k\in\cZ_{\ge1}$. Consider
the {\em Fatou coordinate} 
\[
w=\Fatou(z)=-\frac{1}{2\pi\ii}\Big(\ee^{kz}+\frac{1}{1-\nu}z\Big)
\]
which conjugates the vector field $\partial$ to $\frac{\partial}{\partial w}$.
Then, the elliptic function $\wp(w)$ with period lattice 
\[
\Lambda=\{m+n\frac{1}{\nu}\mid m,n\in\cZ\}
\]
is a first integral for the variation group $\DVar f$, as the action
of such group in the $w$-coordinate is generated by the translations
$\transl_{1}(w)=w+1$ and $\transl_{1/\nu}(w)=w+\frac{1}{\nu}$. In
these coordinates, the Dulac map $f$ assumes the simple form of a
scaling 
\[
f(w)=\frac{-1}{\nu}\ w
\]
which conjugates $\transl_{1}{}^{-1}$ to $\transl_{1/\nu}$. Notice
however that $f$ is not in general an automorphism of the associated
elliptic curve $\cC/\Lambda$, except for some few exceptional cases
(see \emph{e.g.} \cite[Lemma 5.4]{Milnor})
\end{rem}

\subsection{\protect\label{subsec:novo}Topological rigidity of mildly ramified
Dulac germs}

Let $\phi$ be a lift under $\Pi$ of a germ of an orientation preserving
homeomorphism $\Phi\in\mathrm{Homeo}(\mathbb{C},0)$ fixing zero.
By $\mathcal{H}$ we denote the group of all such germs. Since $\Phi$
is well defined on a full neighborhood of $0$ and orientation preserving,
\begin{equation}
[\tau,\phi]=\mathrm{id},~~\ \phi\in\mathcal{H}.\label{eq:h}
\end{equation}

\begin{defn}
We say that a Dulac germ $f\in\mathcal{D}$ is \emph{strongly topologically
rigid} if, for every $\phi\in\mathcal{H}$ and $g\in\mathcal{D}$,
the condition $\phi^{-1}f\phi=g$ implies $\phi\in\mathcal{U}$.
\end{defn}

\begin{rem}
The weaker notion of \emph{topological rigidity} occurs when the existence
of a topological conjugacy $\phi$ implies the existence of a possibly
different $\tilde{\phi}\in\mathcal{U}$ such that $\tilde{\phi}^{-1}f\tilde{\phi}=g$.
\end{rem}

We show below in Theorem~\ref{thm:solu} that a ``generic'' element
of $\mathcal{M}\backslash\mathcal{U}$ is strongly topologically rigid,
which implies all other real $\text{C}^{r}(\mathbb{R}^{2},0)$-rigidity
properties for $r\in\mathbb{N}\cup\{\infty\}$. Then, in Proposition~\ref{prop:irac}
we prove topological rigidity in an exceptional (that is, non-generic)
case.
\begin{rem}
We exclude here the elements of $\mathcal{U}$, the unramified germs,
because their topological rigidity is known to be false in general.
\begin{itemize}
\item Unramified germs $f=z+b+\oo 1$ that are either attracting or repelling
($\re b\neq0$), are all analytically linearizable (Koenigs's theorem)
but not topologically rigid: there are only two topological classes
according to the sign of $\re b$.
\item Unramified rationally indifferent germs $f=z+2\pi\ii\beta+\oo 1\in\mathcal{U}$
with $\beta\in\mathbb{Q}$, have also differing topological~\cite{Cam}
and analytical~\cite{EcalParab,VoroParaboEng} classifications.
\item When $\beta$ is irrational not much is known about the comparison
between topological and analytical classifications.
\end{itemize}
\end{rem}

\smallskip{}

Due to~\eqref{eq:h}, similarly as in~(55), the following commutator
identities hold: 
\begin{equation}
\var{\phi}=\mathrm{Id},\ \var{f\phi}={\var f}^{\phi},\ \var{\phi f}=\var f,\ ~f\in\mathcal{D},\ \phi\in\mathcal{H}.\label{eq:u}
\end{equation}
For $f\in\mathcal{M}\backslash\mathcal{U}$ we denote by $G=G_{f}:=\Pi_{*}\var f,\ H=H_{f}:=\Pi_{*}\var{f^{-1}}$
the non-identity, unramified variations (which would be the holonomy
maps of the saddle point computed on a transverse disc in a saddle
loop of which $f$ were the Poincaré map). We define the following
group, containing $G$ and $H$: 
\[
\mathcal{G}_{f}:=\Pi_{*}\langle\var{f^{n}}\cap\mathcal{U}:\ n\in\mathbb{Z}\rangle<\mathrm{Diff}(\mathbb{C},0).
\]
As is clear from the definition of solvability, it suffices that $\left\langle G,H\right\rangle $
be non-solvable for $\mathcal{G}_{f}$ to be.
\begin{thm}
\label{thm:solu}Assume that $f\in\mathcal{M}\setminus\mathcal{U}$.
If $\mathcal{G}_{f}$ is non-solvable, then $f$ is strongly topologically
rigid.
\end{thm}

\begin{proof}
Let $g\in\mathcal{M}\setminus\mathcal{U}$ and $\phi\in\mathcal{H}$
be given such that $g=\phi^{-1}f\phi$. Then, by~\eqref{eq:h} and~\eqref{eq:u},
$\Phi=\Pi_{*}\phi\in\mathrm{Homeo}(\mathbb{C},0)$ conjugates the
whole groups $\mathcal{G}_{g}=\Phi^{-1}\mathcal{G}_{f}\Phi$. From
Shcherbakov's theorem~\cite[Theorem 1]{Shcher} it follows that $\Phi$
is either holomorphic or anti-holomorphic, hence $\phi\in\mathcal{U}$
(it is orientation preserving, so the latter cannot happen).
\end{proof}
\begin{rem}
Let us explain how the situation described in the theorem can be deemed
generic. 
\begin{enumerate}
\item Polynomial foliations of a fixed degree are generically topologically
rigid, as in the works of Y.~Ilyashenko and A.~Shcherbakov: in~\cite[Theorem 2]{Shcher}
it is shown that picking such a foliation outside a real-analytic
set implies its topological rigidity. It is probable that the property
still holds on the subset of polynomial foliations in the neighborhood
of a saddle loop. 
\item Following~\cite[Section 4]{TeyAna} and references therein, we can
put a topology on $\mathcal{M}$ for which $\mathcal{G}_{f}$ is generically
non-solvable. Let us sketch the construction by considering $\mathcal{D}_{0}$,
the vector space of all holomorphic and bounded functions on some
quadratic standard domain that admit a Dulac-type asymptotic expansion
as $\re z\to+\infty$. The topology inherited from that inductive
limit of Banach spaces is Hausdorff and locally convex, and we put
on $\mathcal{D}=\rr_{>0}\id\oplus\mathcal{D}_{0}$ the product topology,
for which the composition and the inversion in $\mathcal{D}$ are
analytic. Therefore, the variation mapping $\var{\bullet}$ is analytic,
and $\mathcal{M}=\text{var}^{-2}\left(\id\right)$ is an analytic
subset of $\mathcal{D}$. On the other hand, in~\cite{theseFrank}
it is proved that $\left\langle G,H\right\rangle $ is solvable if
and only if $\left[G,\left[G,H^{2}\right]\right]=\id$. That equation
lifts as $\Lambda\left(f\right)=\id$, where $\Lambda$ is the analytic
map 
\begin{align*}
\Lambda~ & :~f\in\mathcal{D}\longmaps\left[\var f,\left[\var f,\var{f^{-1}}^{2}\right]\right]\in\mathcal{D}.
\end{align*}
Because $\Lambda|_{\mathcal{M}}$ is not constant (see \emph{e.g.}
Proposition~\ref{prop:GH} or Proposition~\ref{prop:solvable_pq}),
the analytic set $\mathcal{M}\cap\Lambda^{-1}\left(\id\right)$ has
empty interior in $\mathcal{M}$. Therefore it makes sense to say
that for $f$ in a Zariski-dense open subset of $\mathcal{M}$, the
group $\left\langle G,H\right\rangle $ (and thus $\mathcal{G}_{f}$)
is not solvable.
\end{enumerate}
\end{rem}

\medskip{}

Let us discuss in the rest of the section some exceptional cases.
We fix a germ
\begin{align*}
f & =az+b+\oo 1\in\mathcal{M}\backslash\mathcal{U}~,~~a>0,\ b\in\mathbb{C}.
\end{align*}

On the one hand if $a\neq1$, then $b$ is not a topological invariant
for $\mathcal{H}$-conjugacy: $f$ is indeed analytically conjugate
to a Dulac germ with $b:=0$ through the translation $\tau_{-\frac{b}{a-1}}\in\mathcal{H}$.
On the other hand when $a=1$, all elements of the group $\mathcal{G}_{f}$
are tangent to the identity since $G_{f^{n}}$ has derivative $\ee^{-2\ii\pi a^{n}}$
at $0$. Following~\cite{CerMou} such a group can be solvable only
if it is abelian. With a careful case-by-case study, the results of
the cited reference together with Proposition~\ref{prop:solvable_pq}
can yield topological rigidity for $a$ rational, but the technical
proof would lengthen unnecessarily the present work. We have preferred
to provide a stronger result for $a$ irrational, maybe less expected,
that stems from the formal rigidity stated in Theorem~D.
\begin{prop}
\label{prop:irac} Assume $a\in\mathbb{R}_{>0}\setminus\mathbb{Q}$
and $\mathcal{G}_{f}$ abelian. Then, $f$ is topologically rigid.
\end{prop}

\noindent In the proof of Proposition~\ref{prop:irac} we use the
following lemma.
\begin{lem}
\label{lem:ti} Assume $a\in\mathbb{R}_{>0}\setminus\mathbb{Q}$ and
$\mathcal{G}_{f}$ abelian. Then, $a$ is invariant under $\mathcal{H}$-conjugacy.
\end{lem}

\begin{rem}
Except in the special case described above in Lemma~\ref{lem:ti},
we were unable to prove that the multiplier $a$ is a $\mathcal{H}$-invariant
of a mildly ramified germ. When $f$ is the Poincaré map of a saddle
loop, the multiplier $a=\lambda$ is the opposite of the saddle point's
eigenratio. Motivated by the affirmative results for real saddle loop
foliations in~\cite[Theorems 1,2]{DumRou}, we expect $a$ to be
a topological invariant (even complete if $a\neq1$) also for mildly
ramified germs.
\end{rem}

\begin{proof}
Assume for some $\phi\in\mathcal{H}$ that $g:=\phi^{-1}f\phi$ lie
in $\mathcal{M}\setminus\mathcal{U}$ and write $g=cz+\cst+\oo 1$
with $c>0$. We prove that $a=c$. Note that $f^{2},\ g^{2}\in\mathcal{M}$
since $\mathcal{G}_{f}$ is abelian, by Lemma~\ref{lem:properties-var}.
Therefore, $H_{f^{2}}\in\mathcal{G}_{f}$, with multiplier $\ee^{-2\pi\ii a^{2}}$,
and $H_{g^{2}}\in\mathcal{G}_{g}$ with multiplier $\ee^{-2\pi\ii c^{2}}$.
Since $\mathcal{G}_{g}=\Phi^{-1}\mathcal{G}_{f}\Phi$, for $\Phi=\Pi_{*}\phi\in\mathrm{Homeo}(\mathbb{C},0)$,
we conclude that $H_{g}=\Phi^{-1}H_{f}\Phi$ and $H_{g^{2}}=\Phi^{-1}H_{f^{2}}\Phi$.
The rotation number of an element of $\mathrm{Diff}(\mathbb{C},0)$
is a topological invariant, therefore $a-c\in\mathbb{Z}$ and $a^{2}-c^{2}\in\mathbb{Z}$,
which is not possible for irrational numbers $a,\ c>0$ unless $a=c$.
\end{proof}
\medskip{}

\begin{proof}[Proof of Proposition~\ref{prop:irac}]
Assume $g:=\phi^{-1}f\phi\in\mathcal{M}\setminus\mathcal{U}$ for
some $\phi\in\mathcal{H}$. The group $\mathcal{G}_{f}$ is formally
linearizable in $\widehat{\mathrm{Diff}}(\mathbb{C},0)$. Indeed,
the linear part of $G\in\mathcal{G}_{f}$ is an irrational rotation
$\ee^{-2\pi\ii a^{-1}}\id$, so it is formally linearizable through
some $\hat{\Psi}\in\widehat{\mathrm{Diff}}(\mathbb{C},0)$, that is:
$\hat{\Psi}^{-1}G\hat{\Psi}=\ee^{-2\pi\ii a^{-1}}\id$. Since $\mathcal{G}_{f}$
is abelian, $\hat{\Psi}^{-1}\mathcal{G}_{f}\hat{\Psi}<\mathrm{Center}\left(\ee^{-2\pi\ii a^{-1}}\id,\widehat{\mathrm{Diff}}(\mathbb{C},0)\right)=\left\{ \mu\id:\ \mu\in\mathbb{C}^{\times}\right\} $.
Therefore, 
\begin{equation}
\hat{\Psi}^{-1}H\hat{\Psi}=\ee^{-2\pi\ii a}\id.\label{druga}
\end{equation}

Analogously, the group $\mathcal{G}_{g}$ is formally linearizable,
so there exists $\hat{\Psi}_{1}\in\widehat{\mathrm{Diff}}(\mathbb{C},0)$
such that 
\begin{equation}
\hat{\Psi}_{1}^{-1}G_{g}\hat{\Psi}_{1}=\ee^{-2\pi\ii a^{-1}}\id,\ ~\hat{\Psi}_{1}^{-1}H_{g}\hat{\Psi}_{1}=\ee^{-2\pi\ii a}\id.\label{prva}
\end{equation}
Indeed, by Lemma~\ref{lem:ti}, $f$ and $g$ have the same multiplier
$a$ because they are $\mathcal{H}$-conjugate. If we were able to
prove at a formal level that $f$ and $g$ are $\widehat{\mathcal{U}}$-conjugate,
then Theorem~D would imply they are analytically conjugate, as expected.
Let us prove that property.

Since the deck transforms $\tau_{2\pi\ii k}$ for $k\in\mathbb{Z}$
commute with each element of $\hat{\mathcal{U}}$, and since conjugation
by an element of $\hat{\mathcal{U}}$ cannot change the constant term
of a Dulac germ that is tangent to identity, the identities~\eqref{druga}
and~\eqref{prva} can be lifted to the following variation equations:
\begin{align*}
 & \var{f^{-1}}=\hat{\psi}\tau_{2\pi\ii(a-1)}\hat{\psi}^{-1},\ ~\var f=\hat{\psi}\tau_{2\pi\ii(a^{-1}-1)}\hat{\psi}^{-1},\\
 & \var{g^{-1}}=\hat{\psi}_{1}\tau_{2\pi\ii(a-1)}\hat{\psi}_{1}^{-1},\ ~\var g=\hat{\psi}_{1}\tau_{2\pi\ii(a^{-1}-1)}\hat{\psi}_{1}^{-1},
\end{align*}
where $\hat{\psi},\ \hat{\psi}_{1}\in\hat{\mathcal{U}}$ are such
that $\hat{\Psi}=\Pi_{*}\hat{\psi}$ and $\hat{\Psi}_{1}=\Pi_{*}\hat{\psi}_{1}$.
Proposition~\ref{prop:equalvars} (Appendix) implies the existence
of $\hat{h}_{1},\ \hat{h}_{2}\in\mathrm{Center}(\tau_{2\pi i(a-1)},\hat{\mathcal{U}})$
such that 
\[
\hat{\psi}^{-1}f\hat{\psi}=\hat{h}_{1}\,az,\ ~\hat{\psi}_{1}^{-1}g\hat{\psi}_{1}=\hat{h}_{2}\,az.
\]
Since $a$ is irrational, it follows that $\widehat{h}_{1}$ and $\widehat{h}_{2}$
belong to $\left\{ \tau_{b}:\ b\in\mathbb{C}\right\} $. Therefore,
there exist $b_{1},\,b_{2}\in\mathbb{C}$ such that: 
\[
\hat{\psi}^{-1}f\hat{\psi}=az+b_{1},\ ~\hat{\psi}_{1}^{-1}g\hat{\psi}_{1}=az+b_{2}.
\]
As remarked upon earlier in the section if $a\neq1$, then $az+b_{2}$
is conjugate to $az+b_{1}$ in $\mathcal{U}$ by $\tau_{\frac{b_{2}-b_{1}}{a-1}}$,
so that: 
\[
\left(\hat{\psi}_{1}\tau_{\frac{b_{2}-b_{1}}{a-1}}\right)^{-1}g\,\left(\hat{\psi}_{1}\tau_{\frac{b_{2}-b_{1}}{a-1}}\right)=f.
\]
Therefore $f$ and $g$ are $\hat{\mathcal{U}}$-conjugate, and thus
$\mathcal{U}$-conjugate.
\end{proof}
\begin{rem}
There is no hope to derive strong topological rigidity by sticking
to the arguments used in the proof.
\end{rem}

\section{\protect\label{sec:Saddle-loops0}Equivalence of saddle loops}

\subsection{\protect\label{subsect:preparedsaddle}Prepared saddle foliations}

Let us introduce some preliminary notation. We denote by $\Delta_{r}$
the open disk of radius $r>0$ in $\cC$ centered at the origin and
by $\bD_{r}$ its closure. The open unit disk will be denoted simply
by $\Delta$, and we write $\bD^{\star}=\bD\setminus\{0\}$ for the
punctured unit disk. In order to simplify the notation, the closed
subset of $\cC^{2}$ given by $\bD\times\{0\}$ will also be denoted
simply by $\bD$ when there is no risk of ambiguity.

Given a positive constant $\lambda>0$, a {\em{$\lambda$-adapted
region}} is a subset in $\mathbb{C}^{2}$ of the form 
\begin{equation}
U_{A,B}=\{|y||x|^{\lambda}\leqslant A,|x|\leqslant1,|y|\leqslant B\}\label{UA-base}
\end{equation}
for some constants $A,\,B>0$. In logarithmic coordinates $x=\ee^{-z},y=\ee^{-w}$,
\eqref{UA-base} becomes: 
\begin{equation}
U_{A,B}=\left\{ \re w+\lambda\re z\geqslant-\log A,\,\re z\geqslant0,\,\re w\geqslant-\log B\right\} .\label{UA-log}
\end{equation}
The domains $U_{A,B}$ are illustrated below.

\begin{figure}[htb]
\centering{}\includegraphics[width=8.79796cm,height=3.23295cm]{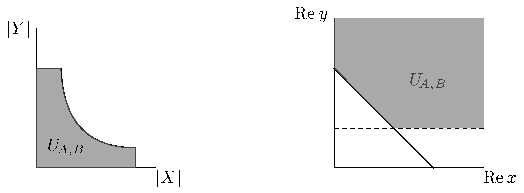}
\end{figure}

A {\em $\lambda$-neighborhood } is an open subset $U\subset\mathbb{C}^{2}$
containing a $\lambda$-adapted region (for some $A,\,B>0$). In the
sequel, it will be convenient to work with a well-chosen prepared
form for a saddle-type singularity, which we now define. We will prove
in Proposition~\ref{prop:preparation-saddles} that any saddle-type
singularity can be put in such a form up to a convenient choice of
local coordinates.
\begin{defn}
\label{def:prepared_saddle} A {\em prepared saddle foliation}
with eigenratio $-\lambda<0$ is a pair $(U,\vF)$ such that:
\begin{enumerate}
\item $U$ is a $\lambda$-neighborhood such that $U\cap\{x=0\}$ and $U\cap\{y=0\}$
are simply connected;
\item $\vF$ is a singular foliation in $U$, defined by a holomorphic differential
$1$-form 
\begin{equation}
\omega=x\dd y+\lambda y\left(1+x^{n}y\,K(x,y\right))\dd x,\label{eq:preparedform}
\end{equation}
for some positive integer $n\ge\lfloor\lambda\rfloor$ and some $K\in\mathcal{O}(U)$
satisfying the bound 
\[
\sup_{U}{|K|}\le\varepsilon,
\]
for some constant $\varepsilon>0$ such that $\eps B<1$, where it
is assumed that $U$ contains the $\lambda$-adapted region $U_{A,B}$.
\end{enumerate}
\end{defn}

From now on, we will denote by $(x,y)$ the coordinates in $\cC^{2}$.
The invariant lines $\{x=0\}$ and $\{y=0\}$ are respectively the
{\em vertical} and the {\em horizontal} separatrix of $\left(U,\vF\right)$.

It follows from the expression of $\omega$ and the choice for $\varepsilon$
that $\vF$ is both transverse to the vertical fibration $\{x=\cst\}$
and to the horizontal fibration $\{y=\cst\}$ in the domain $U_{A,B}$.
\smallskip{}

Some of the constructions that we describe in the sequel are only
valid up to restricting $(U,\vF)$ to some smaller $\lambda$-neighborhood.
Therefore, from now on, we adopt the following {\em germified}
point of view: we say that $(U,\vF)$ satisfies a certain property
(P) if the property holds up to restricting the foliation to a sufficiently
small $\lambda$-neighborhood.

For instance, we will tacitly identify two prepared foliations $(U,\vF)$
and $(\tilde{U},\tilde{\vF})$ if these foliations coincide in some
smaller $\lambda$-neighborhood $V\subset U\cap\tilde{U}$.
\begin{defn}
\label{def:transversals-prep}~
\begin{enumerate}
\item The {\em fixed base transversal} for $(U,\vF)$ is the germ of
the vertical fiber $\{x=1\}$ with base-point $(1,0)$. For brevity,
we will denote it simply by $(\Omega,1)$.
\item A {\em floating base transversal} for $(U,\vF)$ is a germ of a
curve $(\Sigma,\sigma)$ with base-point $y=\sigma\in\mathbb{C}^{\times}$
in the vertical separatrix $\{x=0\}$ such that $|\sigma|\le B$,
and which is both transverse to $\vF$ and to the vertical fibration
$\{x=\cst\}$. Here, $B>0$ is any positive real number such that
$U_{A,B}\subset U$ for some $A>0$.
\end{enumerate}
\end{defn}

\begin{rem}
Equivalently, $\Sigma$ is given locally as the graph of a function
$y=\alpha(x)$, with $\alpha(0)=\sigma$. Notice that we do not assume
$\Sigma$ to be a fiber of the horizontal fibration $\{y=\cst\}$.
Later on, we will see the importance of considering such floating
base transversals when constructing equivalences of prepared saddles.
\end{rem}

\subsubsection{\protect\label{subsect:preparation-saddles}Preparation of saddles
in foliated surfaces}

The following result shows that, up to a convenient choice of local
holomorphic coordinates, a prepared form \eqref{eq:preparedform}
can always be assumed to hold in the vicinity of a saddle type singularity.
\begin{prop}
\label{prop:preparation-saddles} Let $s$ be a saddle type singularity
of eigenratio $-\lambda$ contained in a holomorphic foliated surface
$(S,\mathcal{G})$. There exist local coordinates $(x,y)$ defined
in a neighborhood $V$ of $s$ mapping $(S,\mathcal{G})|_{V}$ to
a \emph{prepared} saddle of the form \eqref{eq:preparedform}.
\end{prop}

\begin{proof}
It well-known (see \emph{e.g.} \cite[Théorème 5.1.2]{Lolo}) that
there exists local holomorphic coordinates $(x,y)$ near $s$ such
that $\mathcal{G}$ is locally generated by 
\[
\omega=x\dd y+y(\lambda-K(x,y))\dd x,
\]
with $K$ holomorphic at $(0,0)$ and divisible by $xy$. Up to a
scaling $x\mapsto\alpha x$, we can assume that $\omega$ is defined
in a neighborhood of the closed disk $\{0\}\times\bar{\Delta}$. In
particular, there exists some $R>1$ such that: 
\[
K(x,y)=\sum_{i\in\mathbb{Z}_{\geqslant1}}c_{i}(y)x^{i}
\]
where each $c_{i}$ is a function in $\mathcal{O}(\Delta_{R})$ vanishing
at $0$. We define the {\tmem{$x$-order}} of $K$ as the smallest
index $i\geqslant1$ such that $c_{i}$ is nonzero.

We now prove that, up to a further holomorphic coordinate change,
we can assume that $K$ has the $x$-order greater than or equal to
$\lfloor\lambda\rfloor$ (the integer part of $\lambda$). To prove
this, let us consider the dual vector field 
\[
\partial=\left(x\frac{\partial}{\partial x}-\lambda y\frac{\partial}{\partial y}\right)+Ky\frac{\partial}{\partial y},
\]
and compute the action of the automorphism $\Phi=\exp\left((ax^{i})y\frac{\partial}{\partial y}\right)$,
where $a\in\mathcal{O}(\Delta_{R})$ is a function of the $y$-variable
and $i\geqslant1$. The Lie bracket of $(ax^{i})y\frac{\partial}{\partial y}$
with the initial diagonal term of $\partial$ gives 
\begin{eqnarray*}
\left[(ax^{i})y\frac{\partial}{\partial y},\left(x\frac{\partial}{\partial x}-\lambda y\frac{\partial}{\partial y}\right)\right] & = & -\left(\left(x\frac{\partial}{\partial x}-\lambda y\frac{\partial}{\partial y}\right)(ax^{i})\right)y\frac{\partial}{\partial y}\\
 & = & (L_{i}(a)x^{i})y\frac{\partial}{\partial y},
\end{eqnarray*}
where $L_{i}$ given by $L_{i}(a)=(\lambda y\frac{\partial a}{\partial y}-ia),\ a\in\mathcal{O}(\Delta_{R}),$
is a linear operator in $\mathcal{O}(\Delta_{R})$. Writing the Taylor
expansion $a(y)=\sum a_{j}y^{j}$, we get: 
\[
L_{i}(a)=\sum_{j\in\mathbb{Z}_{\geqslant0}}(\lambda j-i)a_{j}y^{j},
\]
and we conclude that the operator $L_{i}$ is invertible for a fixed
$i\geqslant1$, provided that $i\notin\lambda\mathbb{Z}_{\geqslant0}$.

Notice that:
\[
\{1,\ldots,\lfloor\lambda\rfloor-1\}\cap\lambda\mathbb{Z}_{\geqslant0}=\emptyset.
\]
Therefore, if $K$ has $x$-order $i_{0}$ with $i_{0}\leqslant\lfloor\lambda\rfloor-1,$
we can eliminate its order $i_{0}$ (in $x$) element by applying
the automorphism considered above with the conveniently chosen coefficient
$a$.

Finally, re-scaling the coordinates one can enforce the remaining
conditions of Definition \ref{def:prepared_saddle}.
\end{proof}

\subsubsection{\protect\label{sect:lift}Partial path lifting}

Let us consider a prepared saddle foliation $(U,\vF)$ and denote
by 
\[
\Fib:U\setminus\{0\}\rightarrow W\supseteq\bD^{\star}
\]
the fibration defined by $\{x=\cst\}$, with base $W\subseteq\{y=0\}$
containing the punctured unit disk $\bD^{\star}$. By the transversality
property, we can intuitively see each leaf of $\vF$ as a covering
space over the base (save for $\left\{ x=0\right\} $). However, we
must be careful with the fact that we only have a {\em partial lifting
property}: not all paths in $\bD^{\star}$ lift to paths in a leaf
of $(U,\mathcal{F})$.

To study this issue, consider a path $\Gamma\subset\bD^{\star}$ with
initial point $x_{0}$, and denote by $\Gamma_{y_{0}}\subseteq U$
the $\vF$-lift of $\Gamma$ through the point $(x_{0},y_{0})\in\tmop{Fib}^{-1}(x_{0})$
(if it exists). Let $(x,y)$ be the end-point of $\Gamma_{y_{0}}$.

Our goal is to estimate the modulus $|y|$ in terms of $|y_{0}|$
and the length of $\Gamma$. It turns out to be easier to make the
estimates in the covering variables $(z,w)$, where $x=\ee^{-z}$
an $y=\ee^{-w}$. We consider the {\em{quasi-first integral}}
\[
\varphi=w+\lambda z.
\]
In the $(z,\varphi)$ coordinates, (\ref{eq:preparedform}) becomes:
\[
\frac{\dd{\varphi}}{\dd z}=-\lambda\ee^{-((n-\lambda)z+\varphi)}k,
\]
where $k(z,w)=K(\ee^{-z},\ee^{-w})$.

Let us choose a point $(z_{0},w_{0})\in\cc^{2}$ with $\re{z_{0}}\geqslant0$
and $\re{w_{0}}\geqslant0$ which projects to $(x_{0},y_{0})$ under
the covering map, and let $\gamma\subset\left\{ z~:~\re z>0\right\} $
be the path with initial point $z_{0}$ which covers $\Gamma$ under
$x=\ee^{-z}$.
\begin{prop}
\label{prop:estimatesphi} The solution of the Cauchy problem 
\begin{equation}
\left\{ \begin{array}{l}
\frac{\dd{\varphi}}{\dd z}=-\lambda\ee^{-((n-\lambda)z+\varphi)}k\\
\varphi_{0}=w_{0}+\lambda z_{0}
\end{array}\right.\label{Cauchy-probl}
\end{equation}
along the curve $\gamma$ satisfies the estimate 
\begin{equation}
\ee^{\re{\varphi_{0}}}-\lambda\varepsilon\tmop{length}(\gamma)\leqslant\ee^{\re{\varphi}}\leqslant\ee^{\re{\varphi_{0}}}+\lambda\varepsilon\tmop{length}(\gamma).\label{eq:s}
\end{equation}
\end{prop}

\begin{proof}
We write 
\[
\ee^{\varphi}\dd{\varphi}=-\lambda\ee^{-sz}k\dd z,
\]
where $s=n-\lambda\geqslant0$. Using the estimate $|k|\leqslant\varepsilon$,
we obtain 
\[
|\ee^{\re{\varphi}}-\ee^{\re{\varphi_{0}}}|\leqslant|\ee^{\varphi}-\ee^{\varphi_{0}}|\leqslant\lambda\varepsilon\int_{\gamma}\ee^{-s\re z}|\dd z|\leqslant\lambda\varepsilon\tmop{length}(\gamma),
\]
which immediately implies \eqref{eq:s}.
\end{proof}
\begin{cor}
Consider the solution $t\rightarrow\varphi(t)$ of the above Cauchy
problem along a parameterized curve $z=\gamma(t)$, with an initial
condition $(z_{0},\varphi_{0}=w_{0}+\lambda z_{0})$ such that $(z_{0},w_{0})\in U_{A,B}$,
where $U_{A,B}$ is given in \eqref{UA-log}. Then, the pair 
\[
(z(t),w(t))=(\gamma(t),\varphi(t)-\lambda\gamma(t))
\]
stays inside the region $U_{A,B}$ provided that the following inequalities
hold for all values of $t$: 
\begin{eqnarray}
\re{\gamma(t)} & \geqslant & 0,\nonumber \\
\ee^{\re{\varphi_{0}}}-\lambda\varepsilon\tmop{length}(\gamma(t)) & \geqslant & \max\left\{ \frac{1}{A},\frac{1}{B}\ee^{\lambda\re{\gamma(t)}}\right\} .\label{ineq2}
\end{eqnarray}
\end{cor}

\begin{proof}
It suffices to plug the lower bound for $\ee^{\re{\varphi(t)}}$ provided
by the Proposition~\ref{prop:estimatesphi} in the three inequalities
in (\ref{UA-log}) that define $U_{A,B}$.
\end{proof}
\medskip{}

\begin{rem}
\label{rem:ovo} Following \cite[Section 5.2]{Lolo}, it is sometimes
useful to have an estimate for the growth of the quasi-integral $x^{\lambda}y$
in terms of the length of the path $\Gamma$ {\em in the original
$(x,y)$-coordinates}. Using the same Cauchy problem (\ref{Cauchy-probl})
and the fact that $\re z\geqslant0,\re w\geqslant-\log B$, and taking
$n\geqslant1$, we get: 
\[
|d\varphi|\leqslant\lambda\varepsilon|\ee^{-z}\dd z|.
\]
Denoting $F=\ee^{-\varphi}=x^{\lambda}y$, we obtain 
\[
\left|\frac{dF}{F}\right|\leqslant\lambda\varepsilon B|\dd x|,
\]
and, by integration, 
\begin{equation}
\ee^{-\lambda\varepsilon B\,\tmop{length}(\Gamma)}\leqslant\frac{|y||x^{\lambda}|}{|y_{0}||x_{0}^{\lambda}|}\leqslant\ee^{\lambda\varepsilon B\,\tmop{length}(\Gamma)}.\label{eq:ineqoh}
\end{equation}
We note that, for a fixed $x_{0}$ and for a fixed path $\Gamma$
with base-point $x_{0}$ (and endpoint $x$), we can always choose
$y_{0}$ sufficiently small in modulus such that 
\begin{equation}
\ee^{\lambda\varepsilon B\cdot\mathrm{length}(\Gamma)}\cdot\left(\frac{|x_{0}|}{|x|}\right)^{\lambda}\cdot|y_{0}|\leqslant B\quad\text{ and }\quad\ee^{\lambda\varepsilon B\cdot\tmop{length}(\Gamma)}\cdot|x_{0}|^{\lambda}\cdot|y_{0}|\le A.\label{eq:heart}
\end{equation}
In this case, it follows from the inequality (\ref{eq:ineqoh}) that
the path $\Gamma$ can be $\vF$-lifted to $(x_{0},y_{0})$, staying
inside the domain $U_{A,B}$ given in \eqref{UA-base}.
\end{rem}

\subsubsection{Examples of path lifting}

Consider a \emph{radial} path $\Gamma(t)=\ee^{-t}x_{0}$ with $|x_{0}|=1$,
$t\in\cR$, and its $\vF$-lift $\Gamma_{y_{0}}(t)$ with initial
point $(x_{0},y_{0})$. The following example estimates the interval
$[0,T]$ such that $\Gamma_{y_{0}}([0,T])$ stays inside a given domain
$U_{A,B}$.
\begin{example}
\label{ex:1} Let us assume $A=B=1$ in $U_{A,B}$ for simplicity.
In the covering coordinates $x=\ee^{-z}$, $y=\ee^{-w}$, the path
$\Gamma$ corresponds to the horizontal path $\gamma(t)=z_{0}+t$,
with $\re{z_{0}}=0$ and $t\in[0,T]$. Then $\tmop{length}(\gamma(t))=t$
and the inequalities \eqref{ineq2} are satisfied along this path
as long as $T$ satisfies: 
\begin{eqnarray*}
\ee^{\re{w_{0}}} & \geqslant & \max\left\{ 1,\ee^{\lambda T}\right\} +\lambda\varepsilon T.
\end{eqnarray*}
\smallskip{}
\end{example}

Consider a \emph{purely circular} path $\Gamma(t)=\ee^{-\ii t}$,
$t\in\cR$, in the $(x,y)$ coordinates. Let us estimate the number
of {\em laps} this path can make before the modulus of $|y|$ becomes
larger than $1$.
\begin{example}
\label{ex:2} Let us again assume $A=B=1$ in $U_{A,B}$ for simplicity.
In the covering coordinates, the path $\Gamma$ corresponds to the
vertical path $\gamma(t)=\ii t$ with endpoint $z=\gamma(T)$ and
initial point $z_{0}=0$. We obtain 
\[
\tmop{length}(\gamma)=|T|.
\]
Therefore, the inequalities~\eqref{ineq2} read 
\[
\begin{array}{lll}
\ee^{\re{w_{0}}}-\lambda\varepsilon|T| & \geqslant & 1\end{array}
\]
or, equivalently, 
\[
|T|\leqslant\frac{\ee^{\re{w_{0}}}-1}{\lambda\varepsilon}
\]
\end{example}

The following class of \emph{exponential paths} is used in the Ilyashenko's
proof of the analytic extension of the Dulac map to quadratic standard
domains.
\begin{example}
\label{ex:iliashenko-paths} Consider the family of exponential paths
\begin{equation}
\xi_{\alpha,C}(t)=t+\ii C(\ee^{\alpha t}-1),\quad t\in\mathbb{R}_{\geqslant0},\label{eq:expo}
\end{equation}
with parameters $C\in\{\pm1\}$ and $\alpha\in[0,\lambda[$.

\begin{figure}[htb]
\centering{}\includegraphics[width=4.10203cm,height=5.18593cm]{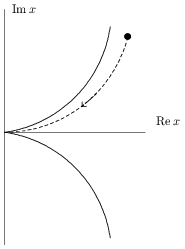}
\end{figure}

Let us fix some parameters $\alpha,C$, some $T\geqslant0$, and consider
the path $\gamma$ going {\em{backwards}} from $z_{0}=\xi_{\alpha,C}(T)=T+\ii C(\ee^{\alpha T}-1)$
to $\xi_{\alpha,C}(0)=0$ along the same trajectory as $\xi_{\alpha,C}$.
Then, we have the estimate 
\[
\tmop{length}(\gamma)\leqslant\frac{2}{\alpha}\ee^{\alpha T}.
\]
As a consequence, the inequalities (\ref{ineq2}) are satisfied (and
$\gamma$ stays inside $U_{A,B}$) as long as $T$ satisfies: 
\begin{eqnarray}
\ee^{\re{w_{0}}+\lambda T}-\lambda\varepsilon\frac{2}{\alpha}\ee^{\alpha T} & \geqslant & \max\left\{ \frac{1}{A},\frac{1}{B}\ee^{\lambda T}\right\} \nonumber \\
\ee^{\re{w_{0}}} & \geqslant & \max\left\{ \frac{1}{A}\ee^{-\lambda T},\frac{1}{B}\right\} +\lambda\varepsilon\frac{2}{\alpha}\ee^{-(\lambda-\alpha)T}\label{eq:kori}
\end{eqnarray}
which, if $\ee^{\re{w_{0}}}\geqslant\max\left\{ \frac{1}{A},\frac{1}{B}\right\} +\lambda\varepsilon\frac{2}{\alpha}$,
holds for all $T\geqslant0$.
\end{example}

\subsubsection{Monodromy and holonomy groupoids of a foliation}

Let us recall the definitions of monodromy and holonomy groupoids
for a foliation. We follow closely the exposition at~\cite[Section 5.2]{Moerdijk2003}
and refer to this book for further details. Let $\left(M,\vF\right)$
be a foliated manifold. The {\tmem{monodromy groupoid}} $\tmop{Mon}\left(M,\vF\right)$
is a groupoid over $M$ with the following arrows:
\begin{enumerate}
\item if $x,y\in M$ lie on the same leaf $L$ of $\vF$, then the arrows
in $\tmop{Mon}\left(M,\vF\right)$ with source $x$ and target $y$
are the homotopy classes (with fixed endpoints) of paths in $L$ from
$x$ to $y$; the set of all such arrows will be denoted by $\tmop{Mon}\left(M,\vF\right)_{x,y}$;
\item if $x,y\in M$ lie on different leaves of $\vF$ then there are no
arrows between them.
\end{enumerate}
The groupoid operation is induced by concatenation of paths. Notice
that the isotropy groups of the monodromy groupoid are the fundamental
groups of the leaves (\emph{i.e.} $\tmop{Mon}\left(M,\vF\right)_{x,x}=\pi_{1}(L,x)$).

Similarly, the {\tmem{holonomy groupoid}} $\tmop{Hol}\left(M,\vF\right)$
is the groupoid over $M$ where the arrows between $x,y\in M$ are
the classes of arrows in $\tmop{Mon}\left(M,\vF\right)$ \emph{modulo}
the following equivalence relation: two arrows $\gamma,\eta\in\tmop{Mon}\left(M,\vF\right)_{x,y}$
are equivalent if
\[
\tmop{hol}_{\vF,\gamma}=\tmop{hol}_{\vF,\eta},
\]
where the holonomy maps are seen as germs $(\Sigma,x)\rightarrow(\Omega,y)$
computed on arbitrarily fixed local transversals $(\Sigma,x)$ and
$(\Omega,y)$. We will often omit the symbol $\vF$ in the notation
when the underlying foliation is clear from the context.
\begin{rem}
\label{rem:identifytransv}Consider two transversals $\Sigma,\Sigma'$
with the same base-point $x$. Then, we can unambiguously identify
each holonomy map $(\Sigma,x)\rightarrow(\Omega,y)$ to a holonomy
map $(\Sigma',x)\rightarrow(\Omega,y)$ by choosing an arbitrary foliated
chart centered on $x$ and identifying these two transversals \emph{via}
the associated holonomy map $h:(\Sigma,x)\rightarrow(\Sigma',x)$.
It is obvious that such identification does not depend on the choice
of the foliated chart. We will tacitly use such an identification
from now on.
\end{rem}

In order to keep the usual terminology, the arrows in $\tmop{Mon}\left(M,\vF\right)$
will be called {\tmem{paths}}, and the arrows in $\tmop{Hol}\left(M,\vF\right)$
will be called {\tmem{holonomy germs}}.

\subsubsection{\protect\label{subsec:shol}Saddle holonomies}

Back to our present setting, let $(U,\vF)$ be a prepared saddle foliation
and $(\Sigma,\sigma)$, $(\Omega,1)$ be two transversals as in Definition~\ref{def:transversals-prep}.
\begin{defn}
\label{def:saddle_holo}Let $\eta$ be the circle $\left\{ x=0,\left|y\right|=|\sigma|\right\} $
and $\gamma$ be the circle $\left\{ \left|x\right|=1,y=0\right\} $,
oriented positively. The germs of a map
\[
\holo[\Sigma]=\holt{\eta}:(\Sigma,\sigma)\longto(\Sigma,\sigma),\quad\holo[\Omega]=\holt{\gamma}:(\Omega,1)\longto(\Omega,1),
\]
will be called the {\em saddle holonomies} of $(U,\vF)$.
\end{defn}

\begin{example}
\label{ex:linear_foliation}The linear foliation $\fol{1:\lambda}$
is defined by the differential $1$-form 
\begin{align*}
\omega_{\lambda} & =\frac{\dd y}{y}+\lambda\frac{\dd x}{x}=\dd{\log\left(yx^{\lambda}\right)}.
\end{align*}
and its leaves coincide with the level sets of the Darbouxian first-integral
\begin{align*}
H_{\lambda}\left(x,y\right) & =y\,x^{\lambda}.
\end{align*}
Consider the transversals $\Omega=\{x=1\}$ and $\Sigma=\{y=1\}$,
parameterized by the restriction of the ambient coordinates. Then
\[
\holo[\Omega]\left(y\right)=\ee^{-2\ii\pi\lambda}y\quad\text{ and }\quad\holo[\Sigma]\left(x\right)=\ee^{-\nf{2\ii\pi}{\lambda}}x.
\]
\end{example}

The following result is an easy consequence of the definition of a
prepared saddle foliation.
\begin{lem}
\label{lem:estimate_holonomy}The linear foliation provides the linear
part of the holonomy of a general prepared saddle foliation $(U,\vF)$
with same eigenratio. More precisely, for an arbitrary choice of local
coordinates $z$ and $w$ on the transversals $\Sigma$ and $\Omega$,
\begin{align*}
\holo[\Sigma](z) & =\ee^{-\nf{2\ii\pi}{\lambda}}z+\oo z,\\
\holo[\Omega](w) & =\ee^{-2\ii\pi\lambda}w+\oo w.
\end{align*}
\end{lem}

\begin{proof}
It suffices to combine (\ref{eq:preparedform}) with an easy perturbation
argument. We refer to \cite[Section 5]{MaMou} for the details.
\end{proof}
\smallskip{}

We now recall the following fundamental result relating the conjugacy
of holonomies and the equivalence of the underlying foliations.
\begin{thm}[{Equivalence of prepared saddles~\cite[Théorème 2]{MaMou}}]
\label{thm-Mattei-Moussu} Let $(U,\vF)$ and $(\widetilde{U},\widetilde{\vF})$
be prepared saddle foliations with the same eigenratio $-\lambda$
and whose holonomies $\holo[\Omega]$ and $\holo[\tilde{\Omega}]$
are conjugate by a holomorphic germ 
\[
\varphi:(\widetilde{\Omega},1)\rightarrow(\Omega,1).
\]
Up to restricting $\widetilde{U}$ and $U$ to smaller $\lambda$-neighborhoods,
there exists a biholomorphic map 
\[
\Phi:\widetilde{U}\rightarrow U
\]
such that:
\begin{enumerate}
\item $\Phi(\widetilde{\vF})=\vF$
\item $\Phi$ preserves the vertical fibers $\{x=\cst\}$.
\item The restriction of $\Phi$ to the transversal $\widetilde{\Omega}$
coincides with $\varphi$.
\end{enumerate}
Moreover, $\Phi$ is uniquely determined by conditions 1., 2. and
3.
\end{thm}

\begin{proof}
The proof uses the path-lifting method. The only remaining technical
point is to prove that $\Phi$, as constructed in Mattei-Moussu's
proof, is a map between two $\lambda$-neighborhoods. This is an immediate
consequence of both the upper and lower estimates in (\ref{eq:ineqoh})
along the paths approaching the origin radially. We refer to \cite[Section 5.2]{Lolo}
for the details.
\end{proof}
\begin{rem}
\label{rem:Realmatteimoussu} For future use, we observe the following
invariance property in Mattei-Moussu's construction. Assume that:
\begin{enumerate}
\item $(U,\vF)$ and $(\widetilde{U},\widetilde{\vF})$ are complexifications
of {\em real analytic foliations};
\item the conjugacy map $\varphi$ is real analytic.
\end{enumerate}
Then the biholomorphism $\Phi$ will also be the complexification
of a real analytic map. In particular, it defines an analytic equivalence
between the underlying real foliations.
\end{rem}

\subsubsection{\protect\label{sub:three}Saddle corner transition maps and determinations}

Let $(U,\vF)$ be a prepared saddle foliation as in Definition \ref{def:prepared_saddle}.
Informally, a corner map establishes a leaf-wise correspondence 
\[
\vF\cap\Sigma\to\vF\cap\Omega
\]
where $(\Omega,1)$ and $(\Sigma,\sigma)$ are respectively the fixed
base transversal and a floating base transversal associated to $(U,\vF)$.

This map is generally multivalued, since a leaf can intersect these
transversals at multiple points. We can choose a {\em determination}
by fixing a path in the punctured horizontal separatrix (say, connecting
a point sufficiently close to the origin to $1$) and considering
its $\vF$-lift, through the fibration $\Fib=\{x=\cst\}$, to a path
in a leaf going from $\Sigma$ to $\Omega$. As we have already observed
in Section \ref{sect:lift}, we must be careful with the fact that
not every path can be $\vF$-lifted. \smallskip{}

In what follows, we use the family of exponential paths 
\[
\{\xi_{\alpha,C}\}_{\alpha\in[0,\lambda),\,C\in\{\pm1\}},
\]
introduced in Example \ref{ex:iliashenko-paths}. Consider the restriction
of the images of this family of paths to the strip $I=\{z\in\cC:\im z\in]-\pi,\pi[\}$.
By $I_{\alpha,C}$ we denote the time domains for these restrictions.
Namely, for each parameter value $\alpha,C$, we consider the interval
$t\in I_{\alpha,C}$ such that $\xi_{\alpha,C}(I_{\alpha,C})\subset I$.
It can easily be computed from \eqref{eq:expo} that: 
\[
I_{\alpha,C}=\left[0,\frac{1}{\alpha}\log(1+\pi)\right).
\]
Now, let 
\begin{equation}
\Xi_{\alpha,C}=\ee^{-\xi_{\alpha,C}\left|_{I_{\alpha,C}}\right.},\ \alpha\in[0,\lambda),\ C\in\{\pm1\},\label{eq:expi}
\end{equation}
denote the image of the restricted paths $\xi_{\alpha,C}\left|_{I_{\alpha,C}}\right.$
under the covering map $x=\ee^{-z}$ (see Figure~\ref{fig:exp} below).

\begin{figure}[htb]
\label{fig:exp}
\centering{}\includegraphics[width=8.67469cm,height=3.52038cm]{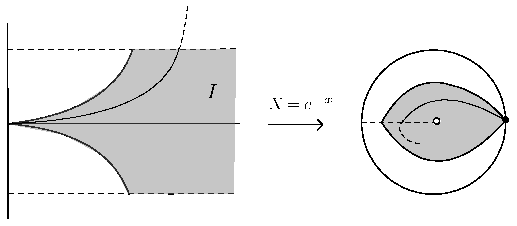}
\end{figure}

By definition, these paths $\Xi_{\alpha,C}$ are disjoint, and there
exists some radius $0<r<1$ (depending on $\lambda$) such that their
union covers the {\em cut disk} 
\[
\bD_{r}^{\cut}=\bD_{r}\setminus\cR_{\le0}
\]
In what follows, we will denote by $\gamma^{-1}$ the reverse path
of a given path $\gamma$ (\emph{i.e.} $\gamma$ traveled {\em backwards}).
\smallskip{}

The next lemma gives conditions under which the paths $\Xi_{\alpha,C}^{-1}$
can be $\vF$-lifted through the points of $\Sigma$.
\begin{lem}
\label{lem:cano_Dulac} Let $(\mathcal{F},U)$ be a prepared saddle
and let $U_{A,B}\subseteq U$, $A,\,B>0,$ be a $\lambda$-adapted
subdomain. Suppose that the base-point $\sigma$ of the floating transversal
$\Sigma$ satisfies the estimate 
\begin{equation}
|\sigma|\leqslant\frac{1}{\max\left\{ \frac{1}{A(1+\pi)},\frac{1}{B}\right\} +2\varepsilon}.\label{eq:estimate-sigma}
\end{equation}
Then, there exists some radius $r>0$ such that, for each point $x\in\bD_{r}^{\cut}$,
the path $\Xi_{\alpha,C}^{-1}$ (containing $x$) has a $\vF$-lift
by $\mathrm{Fib}$ in $U_{A,B}\subseteq U$ through the point $p=\Fib^{-1}(x)\cap\Sigma$.
\end{lem}

\begin{proof}
It is an immediate consequence of the estimates derived in Example
\ref{ex:iliashenko-paths}. Note that $\xi_{\alpha,C}^{-1}(t)$ remains
inside the domain $I$ for $t\in\left[0,\frac{1}{\alpha}\log(1+\pi)\right)$.
By \eqref{eq:kori} in Example~\ref{ex:iliashenko-paths}, putting
$T=\frac{1}{\alpha}\log(1+\pi)$, we see that $\big(\xi_{\alpha,C}\left|_{I_{\alpha,C}}\big)^{-1}\right.$
has a $\vF$-lift in $U_{A,B}$ through the point $(x_{0}=\xi_{\alpha,C}(T),y_{0})$
if: 
\[
\frac{1}{|y_{0}|}\geqslant\max\left\{ \frac{1}{A}\ee^{-\frac{\lambda}{\alpha}\log(1+\pi)},\frac{1}{B}\right\} +\lambda\varepsilon\frac{2}{\alpha}\ee^{-\frac{\lambda-\alpha}{\alpha}\log(1+\pi)}.
\]
Now, for all $\alpha\in[0,\lambda[$: 
\[
\frac{1}{A}\ee^{-\frac{\lambda}{\alpha}\log(1+\pi)}\leqslant\frac{1}{A(1+\pi)},\ \lambda\varepsilon\frac{2}{\alpha}\ee^{-\frac{\lambda-\alpha}{\alpha}\log(1+\pi)}\leqslant2\varepsilon.
\]
Therefore, if 
\begin{equation}
|y_{0}|\leqslant\frac{1}{\max\left\{ \frac{1}{A(1+\pi)},\frac{1}{B}\right\} +2\varepsilon},\label{eq:hi0}
\end{equation}
all paths $\big(\xi_{\alpha,C}\Big|_{I_{\alpha,C}}\big)^{-1}$ can
be $\vF$-lifted through $(x_{0},y_{0})$, for every $\alpha\in[0,\lambda[$
and $C\in\{\pm1\}$. Note that \eqref{eq:hi0} implies also that $|y_{0}|<B$.
\end{proof}
\medskip{}

From now on, we will suppose that the base-point $\sigma$ of the
floating transversal $(\Sigma,\sigma)$ of a prepared saddle $(U,\mathcal{F})$
always satisfies the estimate (\ref{eq:estimate-sigma}) for some
$A,\,B>0$ such that $U_{A,B}\subseteq U$. Moreover, for $r>0$ given
by Lemma~\ref{eq:estimate-sigma}, we define the germ of a {\em
cut transversal} $(\Sigma^{\cut},\sigma)$ by 
\begin{equation}
\Sigma^{\cut}\subset\Fib^{-1}\left(\bD_{r}^{\cut}\right)\cap\Sigma.\label{estimate-sigma1}
\end{equation}
and, based on the previous Lemma, we state the following:
\begin{defn}
\label{def:cano} The {\em canonical corner-transition map} associated
to $(U,\vF)$ is the germ of a map 
\[
D_{[0]}:(\Sigma^{\cut},\sigma)\to(\Omega,1)
\]
which associates to each point $p\in\Sigma^{\cut}$ the endpoint $D_{[0]}(p)\in\Omega$
of the $\vF$-lift of $\Xi_{\alpha,C}^{-1}$ through $p$.
\end{defn}

Lemma~\ref{lem:cano_Dulac} can be further generalized to general
paths with endpoint $1$ and of bounded length $\ell$, as we explain
now.
\begin{rem}
\label{rem:impo}~
\begin{itemize}
\item Using \eqref{eq:heart} we get: for a given base-point $\sigma$ of
the floating transversal $\Sigma$, where $|\sigma|<B$, and $\ell>0$,
there exists a radius $r=r_{\ell,|\sigma|,A,B}>0$ such that, for
every $x\in\bD_{r}$, every path of length less than $\ell$ containing
$x$ and landing at $1$ can be $\vF$-lifted in $U_{A,B}$ through
the point $p=\mathrm{Fib}^{-1}(x)\cap\Sigma.$
\item Choosing a path $\gamma$ from $x\in\bD_{r}^{\cut}$ to $1$ of length
less than $\ell$, the holonomy map $\mathrm{hol}_{\mathcal{F},\gamma}$
along $\gamma$ through $p=\mathrm{Fib}^{-1}(x)\cap\Sigma$ to $\Omega$
can be holomorphically extended to the whole cut transversal $\Sigma^{\cut}=\mathrm{Fib}^{-1}\left(\bD_{r}^{\cut}\right)\cap\Sigma$.
This construction produces one particular determination of the corner
map of the saddle $D:(\Sigma^{\cut},\sigma)\to(\Omega,1)$.
\end{itemize}
\end{rem}

\begin{rem}
\label{rem:otherdeterminations} By Chapter~7 of \cite{Lolo}, all
other determinations of the corner map for a prepared saddle $(U,\vF)$
are obtained from the canonical determination $D_{[0]}$ by the formula
\begin{equation}
D_{[n]}=\holo[\Omega]^{-n}\ D_{[0]},\quad\ n\in\mathbb{Z},\label{eq:ili}
\end{equation}
where $\holo[\Omega]$ is the holonomy of $\Omega$. In other words,
they are obtained by $\vF$-lifting (to a sufficiently small germ
of a floating transversal, see Remark~\ref{rem:impo}), the concatenation
of paths 
\[
(\Xi_{\alpha,C})^{-1}\star\gamma^{-n}
\]
where $\Xi_{\alpha,C}$ are the paths used in the definition of $D_{[0]}$
and $\gamma^{n}$ is the $n^{th}$ composition of the path of Definition~\ref{def:saddle_holo}.

Notice that we obtain all determinations equally by taking the compositions
\begin{equation}
D_{[n]}=D_{[0]}\ \holo[\Sigma]^{n},\quad n\in\mathbb{Z},\label{eq:ili2}
\end{equation}
where $\holo[\Sigma]$ is the holonomy of $\Sigma$.
\end{rem}

\begin{rem}
\label{rem:realsaddlecorner} Assume that $(U,\vF)$ is a {\em real
prepared saddle foliation}, \emph{i.e.} that the differential form
in Definition \ref{def:prepared_saddle} is real analytic. Suppose
further that the floating transversal $\Sigma$ is the {\em complexification}
of a real analytic transversal $\Sigma_{\cR}\subset\cR^{2}$. Then,
by construction, the canonical determination for the corner map preserves
the real axis, namely $D_{[0]}\left(\rr_{>0}\right)\subset\text{\ensuremath{\rr}}_{>0}$.
\end{rem}

Let us now describe a global holomorphic extension of $D_{[0]}$ by
fixing natural coordinates on the transversals. Namely, we consider
the {\em quasi-first integral} $F=x^{\lambda}y$ and parameterize
the transversals $\Omega$ and $\Sigma$ in such a way that 
\[
F|_{\Omega}\left(y\right)=y,\quad\text{and}\quad F|_{\Sigma}\left(x\right)=x^{\lambda}
\]
where $x^{\lambda}$ is chosen with respect to the main branch of
the logarithm.

\smallskip{}

In the next lemma, we study the lift of $y=D_{[0]}(x)$ through the
covering maps $x=\ee^{-z}$, $y=\ee^{-w}$. Note that the lift is
well-defined only up to the action of the deck transformation $\tau=\id+2\pi i$.
\begin{lem}
A lift of $D_{[0]}$ in the $(z,w)$-coordinates defines a Dulac germ
$d\in\Dulac$ having an asymptotic expansion 
\[
w=d(z)=\lambda z+2\pi\ii k+\oo 1
\]
for some integer $k\in\cZ$.
\end{lem}

\begin{proof}
The holomorphic extension of $d$ is simply obtained by considering
the $\vF$-lift to $\Sigma$ of the totality of exponential paths
$\ee^{-\xi_{\alpha,C}}$ (\emph{i.e.} no longer restricted to the
strip $I$). The estimates in Example \ref{ex:iliashenko-paths} show
that all these paths can be lifted (in their totality), and therefore
that $d_{0}$ is holomorphic on a domain of the form 
\[
\mathcal{E}=\{z\in\cC:\re z>R,|\im z|<\ee^{\beta\re z}\}
\]
for some $R>0$ and some $0<\beta<\lambda$. Such domain clearly contains
a $\QSD$. Moreover, the Cauchy problem written in Proposition \ref{prop:estimatesphi}
has the equivalent integral form 
\[
\ee^{\varphi}=\ee^{\varphi_{0}}+\int_{\gamma}F(z,\varphi)\dd z,
\]
where $F(z,\varphi)=-\lambda\ee^{-((n-\lambda)z+\varphi)}$. According
to the computations of Example \ref{ex:iliashenko-paths}, the integral
on the right-hand side is uniformly bounded on $\mathcal{E}$ by a
constant multiplied by $\ee^{\beta\re z}$. With the choice of coordinates
on the transversals described above, and passing to the covering coordinates
$z$ and $w$, we have $\varphi_{0}=\lambda z$ and $\varphi=d(z)$.
Therefore, we can write: 
\[
\ee^{d(z)}=\ee^{\lambda z}\left(1+\ee^{(\beta-\lambda)z}f\right),
\]
where $f$ is some uniformly bounded function on $\mathcal{E}$. The
expansion in the enunciate is obtained by taking logarithms of both
sides.

Finally, by applying the truncated version of the Dulac normal form
theorem to the $1$-form given in Definition \ref{def:prepared_saddle},
we can prove (see, for instance, the {\em Geometric Lemma} in \cite[section 0.3]{IlyaDu})
that $d$ has an asymptotic expansion in the pol-exp scale, as required
in the definition of Dulac germs in Section \ref{subsec:ramified-var}.
\end{proof}
Based on this Lemma, we may fix an unambiguous choice for the lift
of $D_{[0]}$ to the $\left(z,w\right)$ coordinates. The {\em lifted
(canonical) corner transition map} for $(U,\vF)$ is the unique lift
$d_{0}\in\Dulac$ of $D_{[0]}$ having the asymptotic expansion 
\[
d_{0}(z)=\lambda z+\oo 1.
\]
We remark that this lift is simply obtained by taking an appropriate
branch of the logarithm in the map $w=-\log y$ (\emph{i.e.} an appropriate
choice of the iterate of the deck transform in the lift). Note that,
by using the same deck transforms, all other determinations $D_{[k]}$,
$k\in\mathbb{Z}$, lift to germs of the type 
\[
d_{k}(z)=\lambda z+2k\pi\ii+\oo 1.
\]

\begin{rem}
Notice that, by its own definition, the inverse $d_{0}^{-1}$ of $d_{0}$
in $\Dulac$ corresponds to a lift of (some determination of) the
inverse corner transition $\vF\cap\Omega\rightarrow\vF\cap\Sigma$.
However, we stress the fact that it does not necessarily correspond
to the canonical Dulac map that we could define by exchanging the
roles of the horizontal and vertical separatrices in the construction
described in Section \ref{sub:three}.
\end{rem}

\subsubsection{\protect\label{subsect:holvars}Holonomies and variations of the
corner transition maps}

As remarked by Ilyashenko, the monodromy of the canonical corner transition
map and the holonomy of a prepared saddle $(U,\vF)$ are related by
the equation 
\begin{equation}
D_{[0]}(\ee^{2\pi\ii}x)=\holo[\Omega]^{-1}D_{[0]}(x)\label{eq:monodromy-cornermap}
\end{equation}
(see \emph{e.g.} \cite[Section 7.1.4]{Lolo}). Let us rewrite this
equation in the covering coordinate $x=\ee^{-z},~y=\ee^{-w}$. Based
on the expressions given in Lemma~\ref{lem:estimate_holonomy}, we
choose a lift $h_{\Omega}$ of $\holo[\Omega]$ of the form 
\[
h_{\Omega}(w)=w+2\ii\pi(1-\lambda)+\oo 1.
\]
Notice that $h_{\Omega}$ is an element of the Dulac group, convergent
on some right-half plane $\{z\in\cC:\re z>c\}$ and {\em unramified},
\emph{i.e.} 
\[
\var{h_{\Omega}}=\id,
\]
where $\var f=[\tau,f]$ is the functional variation defined in Section
\ref{subsec:ramifmildramif}.

We thus rewrite (\ref{eq:monodromy-cornermap}) as the following equation
in the Dulac group $\Dulac$, 
\[
d_{0}\,\tau=\tau h_{\Omega}^{-1}d_{0}.
\]
In other words, using that $\var{d_{0}^{-1}}=[\tau,d_{0}^{-1}]=\tau^{-1}d_{0}\tau d_{0}^{-1}$,
one obtains 
\[
\var{d_{0}^{-1}}=h_{\Omega}^{-1}.
\]
Similarly, based on Remark \ref{rem:otherdeterminations}, we can
write the relation 
\[
\var{d_{0}}=h_{\Sigma},
\]
where $h_{\Sigma}\in\Dulac$ is the lift of the other holonomy $\holo[\Sigma]$
with the expansion 
\[
h_{\Sigma}(z)=z+2\ii\pi\left(1-\frac{1}{\lambda}\right)+\oo 1.
\]

\begin{rem}
\label{rem:higher}~
\begin{enumerate}
\item It follows from Remark \ref{rem:otherdeterminations} that the lift
of other determinations of the corner transition map are Dulac germs
$d_{n}\in\Dulac$ related to the canonical determination $d_{0}$
by 
\[
d_{n}=h_{\Omega}^{-n}\,d_{0}=d_{0}\,h_{\Sigma}^{n},\ n\in\mathbb{Z}.
\]
\item The variation of a lifted corner transition map does not depend on
its determination. That is, $\var{d_{n}}=\var{d_{0}}$, $n\in\mathbb{Z}$.
This follows directly from the definition of variation operator, the
above item 1. and the fact that holonomies are unramified.
\end{enumerate}
\end{rem}

\subsection{\protect\label{subsec:gct}Germ of a corner transition}

We consider a prepared saddle foliation $\left(U,\vF\right)$, where
the domain $U$ contains some fixed $\lambda$-adapted region $U_{A,B}$
(see Section~\ref{subsect:preparedsaddle}). In what follows, we
will work frequently with germs of a transversal. We will say that
a germ of a floating transversal $(\Sigma,\sigma)$ with base-point
$\sigma$ on the vertical separatrix has a {\tmem{realization}}
in $U$ if it extends to a holomorphic curve in $U_{A,B}$ (which
we denote by the same letter $\Sigma$) such that:
\begin{enumerate}
\item $\Sigma$ is both transversal to the foliation $\vF$ and the vertical
fibration $\tmop{Fib}=\{x=\tmop{cst}\}$;
\item $\Sigma$ is simply connected.
\end{enumerate}
Recall that, for each point $\sigma$ in $U_{A,B}$ lying in the vertical
separatrix, there exists a standard determination of the corner transition
map 
\[
D_{[0]}:(\Sigma^{\tmop{cut}},\sigma)\rightarrow(\Omega,1),
\]
which is defined by lifting the exponential paths to an arbitrary
transversal through $\sigma$ (see Definition~\ref{def:cano}). All
other determinations of the corner transition are given by the identity
\begin{equation}
D_{[n]}=D_{[0]}\,\hhol_{\sigma}^{n},\label{distinctdets}
\end{equation}
where $n\in\mathbb{Z}$ and $\hhol_{\sigma}$ is the holonomy map
associated to the circular path $t\rightarrow(0,\ee^{2\pi\ii t}\sigma)$,
$t\in[0,1]$.

From now on, applying the identification of transversals with the
same base-point (as explained in Remark~\ref{rem:identifytransv}),
we will use the notation $D_{\sigma,[n]}$ to indicate the $n^{\tmop{th}}$
determination of the germ of a corner transition associated to the
{\em initial point} $\sigma\in U_{A,B}$. More generally, we will
write simply $D_{\sigma,\star}$ to refer to one of these determinations,
when it is not necessary to specify the index $n$.

The {\tmem{corner transition class}} $\Corner\left(U,\vF\right)$
is the set $\{D_{\sigma,\star}\}$ of all such corner transition germs,
indexed by the base-point and the corresponding determination. The
groupoid of holonomy germs $\tmop{Hol}(U,\mathcal{F})$ acts on $\Corner\left(U,\vF\right)$
by right composition. Namely, for each germ $D_{2}\in\Corner\left(U,\vF\right)$
with initial point $\sigma_{2}$ and each holonomy germ $h\in\tmop{Hol}(U,\mathcal{F})$
with initial point $\sigma_{1}$ and endpoint $\sigma_{2}$, we define
the action of $h$ on $D_{\sigma_{2}}$ by 
\[
D_{\sigma_{1}}=D_{\sigma_{2}}h
\]
which gives a germ $D_{\sigma_{1}}$ on $\Corner\left(U,\vF\right)$
with initial point $\sigma_{1}$. Notice that~(\ref{distinctdets})
is a particular case of this action, where $\sigma_{1}$ and $\sigma_{2}$
coincide.

\subsubsection{Corner transitions determined by connecting paths}

For what follows, it will be convenient to relate the corner transition
germs and the germs obtained by lifting of other $C^{1}$ paths in
the base, possibly different from the exponential paths used in the
definition of corner transitions in Section~\ref{sub:three}.

A {\tmem{connecting path}} is a piecewise $C^{1}$-path $\Gamma$
in the punctured disk $\bar{\Delta}^{\star}$ with initial point $x_{0}\in\bar{\Delta}^{\star}$
and endpoint $1$.

We will say that a corner transition germ $D=D_{\sigma,\star}\in\Corner\left(U,\vF\right)$
based at $\sigma$ is {\tmem{determined}} by $\Gamma$ if there
exists a realization of a transversal ($\Sigma$, $\sigma$) in $U$
and an open subset $V\subset\Sigma$ such that:
\begin{enumerate}
\item the germ $D$ is holomorphic at $V$;
\item the path $\Gamma$ has a lift $\Gamma_{z}$ through a point $z\in V$
and 
\[
\tmop{hol}_{\Gamma_{z}}=D
\]
(seen as germs at $z$ of a holomorphic map from the transversal $\Sigma$
to the transversal $\Omega$).
\end{enumerate}
In this case we will write $D=D_{\sigma,\Gamma}$ to indicate the
connecting path which determines $D$.
\begin{center}
\raisebox{-0.210834587661154\height}{\includegraphics[width=7.45853cm,height=3.5079cm]{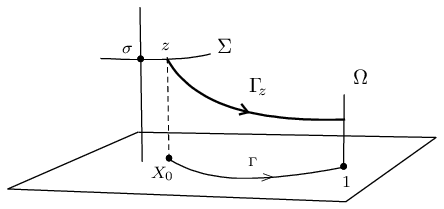}}
\par\end{center}
\begin{rem}
Observe that, by construction, each corner transition $D_{\sigma,[n]}$
has the form 
\[
D_{\sigma,[n]}=D_{\sigma,\Gamma_{n}}
\]
where the path $\Gamma_{n}=\Xi^{-1}\star\gamma^{-n}$ is obtained
as the concatenation of some finite segment $\Xi^{-1}$ of an exponential
path with the $(-n)^{\tmop{th}}$-iteration of the circular path $\gamma:t\rightarrow(\ee^{2\pi\ii t},0)$,
which makes one {\tmem{lap}} around the boundary $\mathbb{S}^{1}=\partial\bar{\Delta}$
of the unit disk. In particular, this shows that all corner transition
germs are determined by at least one connecting path.
\end{rem}

Suppose now that two connecting paths $\Gamma^{0},\Gamma^{1}$ (with
same initial point $x_{0}$) are path-homotopic. Then, by the partial
lifting property discussed in the previous section, for each piecewise-$C^{1}$
path-homotopy 
\[
H:[0,1]\times[0,1]\rightarrow\bar{\Delta}^{\star}
\]
such that $H(0,\cdot)=\Gamma^{0}(\cdot)$, $H(1,\cdot)=\Gamma^{1}(\cdot)$,
$H(\cdot,0)=x_{0}$ and $H(\cdot,1)=1$, there exists a disk $\mathbb{D}_{H}\subset\tmop{Fib}^{-1}(x_{0})$
(lying in the fiber above the initial point and whose radius depends
on both $x_{0}$ and the maximal length of paths within the homotopy
$H$) such that for each $z\in\mathbb{D}_{H}$, the lifted paths $\Gamma_{z}^{0},\Gamma_{z}^{1}\in\tmop{Mon}\left(M,\vF\right)$
are homotopic\footnote{By taking the initial point in a disk $\mathbb{D}_{H}$ of sufficiently
small radius, we ensure that both paths admit a complete lift starting
from this initial point.}.

In particular, we obtain the following result.
\begin{lem}
Let $\Gamma^{0}$ and $\Gamma^{1}$ be two homotopic connecting paths.
Then, there exists a radius $r>0$ such that, for each base-point
$\sigma$ with $|\sigma|<r$, the following statements are equivalent:
\begin{enumerate}
\item $\Gamma^{0}$ determines a corner transition germ based at $\sigma$;
\item $\Gamma^{1}$ determines a corner transition germ based at $\sigma$.
\end{enumerate}
Moreover, in this case, $D_{\sigma,\Gamma^{0}}=D_{\sigma,\Gamma^{1}}$.
\end{lem}

\begin{rem}
The converse of this lemma does not necessarily hold.
\end{rem}

\subsubsection{\protect\label{subsec:cornerrelated}Corner transitions related by
the horizontal holonomy}

Consider two corner transition germs $D_{1}=D_{\sigma_{1},\Gamma^{1}}$
and $D_{2}=D_{\sigma_{2},\Gamma^{2}}$ which are determined by two
connecting paths $\Gamma^{1},\Gamma^{2}$. Note that $\Gamma^{1}$
and $\Gamma^{2}$ are not necessarily path-homotopic, neither have
they necessarily the same initial point. Let $\Sigma_{1},\Sigma_{2}$
be two respective transversals through $\sigma_{1},\sigma_{2}$ and
let 
\[
z_{1}\in\Sigma_{1}\quad\tmop{and}\quad z_{2}\in\Sigma_{2}
\]
be the initial points of the corresponding lifted paths $\Gamma_{z_{1}}^{1}$
and $\Gamma_{z_{2}}^{2}$.

We will say that $D_{1}$ and $D_{2}$ are {\tmem{related by the
horizontal holonomy}} if there exists a path $\Upsilon\subset\bar{\Delta}^{\star}$
which admits a lift $\Upsilon_{z_{1},z_{2}}\in\tmop{Mon}\left(U,\vF\right)$
in the foliation with initial point $z_{1}$ and endpoint $z_{2}$,
and satisfying 
\[
\tmop{hol}_{\Gamma_{z_{1}}^{1}}=\tmop{hol}_{\Gamma_{z_{2}}^{2}}\tmop{hol}_{\Upsilon_{z_{1},z_{2}}},
\]
where the holonomies are computed with respect to the transversals
$\Sigma_{1}$ and $\Sigma_{2}$ (see figure below).
\begin{center}
\raisebox{-0.375474987880968\height}{\includegraphics[width=5.97424cm,height=4.19448cm]{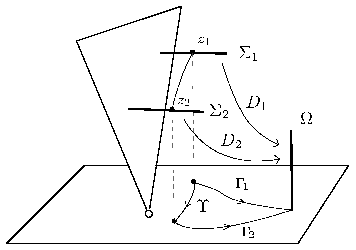}}
\par\end{center}

In other words, for the corresponding germs of a corner transition
$D_{1}$ and $D_{2}$, centered respectively at points $z_{1}$ and
$z_{2}$, we have:
\[
D_{1}=D_{2}\tmop{hol}_{\Upsilon_{z_{1},z_{2}}}.
\]
We will say that $\Upsilon$ is a {\tmem{transporting}} path between
$D_{1}$ and $D_{2}$.

\subsection{Germs of a regular transition and abstract saddle loops}

Recall from the introduction that a saddle loop in a foliated surface
$\left(S,\vG\right)$ is given by a germ of a saddle singularity and
a curve $\Gamma$ which is tangent to the foliation and connects the
two local separatrices. The regular transition map is defined by choosing
two points $\sigma,\omega$ on $\Gamma$ near the singularity and
by considering the holonomy along the sub-path $\gamma\subset\Gamma$
connecting these points.

Clearly, such a curve $\gamma$ has no intrinsic meaning since we
can freely move its initial point and endpoint. However, the genuine
ambiguity in its choice is encoded by the local holonomy groupoid
near the saddle. This motivates the definition of abstract saddle
loop which we give below.
\begin{defn}
Let $\left(U,\vF\right)$ be a prepared saddle foliation. The {\tmem{regular
transition class}} is the set $\tmop{Reg}\left(U,\vF\right)$ whose
elements are germs of a diffeomorphism
\[
R:(\Omega,1)\rightarrow(\Sigma,\sigma)
\]
between the fixed transversal and a transversal with base-point $\sigma\in U$
on the vertical separatrix. We say that $\sigma$ is the {\tmem{endpoint}}
of $R$.
\end{defn}

Following the convention from Remark~\ref{rem:identifytransv}, we
identify two such germs 
\[
R:(\Omega,1)\rightarrow(\Sigma,\sigma),~~\tilde{R}:(\Omega,1)\rightarrow(\tilde{\Sigma},\sigma)
\]
with the {\tmem{same}} endpoint $\sigma$, if there exists a foliated
chart centered on $\sigma$ such that $hR=\tilde{R}$, where $h:(\Sigma,\sigma)\rightarrow(\tilde{\Sigma},\sigma)$
is the germ of a diffeomorphism given by the chart.

The groupoid of holonomy germs $\tmop{Hol}(U,\mathcal{F})$ acts on
$\tmop{Reg}\left(U,\vF\right)$ by left composition. Namely, for each
germ $R_{\sigma_{1}}\in\tmop{Reg}\left(U,\vF\right)$ with endpoint
$\sigma_{1}$ and each germ $h\in\tmop{Hol}(U,\mathcal{F})$ with
initial point $\sigma_{1}$ and endpoint $\sigma_{2}$, we define
the action of $h$ on $R_{\sigma_{1}}$ by 
\begin{equation}
(h,R_{\sigma_{1}})\longmaps R_{\sigma_{2}}=hR_{\sigma_{1}}\label{actionR}
\end{equation}
which gives a germ $R_{\sigma_{2}}\in\tmop{Reg}\left(U,\vF\right)$
with endpoint $\sigma_{2}$. Evidently, this action accounts for the
change of floating transversal.
\begin{defn}
\label{def:regulartransclass} An {\tmem{abstract }}{\tmem{loop}}
on $U$ is a triple $\left(U_{A,B},\vF,\vR\right)$, where $\left(U,\vF\right)$
is a prepared saddle and $\vR\subset\tmop{Reg}(U)$ is an orbit for
the action of $\tmop{Hol}(U,\mathcal{F})$ defined above. In other
words, an orbit $\mathcal{R}$ is a subset of regular germs in $\tmop{Reg}\left(U,\vF\right)$
linked by the relation (\ref{actionR}).
\end{defn}

Notice that, for a fixed base-point $\sigma$, there exist at most
countably many distinct germs $R\in\vR$ with endpoint $\sigma$.
Indeed, given one such germ $R_{[0]}\in\mathcal{R}$, all others are
of the form
\begin{equation}
R_{[n]}=\hhol_{\sigma}^{n}R_{[0]}\,\label{distinctdetsR}
\end{equation}
for some $n\in\mathbb{Z}$, where $\hhol_{\sigma}$ is the holonomy
map associated to the circular path $t\mapsto(0,\ee^{2\pi\ii t}\sigma)$,
$t\in[0,1]$.

\smallskip{}

Following the convention introduced in Section~\ref{subsec:gct}
for germs of a corner transition, we will denote simply by $R_{\sigma,\star}$
a regular transition germ in $\vR$ with endpoint $\sigma$ (whenever
it is irrelevant to specify which one we consider).

\subsubsection{Regular transitions related by the horizontal holonomy}

Let $\left(U,\vF,\vR\right)$ be an abstract saddle loop and let $R_{1}=R_{\sigma_{1},\star}$
and $R_{2}=R_{\sigma_{2},\star}$ be two germs in $\vR$ with respective
endpoints $\sigma_{1},\sigma_{2}$.

As in Section \ref{subsec:cornerrelated}, we will say that $R_{1}$
and $R_{2}$ are {\tmem{related by the horizontal holonomy}} if
there exist a path $\Upsilon\subset\bar{\Delta}^{\star}$, two realizations
of $R_{1},R_{2}$ on respective transversals $\Sigma_{1},\Sigma_{2}$
and two points $z_{1},z_{2}$ in these transversals such that $\Upsilon$
lifts to a path $\Upsilon_{z_{1},z_{2}}\in\tmop{Mon}\left(U_{A,B},\vF\right)$
from $z_{1}$ to $z_{2}$ and 
\[
R_{2}=\tmop{hol}_{\Upsilon_{z_{1},z_{2}}}R_{1},
\]
understood as germs from $z_{1}$ to $z_{2}$ (see figure below).
We will say that $\Upsilon$ is a {\tmem{transporting}} path between
$R_{1}$ and $R_{2}$.
\begin{center}
\raisebox{-0.342678849852111\height}{\includegraphics[width=6.64966cm,height=4.59591cm]{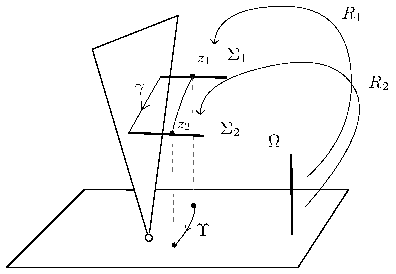}}
\par\end{center}
\begin{rem}
Observe that any two germs $R_{1},R_{2}\in\vR$ are always related
by the horizontal holonomy. Indeed, by the definition of $\vR$, there
exists a path $\gamma\in\tmop{Mon}\left(U,\vF\right)$ between the
base-points $\sigma_{1}$ and $\sigma_{2}$ \ (which is \emph{a fortiori}
contained in the vertical leaf $\{x=0\}$) such that 
\[
R_{2}=\tmop{hol}_{\gamma}R_{1}
\]
(as germs at $\sigma_{1}$ and $\sigma_{2}$ respectively). In order
to obtain a path $\Upsilon$, it suffices to choose a realization
of $\hol_{\gamma},~R_{1}$ and $R_{2}$ on small open neighborhoods
of the transversals and choose points $z_{1}\neq\sigma_{1},z_{2}\neq\sigma_{2}$
such that $\hol_{\gamma}(z_{1})=z_{2}$. By the uniqueness of the
lift, it follows that the germ of $\hol_{\gamma}$ at $z_{1}$ is
simply the holonomy associated to some path $\Upsilon_{z_{1},z_{2}}\in\tmop{Mon}\left(U,\vF\right)$,
which can be projected by the fibration Fib onto a path $\Upsilon\subset\bar{\Delta}^{\star}$
satisfying the above conditions.
\end{rem}

\subsubsection{Poincaré first return maps of an abstract saddle loop}

Let $\left(U,\vF,\vR\right)$ be an abstract saddle loop. Given a
base-point $\sigma$ on the vertical separatrix, let $R=R_{\sigma,\star}\in\tmop{Reg}\left(U,\vF\right)$
be a representative of $\vR$ with endpoint $\sigma$ and let $D=D_{\sigma,\star}\in\Corner\left(U,\vF\right)$
be a corner transition germ with initial point $\sigma$. The composition
\[
P=RD
\]
will be called a {\tmem{Poincaré first return germ based at $\sigma$}}.
\begin{defn}
The set of all such germs will be called the {\tmem{Poincaré first
return class}} of the abstract loop, and denoted by $\Poinc\left(U,\vF,\vR\right)$.
The factorization $P=RD$ as above will be called a {\tmem{dynamical
decomposition}} of the germ $P$.
\end{defn}

For future use, let us state the following obvious result.
\begin{lem}
\label{lem:sameloop}Let $(R_{1},D_{1})$ and $(R_{2},D_{2})$ be
elements in $\vR\times\Corner\left(U,\vF\right)$ which define respective
Poincaré first return germs 
\[
P_{1}=R_{1}D_{1}\quad\text{and}\quad P_{2}=R_{2}D_{2},
\]
based respectively at $\sigma_{1},\sigma_{2}$. Suppose that both
$R_{1},R_{2}$ and $D_{1},D_{2}$ are related by the {\tmstrong{
same}} transporting path $\Upsilon$. Then $P_{1}$ and $P_{2}$
are analytically conjugate.
\end{lem}

\begin{proof}
By definition, there exists a realization of both $R_{1},~D_{1}$
on a common transversal $\Sigma_{1}$ and a realization of $R_{2},~D_{2}$
on a common transversal $\Sigma_{2}$ and two points $(z_{1},z_{2})\in\Sigma_{1}\times\Sigma_{2}$
such that 
\begin{equation}
R_{2}=\tmop{hol}_{\Upsilon_{z_{1},z_{2}}}R_{1}\quad\text{and}\quad D_{1}=D_{2}\tmop{hol}_{\Upsilon_{z_{1},z_{2}}}.\label{eq:us}
\end{equation}
By assumption, $R_{1}$ and $R_{2}$ are analytic diffeomorphisms
defined on connected domains, mapping $1\in\Omega$ respectively to
$\sigma_{1}$ and $\sigma_{2}$. Therefore, putting $\varphi=R_{2}R_{1}^{-1}$,
the leftmost identity in~\eqref{eq:us} implies that $\varphi$ defines
a germ of a diffeomorphism from $(\Sigma_{1},\sigma_{1})$ to $(\Sigma_{2},\sigma_{2})$
such that $\varphi=\tmop{hol}_{\Upsilon_{z_{1},z_{2}}}$ (seen as
a germ from $z_{1}$ to $z_{2}$). Hence, by~\eqref{eq:us},
\[
P_{2}=R_{2}D_{2}=\varphi R_{1}D_{1}\varphi^{-1}=\varphi P_{1}\varphi^{-1}.
\]
\end{proof}

\subsection{\protect\label{sec:gcc}Loop germs and their equivalence}

We now {\tmem{germify}} the notion of abstract saddle loops by
considering the basis of neighborhoods of the disk $\bar{\Delta}$
defined by the family of $\lambda$-neighborhoods.
\begin{defn}
\label{def:loop-germ}A {\tmem{loop germ}} $(\mathcal{F},\mathcal{R})$
is an equivalence class of abstract saddle loops, where two abstract
loops 
\[
\left(U_{1},{\vF}_{1},\vR_{1}\right)\quad\left(U_{2},{\vF}_{2},\vR_{2}\right)
\]
are {\tmem{equivalent}} if there exists a $\lambda$-adapted domain
$U$ contained in $U_{1}\cap U_{2}$ such that 
\[
\mathcal{F}_{1}=\mathcal{F}_{2}\quad\tmop{on}\quad U
\]
and the regular transition classes $\mathcal{R}_{1}$ and $\mathcal{R}_{2}$
coincide when restricted to $U$.
\end{defn}

Accordingly, the {\tmem{corner transition class}} of a loop germ
is the set $\tmop{Corner}\left(\vF\right)$ of Dulac germs which can
be obtained as corner transition germs of some of its realizations
$\left(U,\vF\right)$. We define analogously the {\tmem{Poincaré
first return class}} as the set $\Poinc\left(\vF,\vR\right)$ of
Dulac germs which can be obtained as Poincaré first return maps of
some of its realizations $\left(U,\vF,\vR\right)$.

\medskip{}

Consider now a germ of a fibered diffeomorphism $\Phi\in\tmop{Diff}_{\tmop{fib}}(\mathbb{C}^{2},\bar{\Delta})$.
The {\tmem{action}} of $\Phi$ on a loop germ $\left(\vF,\vR\right)$
defines a new loop germ 
\[
(\widetilde{\vF},\widetilde{\vR})
\]
obtained as follows.
\begin{enumerate}
\item Choose representatives of both $\Phi$ and $\vF$ on some $\lambda$-adapted
neighborhood $U$ and define 
\[
\widetilde{\vF}=\Phi\left(\vF\right).
\]
\item Choose an element $R\in\vR$, \emph{i.e.} a germ of a regular transition
\[
R:(\Omega,1)\longrightarrow(\Sigma,\sigma),
\]
and define $\widetilde{\vR}$ to be the regular transition class generated
by 
\begin{equation}
\tilde{R}=(\Phi|_{\Sigma}\nobracket)R(\Phi|_{\Omega}\nobracket)^{-1}.\label{mapregclass}
\end{equation}
\end{enumerate}
\begin{defn}
\label{def:fib_eqv_abstract_loops}Under conditions 1. and 2. above,
we write $(\widetilde{\vF},\widetilde{\vR})=\Phi\cdot\left(\vF,\vR\right)$,
and we say that $(\widetilde{\vF},\widetilde{\vR})$ and $\left(\vF,\vR\right)$
are {\tmem{$\tmop{Diff}_{\tmop{fib}}(\mathbb{C}^{2},\bar{\Delta})$-equivalent}}.
\end{defn}

\begin{rem}
We remark the following basic correspondence between connecting paths
of corner transition maps in $\tmop{Diff}_{\tmop{fib}}(\mathbb{C}^{2},\bar{\Delta})$-equivalence.
In the above notation, suppose that $\Phi$ has a realization on a
domain $U$ and let $D\in\tmop{Corner}\left(U,\vF\right)$ be a corner
transition map defined on a transversal $\Sigma$ with base-point
$\sigma$. Then 
\begin{equation}
\widetilde{D}\overset{\tmop{def}}{=}(\Phi|_{\Omega}\nobracket)D(\Phi|_{\Sigma}\nobracket)^{-1}\label{mapcornclass}
\end{equation}
is a corner transition germ for $\widetilde{\vF}$, defined on a transversal
with base-point $\tilde{\sigma}=\Phi(\sigma)$. Moreover, if $D=D_{\sigma,\Gamma}$
is determined by ${\vF}$-lifting a connecting path $\Gamma$ then
$\widetilde{D}=D_{\tilde{\sigma},\Gamma}$ is determined by the $\tilde{\vF}$-lifting
of the {\em same} connecting path (since $\Phi$ preserves the
vertical fibration).
\end{rem}

We will say that the germs $\tilde{R}\in\widetilde{\vR}$ and $\tilde{D}\in\tmop{Corner}(\widetilde{\vF})$
defined by (\ref{mapregclass}) and (\ref{mapcornclass}) are respectively
the {\tmem{$\Phi$-correspondents}} of $R\in\vR$ and $D\in\tmop{Corner}\left(\vF\right)$.
We denote this relation by $\tilde{R}=\Phi_{\star}(R)$ and $\tilde{D}=\Phi_{\star}(D)$.

\subsubsection{Equivalence implies conjugacy of Poincaré transition germs}

Our first equivalence theorem states that equivalent loop germs possess
conjugate Poincaré first return classes.
\begin{thm}
\label{thm:A}Consider two $\tmop{Diff}_{\tmop{fib}}(\mathbb{C}^{2},\bar{\Delta})$-equivalent
loop germs $(\widetilde{\vF},\widetilde{\vR})$ and $\left(\vF,\vR\right)$
and let $\Phi$ be a fibered diffeomorphism such that $\left(\vF,\vR\right)=\Phi\cdot(\widetilde{\vF},\widetilde{\vR})$.
Let 
\[
P\in\Poinc\left(\vF,\vR\right),\qquad\tilde{P}\in\Poinc(\widetilde{\vF},\widetilde{\vR})
\]
be two Poincaré first return germs with dynamical decompositions $P=RD$
and $\tilde{P}=\tilde{R}\tilde{D}$, such that the pairs $R_{1},R_{2}\in\mathcal{R}$
and $D_{1},D_{2}\in\tmop{Corner}\left(\vF\right)$ defined by 
\[
R_{1}=R,R_{2}=\Phi_{\star}(\tilde{R})\quad\text{and}\quad D_{1}=D,D_{2}=\Phi_{\star}(\tilde{D})
\]
are related by a{\tmstrong{ same}} transporting path $\Upsilon\in\bar{\Delta}^{\star}$.
Then, $P$ and $\tilde{P}$ are analytically conjugate.
\end{thm}

\begin{proof}
It follows from Lemma~\ref{lem:sameloop} that that $P=P_{1}=R_{1}D_{1}$
and $P_{2}=R_{2}D_{2}$ are analytically conjugate. Moreover, we have
\[
P_{2}=R_{2}D_{2}=(\Phi|_{\widetilde{\Sigma}}\nobracket)\tilde{R}(\Phi|_{\Omega}\nobracket)^{-1}(\Phi|_{\Omega}\nobracket)\tilde{D}(\Phi|_{\widetilde{\Sigma}}\nobracket)^{-1}=(\Phi|_{\widetilde{\Sigma}}\nobracket)\tilde{P}(\Phi|_{\widetilde{\Sigma}}\nobracket)^{-1}
\]
and, therefore, $P_{2}$ is analytically conjugate to $\tilde{P}$.
\end{proof}

\subsubsection{Conjugacy of Poincaré transition germs implies equivalence}

Our next goal is to prove the converse of Theorem~\ref{thm:A}.
\begin{thm}
\label{thm:B}Let $(\widetilde{\vF},\widetilde{\vR})$ and $\left(\vF,\vR\right)$
be two loop germs, and suppose that there exists two Poincaré first
return germs 
\[
P\in\Poinc\left(\vF,\vR\right),\qquad\tilde{P}\in\Poinc(\widetilde{\vF},\widetilde{\vR})
\]
which are analytically conjugate. Then, $(\widetilde{\vF},\widetilde{\vR})$
and $\left(\vF,\vR\right)$ are $\tmop{Diff}_{\tmop{fib}}(\mathbb{C}^{2},\bar{\Delta})$-equivalent.
\end{thm}

\begin{proof}
Let
\[
P\in\Poinc\left(\vF,\vR\right),\qquad\tilde{P}\in\Poinc\left(\tilde{\vF},\tilde{\vR}\right),
\]
from the statement be the germs of a first return map defined on the
respective transversals $(\Sigma,\sigma)$ and $(\tilde{\Sigma},\tilde{\sigma})$.
Then, by assumption, there exists a germ of an analytic diffeomorphism
$\varphi$ between these transversals such that $\varphi\left(\sigma\right)=\widetilde{\sigma}$
and:
\[
\tilde{P}=\varphi P\varphi^{-1}.
\]
Let $P=RD$, $\tilde{P}=\tilde{R}\tilde{D}$ be dynamical decompositions
of $P$ and $\widetilde{P}$ based respectively at $\sigma$ and $\widetilde{\sigma}$.
The conjugacy relation between $P$ and $\tilde{P}$ gives $\tilde{R}\tilde{D}=\varphi RD\varphi^{-1}$,
or, equivalently, 
\[
(R^{-1}\varphi^{-1}\tilde{R})\tilde{D}\varphi=D.
\]
Let us consider the lift of these germs to the covering coordinates
$x=\ee^{-z}$, $y=\ee^{-w}$. Then, we obtain the relation 
\[
v\,\tilde{d}\,u=d,
\]
where the Dulac germs $\tilde{d},\ d$ are the lifts of $\tilde{D},\ D,$
and the unramified germs $u,v$ are respectively the lifts of $\varphi$
and $R^{-1}\varphi^{-1}\tilde{R}$. By applying the variation operator
to both sides, the identity~\eqref{eq:varxz} from the Appendix~\ref{sec:commut}
gives: 
\[
v\,\tmop{var}(\tilde{d})\,v^{-1}=\tmop{var}(d^{-1}).
\]
We now apply Theorem \ref{thm-Mattei-Moussu} to obtain a fibered
germ $\Phi\in\tmop{Diff}_{\tmop{fib}}(\mathbb{C}^{2},\bar{\Delta})$,
which maps $\tilde{\mathcal{F}}$ to $\mathcal{F}$ (as germs, on
some neighborhoods of the separatrices that contain $\lambda$-prepared
domains for $\tilde{\mathcal{F}}$ \emph{resp.} $\mathcal{F}$), such
that 
\begin{equation}
\Phi\Big|_{\Omega}=R^{-1}\circ\varphi^{-1}\circ\tilde{R}.\label{eq:i}
\end{equation}

Without loss of generality, we may assume that the transversal $(\tilde{\Sigma},\tilde{\sigma})$
above is chosen sufficiently close to zero, so that it lies in the
domain of definition of the equivalence $\Phi$ and its image by $\Phi$
belongs to some $\lambda$-adapted region for $\mathcal{F}$. Indeed,
instead of $\tilde{P}=\tilde{R}\tilde{D}$ on $(\tilde{\Sigma},\tilde{\sigma})$,
take another element $\tilde{P}_{1}\in\Poinc(\widetilde{\vF},\widetilde{\vR})$,
defined on some transversal $(\tilde{\Sigma}_{1},\tilde{\sigma}_{1})$
sufficiently close to the origin, so that it lies in the domain of
definition of $\Phi$. Let $\tilde{P}_{1}=\tilde{R}_{1}\tilde{D}_{1}$
be its dynamical decomposition at $\tilde{\sigma}_{1}$. We get that
$\tilde{D}_{1}=\tilde{D}k^{-1}$ and $\tilde{R}_{1}=k\tilde{R}$,
that is, $\tilde{P}_{1}=k\tilde{P}k^{-1}$, for some holonomy $k$
of $\tilde{\mathcal{F}}$. That means that $\tilde{P}_{1}=k\varphi P\varphi^{-1}k^{-1}$.
Applying Theorem~\ref{thm-Mattei-Moussu}, in the same way as above
we get that there exists a diffeomorphism $\Phi_{1}$ of foliations
$\mathcal{F}$ and $\tilde{\mathcal{F}}$ such that 
\[
\Phi_{1}\Big|_{\Omega}=R^{-1}\circ\varphi^{-1}\circ k^{-1}\circ\tilde{R}_{1}=R^{-1}\circ\varphi^{-1}\circ\tilde{R}=\Phi\Big|_{\Omega}.
\]
By the uniqueness of the analytic extension, $\Phi_{1}$ extends as
$\Phi$ analytically over the transversal $(\tilde{\Sigma}_{1},\tilde{\sigma}_{1})$.

Let us now consider the image of the loop germ $(\widetilde{\vF},\widetilde{\vR})$
under the action of $\Phi$, which we denote by $(\vF,\overline{\vR})$.
By construction, using \eqref{eq:i}, the $\Phi$-correspondent of
$\tilde{R}$ is given by 
\begin{equation}
\overline{R}=(\Phi|_{\tilde{\Sigma}})\tilde{R}(\Phi|_{\Omega})^{-1}=\Phi\Big|_{\tilde{\Sigma}}\circ\varphi\circ R.\label{eq:ih}
\end{equation}

What is left to prove is that $(\mathcal{F},\overline{\mathcal{R}})$
and $(\mathcal{F},\mathcal{R})$ define the same loop germ. That is,
that there exist $\overline{R}\in\overline{\mathcal{R}}$ , $R\in\mathcal{R}$
and a holonomy map $h$ of $\mathcal{F}$ such that $R=h\overline{R}$.

Take the decomposition $\overline{R}\,\overline{D}$ such that $\overline{R}$
and $\overline{D}$ are $\Phi$-correspondents of $\tilde{R}$ and
$\tilde{D}$. Then $\overline{P}=\overline{R}\,\overline{D}\in\Poinc(\mathcal{F},\overline{\mathcal{R}})$,
and: 
\begin{equation}
\overline{P}=(\Phi|_{\tilde{\Sigma}}\nobracket)\tilde{P}(\Phi|_{\tilde{\Sigma}}\nobracket)^{-1}=(\Phi|_{\tilde{\Sigma}}\nobracket)\circ\varphi\circ P\circ\varphi^{-1}\circ(\Phi|_{\tilde{\Sigma}}\nobracket)^{-1}.\label{eq:ihh}
\end{equation}

Since $\overline{D}$ and $D$ are corner transition germs of the
same foliation $\mathcal{F}$ (defined in its $\lambda$-invariant
domain), there exists a holonomy germ $h$ such that 
\begin{equation}
\overline{D}=Dh.\label{eq:ik}
\end{equation}
Now, putting~\eqref{eq:ik} and~\eqref{eq:ih} in~\eqref{eq:ihh},
writing $P=RD$ and $\overline{P}=\overline{R}\,\overline{D}$, we
get after simplifications that $h^{-1}=\Phi\Big|_{\tilde{\Sigma}}\circ\varphi$,
that is, by~\eqref{eq:ih}, $\overline{R}=h^{-1}R.$
\end{proof}

\subsection{\protect\label{sec:realization}Geometric realizations of a loop
germ}

We recall from the Introduction that a saddle loop in a foliated complex
analytic surface $(S,\mathcal{G})$ is defined by a saddle singularity
$s\in S$ and an oriented $C^{1}$-path $\Gamma\subset S$ that is
tangent to $\mathcal{G}$ and connects its local separatrices. Using
Proposition \ref{subsect:preparation-saddles}, it is easy to prove
that we can uniquely associate a loop germ $\mathbb{L}=(\vF,\vR)$
to the data given by $(S,\mathcal{G})$, $s$ and $\Gamma$.

In this subsection, our goal is to prove the converse. Namely, we
prove that any germ of a saddle loop can be {\em embedded} in a
foliated complex analytic surface. Let $\mathbb{L}=(\vF,\vR)$ be
a loop germ, as defined in Section~\ref{sec:gcc}. A {\em{gluing
domain}} for $\mathbb{L}$ is an open neighborhood $W\subset\cC^{2}$
of the horizontal unit disk $\overline{\Delta}$ whose traces on the
coordinates axis, $W\cap\{y=0\}$ and $W\cap\{x=0\}$ are simply connected
domains such that both $\vF$ and $\vR$ have realizations on $W$.

We recall that the latter condition means that $\vF$ extends to a
holomorphic foliation on $W$ and that one can choose some horizontal
transverse section $(\Sigma,\sigma)$ with base-point $(0,\sigma)$
in $\{x=0\}\cap W$ and a diffeomorphism 
\begin{equation}
R:(\Omega,1)\longrightarrow(\Sigma,\sigma)\label{realizationofRinW}
\end{equation}
that belongs to the regular transition class $\vR$ (see Definition~\ref{def:regulartransclass}).
\begin{rem}
We will use this map to quotient $W$ out by identifying a point of
(a foliated neighborhood of) $\Omega$ with its image by $R$. Notice
that we do not assume $W$ to be an adapted $\lambda$-neighborhood
as in Definition~\ref{def:prepared_saddle}. In fact, we will see
that $W$ will sometimes have to be shrunk in order to guarantee that
the quotient space is Hausdorff.
\end{rem}

Let $\left(S,\vG\right)$ be a foliated complex analytic surface.
\begin{defn}
\label{def:realizationsaddle} A {\em{(geometric) realization}}
of $\mathbb{L}$ in $\left(S,\vG\right)$ is given by the following
data:
\begin{enumerate}
\item A biholomorphic map $\Psi:W\rightarrow V$ between a gluing domain
$W\subset\cC^{2}$ and an open subset $V\subset S$ such that 
\[
\Psi\left(\vF|_{W}\right)=\left.\vG\right|_{V}
\]
In particular, $s=\Psi(0)$ is a saddle singularity of $\vG$.
\item An oriented curve $\gamma\subset S$ contained in a leaf of $\mathcal{G}$,
connecting the points $\Psi((0,1))$ to $\Psi((0,\sigma))$ and such
that the associated germ of a holonomy map $\holt{\vG,\gamma}$ between
$\Psi(\Omega)$ and $\Psi(\Sigma)$ satisfies 
\[
\holt{\vG,\gamma}=(\Psi|_{\Sigma})R(\Psi|_{\Omega})^{-1}
\]
for some germ $R\in\vR$ as in (\ref{realizationofRinW}).
\end{enumerate}
We will say $\Psi$ is the {\em{embedding map}} and that $\gamma$
is the {\em{connecting path}} of the realization.
\end{defn}

\begin{rem}
\label{rem:realisation-movingR} Notice that a geometric realization
is defined by choosing one representative $R\in\vR$ of the regular
transition class. However, once the embedding map $\Psi:W\rightarrow V$
and the connecting path $\gamma\subset S$ have been chosen, we can
obtain realizations for different choices of representative in $\vR$.
More precisely, in the notation of the previous definition, let $\delta\in\tmop{Mon}(W,\vF)$
be a path on the vertical separatrix between $\sigma$ and another
point $\tilde{\sigma}$ and let $\tilde{R}\in\vR$ be the germ from
$(\Omega,1)$ to $(\tilde{\Sigma},\tilde{\sigma})$ defined by 
\[
\tilde{R}=\holt{\vF,\delta}\,R
\]
(see Definition~\ref{def:regulartransclass}). Thus condition 2.
of the previous definition holds with $R$ replaced by $\tilde{R}$
if we change the path $\gamma$ to $\gamma\star\varepsilon$, where
the $\varepsilon\in\tmop{Mon}(S,\vG)$ is the image of $\delta$ under
the induced morphism of groupoids $\Psi_{*}:\tmop{Mon}(W,\vF)\rightarrow\tmop{Mon}(S,\vG)$.
\begin{figure}[htb]
\centering{}\includegraphics[height=5cm]{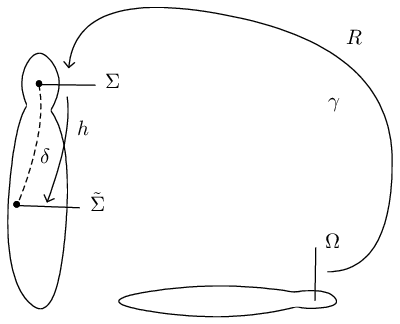}
\end{figure}
\end{rem}

Let us see some examples of a geometric realization.
\begin{example}
\label{example:real0} Consider the foliation $\mathcal{C}$ on $\mathbb{P}^{1}(\mathbb{C})$
defined by the polynomial differential form 
\[
\omega=\dd{(y^{2}-x^{2}+x^{3})}.
\]
The saddle point at the origin is linearizable, locally orbitally
equivalent to the linear foliation $\vF{}_{\mathrm{lin}}$ defined
by $\dd{\left(uv\right)}=0$. The separatrices $\{u=0\}$ and $\{v=0\}$
correspond to the local branches of the cubic $\Gamma=(y^{2}-x^{2}+x^{3}=0)$.
Up to a scaling, we can assume that the image $U$ of this orbital
equivalence contains the closure of the polydisk $\Delta\times\Delta$
in the $u,v$ variables.

Fixing the transverse sections $\Sigma=\{v=1\},\Omega=\{u=1\}$, equipped
with the natural parameterizations defined by the ambient coordinates,
and the connecting path $\gamma$ given by the real trace of $\Gamma$,
we easily see that $(\mathbb{P}^{1}(\mathbb{C}),\mathcal{C})$ contains
a realization of the loop germ 
\begin{equation}
\mathbb{L}=(\mathcal{F}_{1:1},\id).\label{trivialL}
\end{equation}
\end{example}

\begin{example}
\label{example:real0b} More generally, consider a polynomial $h\in\mathbb{R}[z,w]$
which has a simple critical point of saddle type at $p\in\mathbb{R}^{2}$
and such that that the real algebraic curve ${\Gamma_{\mathbb{R}}}=h^{-1}(\{h(p)\})$
contains a smooth component in ${\Gamma_{\mathbb{R}}}\setminus\{p\}$
connecting $p$ to itself. Then, this data also provides a realization
of the abstract saddle loop (\ref{trivialL}) in the foliated surface
$(\mathbb{P}^{1}(\mathbb{C}),\tmop{Fol}(\dd h))$. 
\begin{figure}[htb]
\centering{}\includegraphics[width=7.5cm,height=1.8cm]{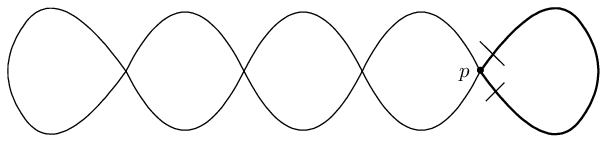}
\end{figure}
\end{example}

\begin{example}
\label{example:real1} Consider the following perturbation of the
integrable foliation given in Example \ref{example:real0}, 
\[
\omega_{\eps}=\dd{(w^{2}-z^{2}+z^{3})}+\eps\frac{(z^{2}-z^{3})^{2}}{wz}\dd{\left(\frac{w^{2}}{z^{2}-z^{3}}\right)}
\]
Notice that the cubic $\Gamma$ is an invariant algebraic curve for
all values of the parameter $\eps$. The ratio of eigenvalues at the
saddle point is given by $\lambda_{\eps}=\frac{\eps-1}{\eps+1}$.
By a suitable compactification, we can prove that the (closure of
the) complex leaf containing $\Gamma$ is homeomorphic to a pinched
torus, and that the holonomies are periodic when $\lambda_{\eps}$
is positive rational. Hence that the saddle point is linearizable.
We do not know if this is true for irrational values of $\lambda_{\eps}$
(we refer to \cite[Example 7.1 and Corollary 9.2]{Brunella2015} for
a more general discussion on foliations tangent to virtually elliptic
curves).
\end{example}

In the Examples~\ref{example:real0}--\ref{example:real1}, the
connecting curve $\gamma$ is contained in an algebraic leaf. Let
us see an example where this is not the case.
\begin{example}
Consider the one-parameter family of foliations defined by $\omega_{b}=P_{b}\dd y-Q_{b}\dd x$,
where 
\begin{align*}
P_{b} & =\frac{1}{4}\,{\left(2\,x+1\right)}{\left(2\,y+1\right)}-b\\
Q_{b} & =\frac{1}{4}\,{\left(2\,x-1\right)}{\left(2\,y-1\right)}-b
\end{align*}
This foliation has a Darbouxian first integral $H=\ee^{y-x}\left(xy+\frac{1}{2}(y-x)-b-\frac{1}{4}\right)$.
For each real $b<\frac{1}{4}$, such a foliation has a saddle loop
inside the non-algebraic leaf given by the level set $\{H=h_{c}\}$,
where $h_{c}=-\frac{1}{2}\left(\sqrt{-4\,b+1}-1\right)\ee^{-\sqrt{-4\,b+1}}$.
\end{example}

More generally, we can exhibit examples of non-algebraic saddle loops
inside non-integrable foliations.
\begin{example}
The {\em Kaypten-Dulac family} is the family of foliations defined
by the 6-parameter differential form $\omega_{\lambda}=P_{\lambda}\dd y-Q_{\lambda}\dd x$,
where 
\begin{align*}
P_{\lambda} & =-y+\lambda_{1}x-\lambda_{3}x^{2}+(2\lambda_{2}+\lambda_{5})xy+\lambda_{6}y^{2}\\
Q_{\lambda} & =x+\lambda_{1}y+\lambda_{2}x^{2}+(2\lambda_{3}+\lambda_{4})xy-\lambda_{2}y^{2}
\end{align*}
For $\lambda=(0,0,0,0,0,-1)$, the form $\omega_{\lambda}=-\dd{(y^{2}/2+y^{3}/3+x^{2}/2)}$
is integrable and has an algebraic saddle loop. It follows from~\cite[Section 5.3.2]{Roussarie1998}
that there exists a nonzero vector $(a,b,c)\in\cR^{3}$ such that
for all small perturbation of parameters lying inside the hypersurface
\[
a\lambda_{1}+b\lambda_{2}\lambda_{4}+c\lambda_{5}=0,
\]
the family still has a saddle loop. For generic values of parameters
inside this hypersurface, the associated foliation cannot be integrable,
since each one of the integrability strata in Kaypten-Dulac family
(given by the Dulac theorem) is of codimension greater or equal than
two.
\end{example}

All previous examples show realizations of abstract saddle loops in
algebraic foliated surfaces. However, it follows\footnote{The argument of~\cite{Genzmer2010} translates to resonant saddle
foliations at least when $\lambda=1$ through a blow-up of the saddle-node
singularity. It most probably does for other rational eigenratios.} from~\cite{Genzmer2010} that not every abstract saddle loop has
a realization in the algebraic category. Let us prove that a realization
always exists in the analytic category.
\begin{prop}
\label{prop:realisation} Each loop germ can be realized in a foliated
analytic complex surface.
\end{prop}

\begin{proof}
Let $(\vF,\vR)$ be a loop germ. We fix once and for all a $\lambda$-neighborhood
$U\subset\cC^{2}$ where $\vF$ is holomorphic and a germ $R$ as
in~(\ref{realizationofRinW}) which represents $\vR$ on $U$. The
basic idea of proof is quite simple: we will construct the surface
$S$ by {\em{holomorphically gluing}} the transverse sections
$\Omega$ and $\Sigma$ through $R$. However, we must be careful
to guarantee that the resulting topological space is Hausdorff.

We consider first two neighborhoods $V_{1},V_{2}\subset U$ of the
base-points $(1,0)$ and $(0,\sigma)$ of the transversals, equipped
with appropriate rectifying coordinates $(x_{1},y_{1})$, $(x_{2},y_{2})$.
More precisely, we choose $V_{1},V_{2}$ such that the next two conditions
are fulfilled.
\begin{enumerate}
\item $V_{1}\cap\{y=0\}$ is a disk centered in $(1,0)$ and there exists
a biholomorphism $\Psi_{1}:V_{1}\rightarrow\Delta\times\Delta$ mapping
$\left.\vF\right|_{V_{1}}$ to the horizontal foliation $\{y_{1}=\cst\}$,
such that $\Psi_{1}(\Omega)=\Delta\times\{0\}$ and $\Psi_{1}(y=0)=(y_{1}=0)$.
\item $V_{2}\cap\{x=0\}$ is a disk centered in $(0,\sigma)$ and there
exists a biholomorphism $\Psi_{2}:V_{2}\rightarrow\Delta\times\Delta$
mapping $\left.\vF\right|_{V_{2}}$ to the vertical foliation $\{x_{2}=\cst\}$,
such that $\Psi_{2}(\Sigma)=\{0\}\times\Delta$ and $\Psi_{2}(x=0)=(x_{2}=0)$.
\end{enumerate}
Up to reducing the size of such neighborhoods, we can further suppose
$\Psi_{1},\Psi_{2}$ have holomorphic extensions to the closures $\overline{V_{1}},\overline{V_{2}}$
and that 
\begin{equation}
\overline{V_{1}}\cap\overline{V_{2}}=\emptyset.\label{emptyinter2}
\end{equation}
We now choose two polydisks centered on the origin 
\begin{equation}
D_{1}=\Delta_{a}\times\Delta_{\eps},\quad D_{2}=\Delta_{\eps}\times\Delta_{b},\label{polyD}
\end{equation}
where $\eps>0$ is some small constant to be chosen later and the
radius $a,b>0$ are going to be chosen such that, intuitively, the
intersections $D_{1}\cap V_{1}$ and $D_{2}\cap V_{2}$ are {\em{sufficiently
small}}.

More precisely, we choose $0<a<1$ such that the region 
\[
J_{1}=\Psi_{1}((\Delta_{a}\times\{0\})\cap V_{1})
\]
is a non-empty open subset of the unit disk whose closure can be separated
from the origin by an affine line (see the figure below).

\begin{figure}[htb]
\centering{}\includegraphics[width=2.659878cm,height=2.70607cm]{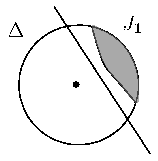}
\end{figure}

We choose similarly $0<b<|\sigma|$ such that $J_{2}=\Psi_{2}((\{0\}\times\Delta_{b})\cap V_{2})$
is a non-empty subset of $\Delta$ whose closure can be separated
from the origin by an affine line. It follows that we can choose a
rotation $r_{\lambda}(z)=\lambda z$, \ $|\lambda|=1$, \ such that
$r_{\lambda}\left(\overline{J_{1}}\right)\cap\overline{J_{2}}=\emptyset$.

We now modify one of the rectifying chart, say $\left(x_{1}{,y_{1}}\right),$
by post-composing the horizontal coordinate with the above rotation
$x_{1}\mapsto r_{\lambda}(x_{1})$. Therefore, in these new coordinates,
we obtain simply 
\begin{equation}
\overline{J_{1}}\cap\overline{J_{2}}=\emptyset.\label{s1s2}
\end{equation}
Up to taking $\eps>0$ sufficiently small, we can assume that the
same separation property holds for the whole polydisks, namely that
the sets 
\begin{equation}
\overline{\Psi_{1}(D_{1}\cap V_{1})}\quad\tmop{and}\quad\overline{\Psi_{2}(D_{2}\cap V_{2})}\label{separation}
\end{equation}
are disjoint. Up to further reducing $\eps$, we can therefore suppose
that the intersections 
\begin{equation}
\text{\ensuremath{\overline{D_{1}}\cap\overline{V_{2}}} and \ensuremath{\overline{D_{2}}\cap\overline{V_{1}}}}\label{emptyinter}
\end{equation}
are empty.

Now, using the coordinate charts $(x_{1},y_{1})$ and $(x_{2},y_{2})$
given above we fix a realization $x_{2}=\tilde{R}(y_{1})$ of the
germ $R$ between two small simply connected neighborhoods $\tilde{\Sigma}\subset\Sigma$
and $\tilde{\Omega}\subset\Omega$ of the corresponding base-points
and define a biholomorphism $\GlueM:\tilde{\Omega}\times\Delta\rightarrow\Delta\times\tilde{\Sigma}$
by 
\[
(x_{2},y_{2})=\GlueM(x_{1},y_{1})=(y_{1},\tilde{R}(x_{1})),
\]
which we call the {\em{gluing map}}. Notice that, according to
this definition, the restricted map $\GlueM|_{y_{1}=0}$ is the identity.
In particular, $\GlueM|_{y_{1}=0}\left(\overline{J_{1}}\right)=\overline{J_{1}}$
is disjoint from $\overline{J_{2}}$, according to (\ref{s1s2}).
Therefore, up to a further decrease of $\eps>0$, we can assume from~(\ref{separation})
that 
\begin{equation}
\GlueM\left(\overline{\Psi_{1}(D_{1}\cap V_{1})}\right)\text{ and }\overline{\Psi_{2}(D_{2}\cap V_{2})}\label{disj3}
\end{equation}
are disjoint.

We now consider the gluing domain given by the union $W=D_{1}\cup D_{2}\cup(\tilde{\Omega}\times\Delta)\cup(\Delta\times\tilde{\Sigma})$
and the quotient set 
\[
S=W/\thicksim
\]
where the equivalence relation $\thicksim$ identifies each point
$p$ in $\tilde{\Omega}\times\Delta$ to its image point $\GlueM(p)\in\Delta\times\tilde{\Sigma}$.

\begin{figure}
\centering{}\includegraphics[width=8.5895cm,height=4.90865cm]{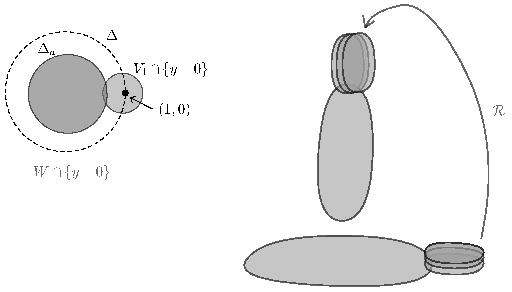}
\end{figure}

It is obvious that $S$ has a holomorphic structure inherited from
the holomorphic structure on $W$. Moreover, if we denote by $\pi:W\rightarrow S$
the quotient map, the image $\mathcal{G}=\pi(\mathcal{F})$ is a holomorphic
foliation in $S$ satisfying all the above requirements.

The only thing that remains to be proved is that $S$ is Hausdorff.
Or, equivalently, to show that the diagonal $\Delta(S)$ is closed
in the product $S\times S$. By the above construction, this amounts
to proving that the graph of the gluing map $\GlueM$, 
\[
\tmop{Graph}(\GlueM)=\{(p,\GlueM(p)):p\in(\tilde{\Omega}\times\Delta)\}
\]
is a closed subset of $W\times W$. But this property is a consequence
of the separation property (\ref{separation}).

Indeed, let us suppose by contradiction that there exists a limit
point $\left(\overline{p},\overline{q}\right)\in W\cap\partial\tmop{Graph}(\GlueM)$.
It follows from (\ref{emptyinter2}) and (\ref{emptyinter}) that
$\overline{p}\in D_{1}$ and $\overline{q}\in D_{2}$. On the other
hand, there exists a sequence 
\[
\{(p_{n},\GlueM(p_{n}))\}_{n\geqslant1}\subset\tmop{Graph}(\GlueM)
\]
converging to $\left(\overline{p},\overline{q}\right)$. Consider
the image of this sequence using the $\Psi_{1}$ and $\Psi_{2}$ maps.
Then, for all $n$ sufficiently large, $p_{n}$ lies in $\overline{\Psi_{1}(D_{1}\cap V_{1})}$
and $\GlueM(p_{n})$ lies in \ $\overline{\Psi_{2}(D_{2}\cap V_{2})}$,
which contradicts~(\ref{disj3}). The proposition is proved.
\end{proof}
\begin{rem}
\label{rem-Stein}Based on the above construction, a natural question
is whether one can choose $S$ to be a Stein manifold (that can be
embedded into $\cC^{n}$ for some sufficiently large $n$). Answering
this question will be the subject of an upcoming work.
\end{rem}

\subsection{\protect\label{subsec:real_loops}Real saddle loops}

We now consider the problem of classifying real analytic saddle loops.
We will say that a loop germ $\mathbb{L}=(\vF,\vR)$ is {\em{real}}
if:
\begin{enumerate}
\item we can choose a generator $x$ of $\mathcal{F}$ which is a real analytic
vector field;
\item we can choose a representative 
\[
R:(\Sigma,\sigma)\rightarrow(\Omega,1)
\]
that belongs to the transition class $\vR$, such that $\Sigma,\Omega$
are complexifications of real analytic curves (which we denote by
$\Sigma_{\mathbb{R}}$ and $\Omega_{\mathbb{R}}$) and $R:\Sigma_{\mathbb{R}}\rightarrow\Omega_{\mathbb{R}}$
is a real analytic map.
\end{enumerate}
Let us fix one such real analytic germ $R$. Up to a reflection and
multiplication of $x$ by $-1$, we can suppose that the configuration
of the transversals and the orientation of the solution curves of
$x$ is as illustrated below.

\begin{figure}[htb]
\centering{}\includegraphics[width=4.23383cm,height=4.21704cm]{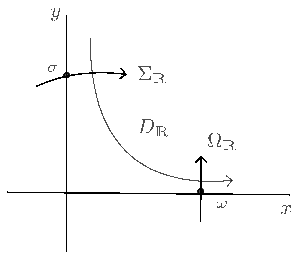}
\end{figure}

In the picture, we indicate by an arrow the orientation of the real
transversals in such a way that both ordered basis $\{x(\omega),T_{\omega}\Omega\}$
and $\{T_{\sigma}\Sigma,x(\sigma)\}$ are positively oriented (with
respect to the standard orientation of $\mathbb{R}^{2}$). Based on
these orientations, let us denote by $\Sigma_{\mathbb{R}_{>0}}$ and
$\Omega_{\mathbb{R}_{>0}}$ the positive parts of these transversals
(and define similarly $\Sigma_{\mathbb{R}_{<0}}$ and $\Omega_{\mathbb{R}_{<0}}$).

The leaf correspondence induced by the real trajectories uniquely
defines a determination of the corner transition that maps $\Sigma_{\mathbb{R}_{>0}}$
into $\Omega_{\mathbb{R}_{>0}}$. This is precisely the canonical
determination considered in Section~\ref{sub:three} (see Remark~\ref{rem:realsaddlecorner}).
We call it the {\em{real}} corner transition map associated to
$\mathbb{L}$, and write it simply $D$.

For later use, we also distinguish two possible cases concerning the
regular real transition map $R$ considered in the item 2. of the
above definition.
\begin{defn}
We will say that $R$ {\em{preserves the orientation}} if it
maps $\Omega_{\mathbb{R}_{>0}}$ to $\Sigma_{\mathbb{R}_{>0}}$. We
say that $R$ {\em{reverses the orientation}} if it maps $\Omega_{\mathbb{R}_{>0}}$
to $\Sigma_{\mathbb{R}_{<0}}$.
\end{defn}

\begin{rem}
If $R$ preserves the orientation then the {\em real} Poincaré
first return map, 
\[
P=RD
\]
is a real analytic germ mapping $\Sigma_{\mathbb{R}_{>0}}$ into itself.
\end{rem}

\subsubsection{Planar realizations \emph{via} Morrey-Grauert embedding}

Thanks to the realness assumption on $\mathbb{L}$, we prove a more
refined version of the Proposition~\ref{prop:realisation}, showing
that a {\em planar} realization always exists. In particular, this
provides a positive answer to the question asked in Remark~\ref{rem-Stein}
for real loop germs.

We specialize Definition~\ref{def:realizationsaddle} to the real
context.
\begin{defn}
\label{def:real-real} Let $\mathbb{L}=(\vF,\vR)$ be a real loop
germ. We say that a geometric realization of $\mathbb{L}$ in a foliated
surface $(S,\vG)$ is {\em real and planar} if:
\begin{enumerate}
\item $(S,\vG)$ is a complexification of a pair $(U,\mathcal{X})$ formed
by an open subset $U\subset\cR^{2}$ equipped with a real analytic
(oriented) foliation $\mathcal{X}$;
\item the embedding map $\Psi$ is the complexification of a real analytic
diffeomorphism, mapping the real solution curves of $\vF$ to the
real solution curves of $\mathcal{X}$ (preserving the orientations);
\item the connecting path $\gamma$ is real (\emph{i.e.} contained in $U$).
\end{enumerate}
Let us denote by $\Gamma\subset U$ the curve given by the saturation
of $\gamma$ by the leaves of $\mathcal{X}$. For shortness, we will
say that the triple $(U,\mathcal{X},\Gamma)$ is {\em a real planar
realization of $\mathbb{L}$}.
\end{defn}

\begin{rem}
\label{rem:realisation-movingRreal} We notice that a real planar
realization is defined by picking one representative $R\in\vR$ of
the regular transition class. As in Remark~\ref{rem:realisation-movingR},
once such a realization is obtained we can modify the representative
$R$ by pre-composing it with the holonomy of some path $\delta$,
\emph{i.e.} by taking 
\[
\tilde{R}=\holt{\vF,\delta}\,R
\]
and modifying accordingly the connecting path $\gamma$ to $\gamma\star\delta$.
Notice however that since we want to preserve the real transversals,
here we restrict to the case where the path $\delta$ is a segment
of the {\em real} vertical separatrix $\{x=0\}\subset\cR^{2}$
(see Figure~\ref{fig:real_transervsals_connection}). 
\begin{figure}[htb]
\begin{centering}
\includegraphics[height=4cm]{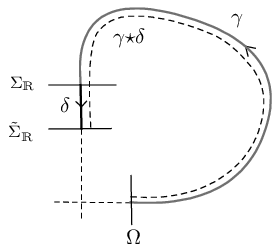}
\par\end{centering}
\caption{\protect\label{fig:real_transervsals_connection}}
\end{figure}
\end{rem}

The following result is essentially a consequence of the Morrey-Grauert
embedding theorem and the gluing argument that we used in the proof
of Proposition \ref{prop:realisation}. A similar construction is
provided in \cite[section 0.3 §C]{IlyaDu}.
\begin{prop}
\label{prop:realplanarreal} Suppose that $\mathbb{L}=(\vF,\vR)$
is real and that there exists a representative $R\in\vR$ which preserves
the orientation. Then, $\mathbb{L}$ has a real planar realization.
\end{prop}

\begin{proof}
We define a gluing domain $W\subset\mathbb{C}^{2}$ for $\mathbb{L}$
as in the proof of the Proposition \ref{prop:realisation}. By construction,
its real trace $W_{\mathbb{R}}=W\cap\mathbb{R}^{2}$ is as illustrated
in Figure~\ref{fig:real_flow_boxes}, where $F_{1}$ and $F_{2}$
are real flow boxes.

\begin{figure}[htb]
\begin{centering}
\includegraphics[width=5.35445cm,height=4.32305cm]{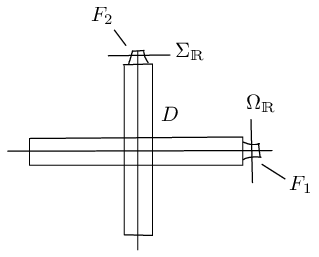}
\par\end{centering}
\caption{\protect\label{fig:real_flow_boxes}}
\end{figure}

Notice that here the gluing map $\GlueM$ can be chosen in such a
way that the following additional properties hold:
\begin{enumerate}
\item $\GlueM(F_{1})=F_{2}$;
\item $\GlueM$ maps the solution curves of $X$ on $F_{1}$ into the solution
curves of $X$ on $F_{2}$ {\em preserving the orientation}.
\end{enumerate}
As a result, the glued foliation $\mathcal{G}$ is real and preserves
the natural orientation of the real solution curves. Notice that the
orientation-preserving condition on $R$ is equivalent to the condition
that $S_{\mathbb{R}}$ be an orientable surface.

\textbf{Claim: }There exists a $C^{\infty}$ diffeomorphism mapping
$S_{\mathbb{R}}$ to an open subset $V_{0}$ of $\mathbb{R}^{2}$.

Indeed, let us choose an arbitrary real smooth curve $\gamma\subset\mathbb{R}^{2}$
connecting $F_{1}$ to $F_{2}$ as depicted by Figure~\ref{fig:gluing_real_flow_boxes}.

\begin{figure}[htb]
\begin{centering}
\includegraphics[width=5.80851cm,height=3.99772cm]{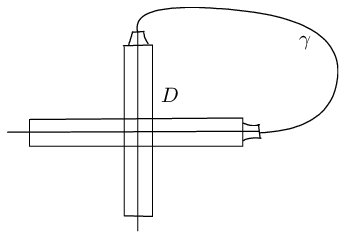}
\par\end{centering}
\caption{\protect\label{fig:gluing_real_flow_boxes}}
\end{figure}

Next we use a partition of unity subordinated to the covering $D\cup(F_{1}\thicksim_{\GlueM}F_{2})$,
and we can construct a smooth diffeomorphism $\Psi_{s}$ mapping $S_{\mathbb{R}}$
to $D\cup V_{0}$, where $V_{0}$ is an open neighborhood $\gamma$
(and such that $\Psi_{s}|_{D}=\id$). We leave the details to the
reader.

On the other hand, it follows from the Morrey-Grauert embedding theorem~\cite{10.2307/1970257}
that that there exists a real analytic proper embedding $\varphi:S_{\mathbb{R}}\rightarrow\mathbb{R}^{N}$
for some sufficiently large $N$. Using the embedding and the Weierstrass
Approximation Theorem on $\mathbb{R}^{N}$, we can find a real analytic
diffeomorphism $\Psi_{a}$ (arbitrarily close to $\Psi_{s}$) between
$S_{\mathbb{R}}$ and an open subset $U\subset\mathbb{R}^{2}$. The
realization of $\mathbb{L}$ in $U$ is obtained by taking the image
$\mathcal{X}$ of the foliation $\mathcal{G}$ under $\Psi_{a}$.
\end{proof}
\begin{rem}
Suppose that we modify the hypothesis by requiring instead that $R$
reverse the orientation. We could construct a realization of $\mathbb{L}$
in a Moëbius band using the same strategy as above.
\end{rem}

\subsubsection{Equivalence of planar realizations}

Using the Euclidean structure of $\cR^{2}$, we can easily construct
a global real analytic equivalence between planar realizations of
the same real loop germ.
\begin{prop}
\label{prop:equivalencesameloop} Let $(U,\mathcal{X},\Gamma)$ and
$(\tilde{U},\widetilde{\mathcal{X}},\tilde{\Gamma})$ be two planar
realizations of the \textbf{{same}} real loop germ $\mathbb{L}=(\vF,\vR)$.
Then, up to choosing $U$ and $\tilde{U}$ as smaller neighborhoods
of $\Gamma$ and $\tilde{\Gamma}$, there exist a real analytic diffeomorphism
$H$ mapping $(\tilde{U},\widetilde{\mathcal{X}})$ to $(U,\mathcal{X})$
and such that $H(\tilde{\Gamma})=\Gamma$.
\end{prop}

\begin{proof}
We use the Euclidean metric on $\mathbb{R}^{2}$ to define a local
transverse fibration in the vicinity of $\Gamma\setminus\{s\}$, where
$s$ is the saddle point. Namely, on each point $p\in\Gamma\setminus\{s\}$,
we define the fiber through $p$ as the affine line 
\[
\tmop{Fib}^{-1}(p)=\left\{ p+u\frac{X^{\perp}}{\|X^{\perp}\|}:u\in\mathbb{R}\right\} 
\]
where $X$ is an arbitrary local generator of $\mathcal{X}$. We define
similarly a transverse fibration $\widetilde{\tmop{Fib}}$ in the
vicinity of $\tilde{\Gamma}\setminus\{\tilde{s}\}$.

Let $\Psi$ and $\tilde{\Psi}$ be the embedding maps associated to
the two realizations (see Definition~\ref{def:real-real}). The real
analytic map $\Phi=\Psi\tilde{\Psi}^{-1}$ establishes an orientation
preserving equivalence between $\widetilde{\mathcal{X}}$ and $\mathcal{X}$
in corresponding neighborhoods of the saddle points $\tilde{s}$ and
$s$. Therefore, we need to show that this equivalence, which in principle
is only defined in the vicinity of the saddles, extends analytically
to a whole neighborhood of the loops.

Up to moving the transversals according to Remark~\ref{rem:realisation-movingRreal},
we can assume that $\Phi$ maps the transversals $\Psi(\Sigma),\Psi(\Omega)$
to $\tilde{\Psi}(\Sigma),\tilde{\Psi}(\Omega)$ respectively. Furthermore,
according to the identification made in Remark~\ref{rem:identifytransv},
we can assume that the transversals are fibers of the respective fibrations
$\tmop{Fib},\widetilde{\tmop{Fib}}$ defined above.

Note that, by the item 2. of the definition of a realization, we have
\[
\holt{\mathcal{X},\gamma}=(\Psi|_{\Sigma})R(\Psi|_{\Omega})^{-1},\quad\text{and}\quad\holt{\widetilde{\mathcal{X}},\tilde{\gamma}}=(\tilde{\Psi}|_{\Sigma})R(\tilde{\Psi}|_{\Omega})^{-1}
\]
where the connecting curve $\gamma$ (\emph{resp.} $\tilde{\gamma}$)
is simply the part of $\Gamma$ (\emph{resp.} ${\Gamma}$) lying between
$\Omega$ and $\Sigma$ (\emph{resp.} $\tilde{\Omega}$ and ${\Gamma}$).

\begin{figure}[htb]
\centering{}\includegraphics[width=5.51933cm,height=5.49383cm]{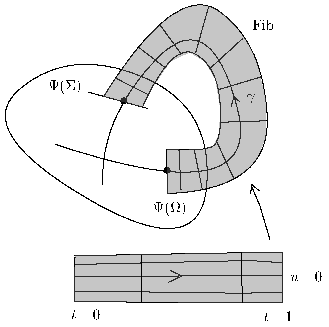}
\end{figure}

Let us fix real analytic regular parameterizations $p:[0,1]\rightarrow\gamma$,
\ $\tilde{p}:[0,1]\rightarrow\tilde{\gamma}$ of these curves.\quad{}We
obtain diffeomorphisms 
\[
(t,u)\rightarrow p(t)+u\frac{X^{\perp}}{\|X^{\perp}\|},\quad(t,u)\rightarrow\tilde{p}(t)+u\frac{\tilde{X}^{\perp}}{\|\tilde{X}^{\perp}\|}
\]
 between neighborhoods of $[0,1]\times\{0\}$ in $[0,1]\times\mathbb{R}$
and neighborhoods of $\gamma$ and $\tilde{\gamma}$ in $V$ and $\tilde{V}$,
respectively. Let us keep the notations $\mathcal{X}$ and $\widetilde{\mathcal{X}}$
to indicate the image of the above foliations under these diffeomorphisms.

We now define an analytic diffeomorphism between two open neighborhoods
of $[0,1]\times\{0\}$ in $[0,1]\times\mathbb{R}$, mapping $\widetilde{\mathcal{X}}$
to $\mathcal{X}$. To wit, we set 
\[
\rho(t,u)=(t,\holt{\mathcal{X},t}(\Phi|_{\tilde{\Psi}(\Omega)})\holt{\widetilde{\mathcal{X}},t}^{-1}(u))
\]
where $\holt{\mathcal{X},t}$ (\emph{resp.} $\holt{\widetilde{\mathcal{X}},t}$)
is the holonomy map of $\mathcal{X}$ (\emph{resp.} $\widetilde{\mathcal{X}}$)
going from the fiber $\{0\}\times\mathbb{R}$ to the fiber $\{t\}\times\mathbb{R}$.

It remains to show that the equivalence defined by $\rho$ is compatible
with the local equivalence in the neighborhood of the saddle points
defined by $\Phi$. For this, it suffices to show that 
\[
\rho(0,\cdot)=\Phi|_{\tilde{\Psi}(\Omega)}\quad\tmop{and}\quad\rho(1,\cdot)=\Phi|_{\tilde{\Psi}(\Sigma)}.
\]
The first equality is obvious from the definition. For the second,
we observe that, by construction, 
\[
\holt{\mathcal{X},1}=(\Psi|_{\Sigma})R(\Psi|_{\Omega})^{-1}\quad\text{and}\quad\holt{\widetilde{\mathcal{X}},1}=(\tilde{\Psi}|_{\Sigma})R(\tilde{\Psi}|_{\Omega})^{-1}.
\]
Therefore the second component of $\rho(1,\cdot)$ is given by 
\[
(\Psi|_{\Sigma})R(\Psi|_{\Omega})^{-1}(\Phi|_{\tilde{\Psi}(\Omega)})(\tilde{\Psi}|_{\Omega})R^{-1}(\tilde{\Psi}|_{\Sigma})^{-1}=\Phi|_{\tilde{\Psi}(\Sigma)},
\]
where the middle terms simplify because $\Phi=\Psi\tilde{\Psi}^{-1}$.
This concludes the proof.
\end{proof}

\subsubsection{$\protect\cR$-equivalence and proof of Theorem C}

In the present setting, it is natural to adapt the Definition~\ref{def:fib_eqv_abstract_loops}
of equivalence stated in Section~\ref{sec:gcc} by requiring that
the equivalence map preserve the reals. The group of {\em fibered
orientation preserving real analytic diffeomorphisms} is the group
$\tmop{Diff}_{\tmop{fib}}^{+}(\mathbb{\cR}^{2},[-1,1])$ of germs
of the form 
\[
\Phi(x,y)=(x,\varphi(x,y)),
\]
where $\varphi$ is real analytic in a neighborhood of $[-1,1]\times\{0\}$
and satisfies the following three conditions:
\begin{itemize}
\item $\varphi(x,0)=0$;
\item ${\displaystyle {\frac{\partial\varphi}{\partial y}}(x,0)>0}$;
\item $\varphi$ extends holomorphically to a neighborhood of the unit closed
disk $\bar{\Delta}\times\{0\}$ in $\cC^{2}$.
\end{itemize}
This last condition implies that $\tmop{Diff}_{\tmop{fib}}^{+}(\mathbb{\cR}^{2},[-1,1])$
is a subgroup of $\tmop{Diff}_{\tmop{fib}}(\mathbb{\cC}^{2},\bar{\Delta})$.
\begin{defn}
Two real loop germs $\mathbb{L}=(\vF,\vR)$ and $\tilde{\mathbb{L}}=(\widetilde{\vF},\widetilde{\vR})$
are {\em $\mathbb{R}$-equivalent} if 
\[
(\widetilde{\vF},\widetilde{\vR})=\Phi\cdot\left(\vF,\vR\right)
\]
for some germ $\Phi\in\tmop{Diff}_{\tmop{fib}}^{+}(\mathbb{\cR}^{2},[-1,1])$.
\end{defn}

The following result is a consequence of Theorems~A and~B for real
loop germs.
\begin{thm}
\label{thm:realequiv-poincare} Two real loop germs $\mathbb{L}$
and $\tilde{\mathbb{L}}$ are $\mathbb{R}$-equivalent if and only
if the corresponding real Poincaré first return maps $P$ and $\tilde{P}$
are conjugate by an orientation-preserving real analytic map.
\end{thm}

\begin{proof}
This enunciate is a particular case of Theorems~\ref{thm:A} and~\ref{thm:B}
proved above. Recalling the arguments used in the proof of those results,
we remark that the realness of $\mathbb{L}$ and $\tilde{\mathbb{L}}$
leads to the following significant simplifications:
\begin{enumerate}
\item the real Poincaré first return maps $P$ and $\widetilde{P}$ associated
respectively to $\mathbb{L}$ and to $\tilde{\mathbb{L}}$ do not
depend on a specific choice of a connecting curve;
\item according to the Remark~\ref{rem:Realmatteimoussu}, the path-lifting
method of Mattei-Moussu will give a fibered biholomorphism which is
a complexification of a real analytic map.
\end{enumerate}
We further observe that, assuming $\varphi$ to preserve the orientation,
the analytic equivalence obtained by the Mattei-Moussu theorem will
give a germ lying in the group $\tmop{Diff}_{\tmop{fib}}^{+}(\mathbb{\cR}^{2},[-1,1])$
which establishes an equivalence between the underlying real analytic
foliations. We leave the details to the reader.
\end{proof}
As a consequence, we derive the proof of Theorem C.
\begin{proof}[Proof of Theorem C]
The fact that equivalent real saddle loops have analytically conjugate
(real) Poincaré maps is trivial since there is no ambiguity on the
choice of a determination of the corner transition map.

Let us prove the converse. Let $\mathbb{L}$ and $\tilde{\mathbb{L}}$
be the saddle loop germs associated to $(U,X,\Gamma)$ and $(\tilde{U},\tilde{X},\tilde{\Gamma})$
respectively. Then, by Theorem~\ref{thm:realequiv-poincare}, $\mathbb{L},\tilde{\mathbb{L}}$
are $\mathbb{R}$-equivalent. Using Proposition~\ref{prop:realplanarreal},
we find a common real planar realization $(V,\mathcal{Y},\Omega)$
of both $\mathbb{L},\tilde{\mathbb{L}}$. Proposition~\ref{prop:equivalencesameloop},
establishes a local equivalence between $(V,\mathcal{Y},\Omega)$
and $(U,\mathcal{X},\Gamma)$ and also between $(V,\mathcal{Y},\Omega)$
and $(\tilde{U},\widetilde{\mathcal{X}},\tilde{\Gamma})$. The result
follows from the transitivity of the relation of local equivalence.
\end{proof}

\section{\protect\label{sec:integrability}Integrability of loop germs}

In the following we consider a complex saddle loop given as a loop
germ $\left(\fol{},\mathcal{R}\right)$ as in Section~\ref{sec:gcc}.
It means that we work in a $\lambda$-neighborhood $U$ of the unit
polydisk $\bar{\Delta}\times\bar{\Delta}=\left\{ \left|x\right|\leqslant1~,~\left|y\right|\leqslant1\right\} $
in $\cc^{2}$ and consider the foliated manifold $\left(S,\mathcal{G}\right)$
given by Proposition~\ref{prop:realisation}. We recall that we pick
some representative $R\in\mathcal{R}$ to identify the leaves of $\fol{}$
near the transversals $\Sigma=\left\{ y=1\right\} $ and $\Omega=\left\{ x=1\right\} $.

A theorem of M.~Singer~\cite{SingLiou} guarantees that $\fol{}$
admits a Liouvillian first-integral if and only if there exists a
\emph{Godbillon-Vey sequence of length 2, i.e. }a pair $\left(\omega,\eta\right)$
of differential 1-forms on $U$ such that $\omega$ is holomorphic,
$\eta$ is meromorphic and
\begin{align}
\begin{cases}
\dd{\omega} & =\eta\wedge\omega\\
\dd{\eta} & =0
\end{cases} & ,\label{eq:GodVey}
\end{align}
where the distribution $\ker\omega$ defines the foliation $\fol{}$.
In that case, the (generally multivalued) function obtained by integrating
the closed 1-form $\exp\left(-\int\eta\right)\omega$,
\begin{align}
H & =\int\frac{\omega}{\exp\int\eta},\label{eq:liouvillian_first-integ}
\end{align}
is a Liouvillian first-integral of $\fol{}$ since $\dd H\wedge\omega=0$.

The study we conduct here amounts to picking among all integrable
prepared saddles $\floc$ those that can be glued by some germ of
a biholomorphism $R$ to obtain a loop germ $\left(\floc,\mathcal{R}\right)$,
in such a way that $R$ preserves the ``niceness'' of the transverse
structure provided by Godbillon-Vey sequences. Because we want $H|_{\Sigma}R$
to be the restriction of a first-integral of $\floc$ near $\Omega$,
it must coincide with a determination of $H|_{\Omega}$. Observe that
the special form~\eqref{eq:liouvillian_first-integ} of the Liouvillian
first-integral $H$ forces its monodromy to be affine (since the monodromy
of $\int\eta$ is additive). Therefore, we are naturally led to give
the following definition of integrability for loop germs.
\begin{defn}
\label{def:integrable}The loop germ $\left(\fol{},\mathcal{R}\right)$
is said to be \emph{integrable }whenever there exists a Liouvillian
first integral $H$ of $\fol{}$ which is compatible with a gluing
map $R\in\mathcal{R}$:
\begin{align}
\exists R\in\mathcal{R},~\alpha\in\cc^{\times},~c\in\cc~:~~~~ & H_{\text{}}|_{\Sigma}R=\alpha H|_{\Omega}+c.\label{eq:liouvillian-R}
\end{align}
\end{defn}

In the rest of the section, we use the following immediate characterization
of integrability.
\begin{lem}
\label{lem:integrable_regluing}A loop germ $\left(\floc,\mathcal{R}\right)$
is integrable if and only if both conditions are met:
\begin{itemize}
\item $\floc$ is integrable with Liouvillian first-integral $H=\int\frac{\omega}{\exp\int\eta}$;
\item there exists $R\in\mathcal{R}$ and $\alpha\in\cc^{\times}$ such
that $R^{*}\frac{\omega}{\exp\int\eta}|_{\Sigma}=\alpha~\frac{\omega}{\exp\int\eta}|_{\Omega}$.
\end{itemize}
Here $\frac{\omega}{\exp\int\eta}|_{\Sigma}$ (resp. $\frac{\omega}{\exp\int\eta}|_{\Omega}$)
denotes the pulled-back 1-form $\iota_{j}^{*}\frac{\omega}{\exp\int\eta}$
by the respective inclusion $\iota_{1}~:~x\mapsto\left(x,1\right)$
and $\iota_{2}~:~y\mapsto\left(1,y\right)$.
\end{lem}

\begin{proof}
It suffices to differentiate both sides of~\eqref{eq:liouvillian-R}.
\end{proof}
\begin{rem}
If $\left(\omega,\eta\right)$ is a Godbillon-Vey sequence of length
2 and $u$ is a meromorphic function, then $\left(u\omega,u\eta+\dd u\right)$
satisfies again~\eqref{eq:GodVey}. Said differently, we can freely
choose the holomorphic 1-form $\omega$ defining the saddle foliation,
and we may even choose it meromorphic if it better suits our needs
(\emph{e.g.} if it allows us to work with closed 1-forms defining
the same foliation).
\end{rem}

\subsection{Integrability and holonomy}

We let $\holo[\Sigma]$ and $\holo[\Omega]$ be the holonomy mappings
of $\floc$ as defined in Definition~\ref{def:saddle_holo}, and
denote by $\cloc\in\Corner\left(\mathcal{F},\mathcal{R}\right)$ a
representative of the corner transition map of the saddle foliation
$\fol{}$. We recall that the holonomy group of $\mathcal{G}$ is
given by $\Var f=\left\langle R^{*}\holo[\Sigma],\holo[\Omega]\right\rangle $,
where $f\in\Poinc$$\left(\mathcal{F},\mathcal{R}\right)$ is a Poincaré
map of the loop germ associated to the dynamical decomposition $f=R\cloc$
for some $R\in\mathcal{R}$.
\begin{thm}[\cite{BeCeLi}]
If the foliation $\mathcal{G}$ is integrable, then its holonomy
group is solvable.
\end{thm}

Solvability imposes very strict limitations on the holonomy group
of $\fol{}$ and is seldom achieved. Solvable subgroups of $\diff[\cc,0]$
were intensively studied in the 90's, hence we can use their analytical
classification (we refer to the Appendix for the precise sources and
results we rely on here).

If $\Var f$, coming from some loop germ $\left(\floc,\mathcal{R}\right)$,
is analytically conjugate to $\Var{\widetilde{f}}$, coming from $\left(\widetilde{\floc},\widetilde{\mathcal{R}}\right)$,
then $\fol{}$ is analytically conjugate to $\widetilde{\fol{}}$
(by Mattei-Moussu theorem) and, in this new $\U$-coordinate, $f=r\widetilde{f}$
for some unramified germ $r\in\U$ commuting with a generator of the
holonomy (Proposition~A.4). In most situations, $r$ must itself
be a power of the holonomy, hence (by definition of $\mathcal{R}$
and Theorem~B) the loop germs $\left(\floc,\mathcal{R}\right)$ and
$\left(\widetilde{\floc},\widetilde{\mathcal{R}}\right)$ are equivalent.
In the remaining cases, we are able to describe exactly which $r$
correspond to integrable loop germs.

Together with the theory developed by M.~Berthier and F.~Touzet
in~\cite{BerTouze}, this argument allows us to give the complete
list of (analytic classes of) all integrable abstract saddle loops.

\subsection{Admissible gluing maps}

The classification result presented here is a consequence of the following
characterization of germs $R$ appearing in Lemma~\ref{lem:integrable_regluing}
under suitable conditions. The fact that the restricted $1$-forms
$\frac{\omega}{\exp\int\eta}|_{\Sigma}$ and $\frac{\omega}{\exp\int\eta}|_{\Omega}$
meet the requirements of the proposition below will be made clear
in the course of the classification.
\begin{prop}
\label{prop:transverse_groupoid}For a finite collection of complex
numbers $a$, $b_{1},\ldots,b_{k}$ and a germ $A$ at $0$ of a non-vanishing
holomorphic function, define 
\begin{align*}
F\left(y\right) & =A\left(y\right)y^{-a}\prod_{0<j\leqslant k}\exp\left(b_{j}y^{-j}\right).
\end{align*}
Assume $R\in\diff[\cc,0]$ and $\alpha\in\cc^{\times}$ are such that
\begin{align*}
R'\left(y\right)\times F\left(R\left(y\right)\right) & =\alpha F\left(y\right).
\end{align*}
Then, either $R$ is analytically linearizable or $R'\left(0\right)\in\ee^{2\ii\pi\qq}$
and $R$ is analytically conjugate to $R'\left(0\right)R_{0}$, where
$R_{0}$ is the (tangent to the identity) holonomy of a Bernoulli
differential equation computed on $\left\{ x=1\right\} $: 
\begin{align}
y^{k+1}\dd x & =x\left(1+ay^{k}+x^{d}y^{\sigma}Q\left(y\right)\right)\dd y\label{eq:bernoulli_diff}
\end{align}
for some polynomial $Q\left(y\right)$ of degree at most $k-1$, some
integers $d\in\zz_{\geqslant-1}\backslash\left\{ 0\right\} $ and
$\sigma\in\zz_{\geqslant0}$ chosen in such a way that $\sigma+ad\notin\rr_{\leqslant0}$.
\end{prop}

\begin{defn}
\label{def:bernoulli}A \emph{Bernoulli diffeomorphism of order $k$}
is a germ $R=\theta R_{0}$, where $R_{0}$ is given by the holonomy
of a Bernoulli differential equation~\eqref{eq:bernoulli_diff} as
in the previous proposition and $\theta\in\ee^{2\ii\pi\qq}$.
\end{defn}

\begin{rem}
~
\begin{enumerate}
\item This case also covers the situation where $R_{0}$ embeds in a holomorphic
flow ($Q=0$).
\item Since Bernoulli equations are explicitly solvable, Bernoulli diffeomorphisms
can be expressed using a quadrature.
\end{enumerate}
\end{rem}

\begin{proof}
Without loss of generality we can assume that $\left|R'\left(0\right)\right|=1$,
else $R$ is analytically linearizable, so that $R\left(y\right)=\ee^{2\ii\pi\beta}y+\oo y$
for some $\beta\in\rr$. The hypothesis implies $\exp\left(\left(1-a\right)2\ii\pi\beta\right)=\alpha$
and, if $b_{j}\neq0$, $j\beta\in\zz$. Let $H$ be an antiderivative
of $F$; there exists $c\in\cc$ such that $H$ semi-conjugates $R$
to the affine map $\rho~:~t\mapsto\alpha t+c$:
\[
HR=\alpha H+c.
\]

The case $\beta\in\qq$ has been dealt with in~\cite[Proposition 4.1 and Proposition 5.5]{BerTouze},
where they conclude $R\left(y\right)=\ee^{2\ii\pi\beta}R_{0}\left(y\right)$,
with $R_{0}$ being the holonomy of a Bernoulli foliation. The fact
that this equation can be chosen polynomial of the specified form
is proved in~\cite[Theorem 2]{Tey-ExSN}.

We are left with the case $\beta\notin\qq$, so that all $b_{j}$
vanish. We assume for the sake of contradiction that $R$ is not analytically
linearizable. According to the proof of Dulac-Moussu's conjecture
by R.~Perez-Marco~\cite[Section IV.2]{PM}, there exists a forward
sub-orbit $\left(y_{n}\right)_{n\in\zp}=\left(R^{\circ k_{n}}\left(y_{0}\right)\right)_{n\in\zp}$,
with $y_{0}\neq0$ arbitrarily close to $0$, converging to the fixed-point
$0$. Using a case-by-case analysis, we show that in every possible
situation this property cannot be fulfilled. The contradiction we
reach is that the complete orbit $\left(R^{\circ n}\left(y_{0}\right)\right)_{n\in\zp}$
converges to $0$, but this cannot be the case as the sub-orbit $\left(R^{\circ q_{n}}\left(y_{0}\right)\right)_{n\in\zp}$
converges towards $y_{0}$ when $\left(q_{n}\right)_{n}$ is the sequence
of denominators of the convergents of $\lambda=\lim_{n\to+\infty}\frac{p_{n}}{q_{n}}$.
\begin{enumerate}
\item Suppose first $\alpha=1$, so that $a=1$ and $R$ is semi-conjugate
to the translation $\rho~:~t\mapsto t+c$ by $H\left(y\right)\sim A\left(0\right)\log y$.
The limit of $\left|H\left(y\right)\right|$ as $y$ goes to $0$
is $\infty$. For $H\left(R^{\circ n}\left(y_{0}\right)\right)=H\left(y_{0}\right)+nc$
to tend to $\infty$ it is necessary that $c\neq0$, from which we
deduce $\lim_{n\to+\infty}R^{\circ n}\left(y_{0}\right)=0$.
\item Assume next that $\alpha\neq1$. By subtracting the constant $\frac{c}{\alpha-1}$
to $H$ we may assume without loss of generality that $c=0$ and $R$
is semi-conjugate to the linear map $\rho~:~t\mapsto\alpha t$.
\begin{enumerate}
\item If $a\notin\rr$, then $\left|\alpha\right|\neq1$ and $\lim_{n\to+\infty}\alpha^{n}$
is either $0$ or $\infty$. Therefore $H\left(R^{\circ n}\left(y_{0}\right)\right)=\alpha^{n}H\left(y_{0}\right)$
also converges towards the same limit, because we may choose $y_{0}$
in such a way that $H\left(y_{0}\right)\notin\left\{ 0,\infty\right\} $.
But this can only mean $\lim_{n\to+\infty}R^{\circ n}\left(y_{0}\right)=0$.
\item If $a\in\rr$, then $\left|\alpha\right|=1$ and $H\left(y\right)$
tends to $0$ or $\infty$ as $y\to0$. For all $n\in\zp$ we have
$\left|H\left(R^{\circ k_{n}}\left(y_{0}\right)\right)\right|=\left|H\left(y_{0}\right)\right|$,
which is clearly impossible.
\end{enumerate}
\end{enumerate}
\end{proof}

\subsection{Classification of integrable loop germs}

We are now ready to state the classification of integrable saddle
loops. The most striking fact is that the holonomy group of an integrable
foliation must be abelian: no purely metabelian holonomy can correspond
to integrable loop germs.
\begin{thm}
\label{thm:regluing_integrability}Any integrable saddle loop, given
as a loop germ, is $\tmop{Diff}_{\tmop{fib}}(\mathbb{C}^{2},\bar{\Delta})$-equivalent
to one of the loop germs appearing in the following list.
\begin{lyxlist}{00.00.0000}
\item [{\textbf{Linear~model}}] $\left(\fol{1:\lambda},\mathcal{R}\right)$,
with eigenratio $-\lambda<0$ and $R\in\GL 1{\cc}$, defined by the
closed $1$-form
\begin{align*}
\omega_{1:\lambda} & =\frac{\dd y}{y}+\lambda\frac{\dd x}{x}.
\end{align*}
\item [{\textbf{Bernoulli~model}}] $\left(\fol{1:1},\mathcal{R}\right)$,
where $R$ is a Bernoulli diffeomorphism.
\item [{\textbf{Poincaré-Dulac~model}}] $\left(\fol{k,\mu},\mathcal{R}\right)$,
with eigenratio $-1$ and:
\begin{itemize}
\item $\fol{k,\mu}$ is the $1:1$ resonant saddle foliation defined by
a closed $1$-form 
\begin{align*}
\omega_{k,\mu} & =\frac{\dd y}{y}-\frac{1+\mu u^{k}}{u^{k}}\omega_{1:1}
\end{align*}
for some $k\in\zz_{\geqslant1}$ and $\mu\in\cc$, where $u=xy$ is
the resonant monomial (so that $\omega_{1:1}=\frac{\dd u}{u}$);
\item $R=\exp\chi$ is the exponential of a derivation 
\begin{align*}
\chi\left(y\right) & =\frac{y^{k+1}}{1+\nu y^{k}}\pp y~,~\nu\in\cc.
\end{align*}
\end{itemize}
\end{lyxlist}
\end{thm}

\begin{proof}
We start from an integrable loop germ $\left(\floc,\mathcal{R}\right)$
with gluing map $R$, and let $\left\langle R^{*}\holo[\Sigma],\holo[\Omega]\right\rangle $
be its holonomy group.
\begin{enumerate}
\item Let us first deal with the simplest case $\lambda\notin\qq$. Theorem~3.1
from~\cite{BerTouze} states in that case that $\floc$ is analytically
linearizable. Hence, up to considering the pulled-back loop germ as
in Definition~\ref{def:fib_eqv_abstract_loops}, one can suppose
that $\floc=\fol{1:\lambda}$ and choose $\omega=\omega_{1:\lambda}$.
Because $\dd{\omega}_{1:\lambda}=0=\eta\wedge\omega_{1:\lambda}$,
we deduce that $\eta=a\omega_{1:\lambda}$ for some meromorphic germ
$a$. Plugging this identity into $\dd{\eta}=0$ we infer that $a$
is a meromorphic first-integral of $\omega_{1:\lambda}$, hence a
constant $a\in\cc$. Therefore $\exp\int\eta=cx^{\lambda a}y^{a}$
for some $c\in\cc^{\times}$.\\
Next we apply Lemma~\ref{lem:integrable_regluing}: there exists
$\alpha\in\cc^{\times}$ for which $R$ must solve
\begin{align*}
\lambda\frac{R'}{R^{1+\lambda a}}\left(y\right) & =\alpha\frac{1}{y^{1+a}}.
\end{align*}
Because $R\left(y\right)=\gamma y\left(1+\OO y\right)$ for some $\gamma\neq0$
we deduce $a=0$ and $\lambda=\alpha$. Finally $R\in\GL 1{\cc}$.
\item Assume next $\lambda=\frac{p}{q}\in\qq_{>0}\setminus\left\{ 1\right\} $.
Following Proposition~\ref{prop:solvable_pq} from the Appendix,
and after a convenient analytic change of coordinates, we can assume
that $\holo[\Omega]$ is linear. It stems from Mattei-Moussu theorem
that $\floc$ can be chosen as the linear foliation $\fol{1:\lambda}$,
that is $\omega=\omega_{1:\lambda}$. We let $u=x^{p}y^{q}$ be the
resonant monomial so that $\omega_{1:\lambda}=\frac{\dd u}{qu}$ and
every meromorphic first-integral of $\omega_{1:\lambda}$ factors
meromorphically through $u$~\cite{MaMou}. Following the same line
of reasoning as for 1., we have $\eta=G\left(u\right)\omega_{1:\lambda}$
for some meromorphic germ $G$. We compute directly 
\begin{align*}
F\left(u\right) & =\exp\int\eta=A\left(u\right)u^{-a}\prod_{0<j\leqslant k}\exp\left(b_{j}u^{-j}\right)~~,~a,~b_{1},\ldots,b_{k}\in\cc,
\end{align*}
where $A\left(0\right)\neq0$. Thanks to Lemma~\ref{lem:integrable_regluing},
we know there exists $\alpha\in\cc^{\times}$ such that
\begin{align*}
\lambda\frac{R'}{R\times F\left(R^{p}\right)} & =\alpha\frac{1}{yF\left(y^{q}\right)}.
\end{align*}
Because $p\neq q$ and $R\in\diff$, the function $F$ must be constant:
we recover the case studied in 1., yielding a linear model $R\in\GL 1{\cc}$.
Observe in particular that the only possible alternative described
by Proposition~C.1 is $\varepsilon=0$.
\item Let us continue the classification with the case $\lambda=1$. The
symmetry offered by the variables $x$ and $y$ explains why this
case is so rich. Let us first consider the case of a foliation $\floc$
whose formal Poincaré-Dulac normal form is linear: it is actually
analytically linearizable and we may assume that $\omega=\omega_{1:1}$
and $\floc=\fol{1:1}$. We let $u=xy$ be the resonant monomial so
that $\omega_{1:1}=\frac{\dd u}{u}$ and every meromorphic first-integral
of $\omega_{1:1}$ factors meromorphically through $u$.\\
In that case the holonomy group is trivial, the corner transition
map $\text{\ensuremath{\cloc\ }}$ is the identity and the Poincaré
map $f$ coincides with $\Pi^{*}R$. According to Theorem~B, the
analytical class of $R$ determines that of $\left(\floc,\mathcal{R}\right)$.
Moreover, an equivalence $\varphi^{*}R=\widetilde{R}$ between two
such analytic germs induces a fibered isotropy of $\fol{}$ through
$\left(x,y\right)\mapsto\left(x,y\frac{\varphi\left(u\right)}{u}\right)$,
sending $u$ to $\varphi\left(u\right)$. Hence, we can change $R$
for any element $\widetilde{R}$ of its conjugacy class while preserving
at the same time the linearity of $\floc$ and the analytical class
of $\left(\floc,\mathcal{R}\right)$.\\
We have again $\eta=G\left(u\right)\omega_{1:1}$ for some meromorphic
germ $G$. Thanks to Lemma~\ref{lem:integrable_regluing}, the germ
$R$ and the multivalued function $\exp\int\eta$ meet the requirements
of Proposition~\ref{prop:transverse_groupoid}. $R$ is thus either
analytically conjugate to the rotation $\ee^{2\ii\pi\beta}\id$, giving
a linear model, or to $\ee^{2\ii\pi\beta}R_{0}$ (only if $\beta\in\qq$),
where $R_{0}$ is a Bernoulli diffeomorphism of order $\left|m\right|$,
giving a Bernoulli model.
\item Let us conclude our proof when $\lambda=1$ and the Poincaré-Dulac
normal formal of $\floc$ is nonlinear. Each generator $\holo[j]$
of the holonomy group is tangent to the identity, hence $\Var f$
is solvable if and only if it is abelian~\cite{CerMou}. Proposition~\ref{prop:GH}
describes the only two situations that may happen.
\begin{enumerate}
\item In convenient analytic coordinates, $\holo[\Omega]=\Exp{\partial}$
and $R^{*}\holo[\Sigma]=\Exp{\left(c\partial\right)}$ are embedded
in the same holomorphic flow. Then $\floc$ is analytically conjugate
to its Poincaré-Dulac normal form $\fol{k,\mu}$. Since $\holo[\Sigma]=\exp\left(c'\partial\right)$
for some other $c'\in\cc$, it follows that $R$ also embeds in the
same holomorphic flow, yielding a Poincaré-Dulac model.
\item $R^{*}\holo[\Sigma]=\holo[\Omega]$ and $\mu=\frac{1}{2}$. According
to~\cite[Proposition 6.2]{BerTouze}, the foliation $\floc$ is
analytically conjugate to a foliation defined by $\omega=\omega_{k,\mu}+P\left(u\right)\dd u$
for some analytic germ $P$ and $d\in\zz_{\geqslant-1}\setminus\left\{ 0\right\} $,
which can actually be chosen as a polynomial~\cite[Theorem 2]{Tey-ExSN}
like the one in Proposition~\ref{prop:transverse_groupoid}. Therefore
$\holo[\Sigma]$ and $\holo[\Omega]$ are Bernoulli diffeomorphisms.
If moreover $P=0$, then we obtain again a Poincaré-Dulac model. Let
us prove that this is the only possible outcome by assuming $P\neq0$
in the sequel.\\
Requesting that $\eta=a\omega+b\dd u$ satisfy $\dd{\omega}=\eta\wedge\omega$
yields $b=dy^{d}u^{\sigma}P\left(u\right)$, and requesting that it
be closed yields $\dd a\wedge\omega=b\left(a-d\right)\dd x\wedge\dd y$.
Therefore $S=a-d$ is a separatrix of $\floc$, hence $S=y^{n}u^{m}s$
for some holomorphic function $s$ (which is either $0$ or does not
vanish at $0$) and some $n,~m\in\zz$. Indeed, the only branches
of a separatrix of $\floc$ are $\left\{ x=0\right\} $ and $\left\{ y=0\right\} $.\\
Let us first prove that $s=0$ by contraposition. Since $\frac{\dd{\left(y^{n}u^{m}\right)}}{y^{n}u^{m}}\wedge\omega=y\left(n\left(1+\mu u^{k}\right)+mu^{k}\right)\dd x\wedge\dd y$,
the non-vanishing holomorphic germ $s$ solves the cohomological equation
\begin{align*}
X\cdot\log s & =y^{d+1}u^{k+1}P\left(u\right)\left(d-n\right)-y\left(n\left(1+\mu u^{k}\right)+mu^{k}\right),
\end{align*}
where $X\left(x,y\right)=xu^{k}\pp x+\left(1+\mu u^{k}+u^{k+1}y^{d}P\left(u\right)\right)\left(x\pp x-y\pp y\right)$
is the vector field dual to $yu^{k+1}\omega$. The only way to reach
the terms $y\left(n\left(1+\mu u^{k}\right)+mu^{k}\right)$ is through
$X\cdot\left(y\phi\left(u\right)\right)$ for some holomorphic $\phi$.
Observe that $d\neq0$, hence $y^{d+1}u^{k+1}P\left(u\right)$ does
not mix up with the other terms, and $\phi$ must solve
\begin{align*}
u^{k+1}\phi'\left(u\right)+\left(1+\mu u^{k}\right)\phi\left(u\right) & =-n\left(1+\mu u^{k}\right)+mu^{k}.
\end{align*}
An immediate formal computation on $\phi\left(u\right)=\sum_{j\geqslant0}\phi_{j}u^{j}$
gives the only possible solution
\begin{align*}
\phi_{0}=-n, & ~~\phi_{j}=0\text{ for }0<j<k,\\
\phi_{k}=m~~\text{and } & \phi_{j+1}=\left(j-\mu\right)\phi_{j}\text{ for }k\leqslant j.
\end{align*}
The only way for this power series to have a positive radius of convergence
is to be a finite sum, so that $\mu\in\zz_{\geqslant k}$ and in particular
$\mu\neq\frac{1}{2}$.\\
So far we have proved $a=d$ and $b=dy^{d}P\left(u\right)$, from
which we deduce that $\eta=a\omega_{k,\mu}$ and $\exp\int\eta=y^{d}u^{-d\mu}\exp\frac{d}{ku^{k}}$.
Consequently, the condition expressed by Lemma~\ref{lem:integrable_regluing}
amounts to
\begin{align*}
\frac{1+\frac{1}{2}R^{k}+R^{k+1}P\left(R\right)}{R^{1+k-\nf d2}\exp\frac{dR^{-k}}{k}}R' & =\alpha\frac{1-\frac{1}{2}y^{k}+y^{d+k+1}P\left(y\right)}{y^{1+k+\nf d2}\exp\frac{dy^{-k}}{k}}.
\end{align*}
Since $d\neq0$ there is no solution belonging to $\diff$, as can
be seen by equating the respective first-order contributions of $R'\times R^{-1-k+\nf d2}$
and $\alpha y^{-1-k-\nf d2}$. This completes the classification.
\end{enumerate}
\end{enumerate}
\end{proof}
To conclude this classification, let us give an example of a Bernoulli
loop germ.
\begin{example}[Bernoulli model]
The linear foliation $\fol{1:1}$ with Godbillon-Vey sequence
\begin{align*}
\omega=\omega_{1:1}~~~\text{and~~~} & \eta=\frac{1+u}{u^{2}}\dd u~~~,~u=xy~,
\end{align*}
corresponds to the Liouvillian first integral
\begin{align*}
H & =\int\frac{\exp\frac{1}{u}}{u^{2}}\dd u=-\exp\frac{1}{u}.
\end{align*}
To obtain an associated loop germ with gluing map $R$, one must solve
the equation $HR=\alpha H+c$ for complex constants $\alpha$ and
$c$, that is:
\begin{align*}
\exp\frac{1}{R} & =\alpha\exp\frac{1}{y}-c.
\end{align*}
Every admissible gluing germ $R$ for $c=0$ has the form
\begin{align*}
R\left(y\right) & =\frac{y}{1+y\log\alpha}.
\end{align*}
For $c\neq0$ one must take $R$ as the time-$\log\alpha$ flow of
the vector field $\frac{u^{2}}{1+cu}\pp u$.
\end{example}

\appendix

\section{\protect\label{sec:commut}Some general commutator identities}

Given a group $G$, we define the \emph{commutator} of two elements
$x,y\in G$ by 
\[
[x,y]=x^{-1}y^{-1}xy.
\]
For the sake of brevity, we write the conjugation of $x$ by $y$
simply as 
\begin{align*}
x^{y} & =y^{-1}xy.
\end{align*}
Let us list some well-known identities 
\[
\begin{array}{ccl}
(\text{C}0) &  & [x,y]=x^{-1}x^{y}\\
(\text{C}1) &  & [x,y]^{-1}=[y,x]\\
(\text{C}2) &  & [x,y]^{z}=[x^{z},y^{z}]\\
(\text{C}3) &  & [x,y^{-1}]=[y,x]^{y^{-1}}\text{ and }[x^{-1},y]=[y,x]^{x^{-1}}\\
(\text{C}4) &  & [x,yz]=[x,z]\,[x,y]^{z}\text{ and }[xy,z]=[x,z]^{y}\,[y,z]
\end{array}
\]
Let $\tau\in G$ be fixed. The {\em $\tau$-functional variation}
(or simply the \emph{variation}) $\varD:G\rightarrow G$ is defined
as the commutator 
\[
\varD(x)=[\tau,x].
\]
Consider the sets 
\begin{equation}
\Cent_{0}(\tau)=\{\idD\},\qquad\Cent_{k}(\tau)=\{x\in G:\varD(f)\in\Cent_{k-1}(\tau)\},\quad k\ge1,\label{eq:definitions-Ck}
\end{equation}
where $\idD$ denotes the identity element. Notice that $\Cent_{1}(\tau)$
is a subgroup of $G$, since it is precisely the centralizer of $\tau$
in $G$. This property does not hold for $\Cent_{k}(\tau)$ for $k\ge2$
since the $\varD$ operator is not a morphism of groups.
\begin{lem}
\label{lem:x2} Let $x\in G$. Then 
\begin{equation}
\varD(x^{-1})^{x}=\varD(x)^{-1}.\label{eq:varfm}
\end{equation}
\end{lem}

\begin{proof}
This is immediate from the definition of $[\tau,x]$.
\end{proof}
We will be mostly interested in $\Cent_{2}(\tau)$, formed by the
set of group elements $x$ such that $[\tau,[\tau,x]]=\idD$. We now
show that such a set is closed under inversion and a partial composition
operation.
\begin{lem}
\label{lem:properties-var} Let $x,y\in\Cent_{2}(\tau)$. Then:
\begin{enumerate}
\item $x^{-1}\in\Cent_{2}(\tau)$;
\item $\varD(x)$ and $\varD(y^{-1})$ commute if and only if $xy$ belong
to $\Cent_{2}(\tau)$;
\item $\varD(x)=\varD(y^{-1})$ if and only if $xy$ belongs to $\Cent_{1}(\tau)$.
\end{enumerate}
\end{lem}

\begin{proof}
~
\begin{enumerate}
\item If we denote $z=\varD(x)^{-1}$ then, by~(\ref{eq:varfm}), it suffices
to prove that $z^{x^{-1}}$ is in $\Cent_{1}(\tau)$. We compute its
variation 
\[
\varD(z^{x^{-1}})=\big[\tau,z^{x^{-1}}\big]=\big[\tau^{x},z\big]^{x^{-1}}=\big[\tau z^{-1},z\big]^{x^{-1}}=\idD,
\]
where we have used the identity (C0) to write $\tau^{x}=\tau z^{-1}$
and the fact that $z$ commutes with $\tau$ to obtain the last equality.
\item Using identities~(C3),~(C4) and~(\ref{eq:varfm}), we write: 
\begin{equation}
\varD(xy)=\left[\tau,xy\right]=[\tau,y]\,[\tau,x]^{y}=\varD(y)\varD(x)^{y}=\left(\varD(y^{-1})^{-1}\varD(x)\right)^{y}.\label{eq:varxy-trivial}
\end{equation}
Put $z=\varD(x)$ and $w=\varD(y^{-1})$. Then:
\[
\left[\tau,\varD(xy)\right]=\left[\tau,(w^{-1}z)^{y}\right]=[\tau^{y^{-1}},w^{-1}z]^{y}=[\tau w,zw^{-1}]^{y},
\]
where in the last equality we have used that $\tau^{y^{-1}}=y\tau y^{-1}=\tau w$.
From the fact that $z$ and $w$ commute with $\tau$, we obtain,
using (C4), that 
\[
[\tau w,w^{-1}]=[w,z].
\]
\item is an immediate consequence of~(\ref{eq:varxy-trivial}).
\end{enumerate}
\end{proof}
Given an element $x\in\Cent_{2}(\tau)$, both the left coset $x\Cent_{1}(\tau)$
and the right coset $\Cent_{1}(\tau)x$ are entirely contained in
$\Cent_{2}(\tau)$. More precisely, for any $x\in G$ and $z\in\Cent_{1}(\tau)$,
a simple computation shows that 
\begin{equation}
\varD(xz)=\varD(x)^{z},\quad\text{and}\quad\varD(zx)=\varD(x).\label{eq:varxz}
\end{equation}
It follows that the right coset $\Cent_{1}(\tau)x$ is contained in
the set $\varD^{-1}(\varD(x))$. We prove that in fact 
\[
\Cent_{1}(\tau)x=\varD^{-1}(\varD(x)).
\]
Indeed, suppose that $y\in\Cent_{2}(\tau)$ is such that $\varD(x)=\varD(y)$.
Then, using (C4) we obtain 
\[
\varD(yx^{-1})=[\tau,yx^{-1}]=[\tau,x^{-1}]\,[\tau,y]^{x^{-1}}=\varD(x^{-1})\varD(x)^{x^{-1}}=\idD,
\]
where the last equality follows from~(\ref{eq:varfm}). Therefore,
$y=zx$ for some $z\in\Cent_{1}(\tau)$. This discussion can be restated
as follows.
\begin{prop}
\label{prop:cosetinj} The operator $\varD$ defines an embedding
\[
\varD:\Cent_{1}(\tau)\backslash\Cent_{2}(\tau)\longrightarrow\Cent_{1}(\tau)
\]
(which is not necessarily surjective).
\end{prop}

Let $x,y\in\Cent_{2}(\tau)$ be such that 
\begin{equation}
\varD(x)=\varD(y),\quad\text{and}\quad\varD(x^{-1})=\varD(y^{-1})\label{eq:equal-vars}
\end{equation}
Then, by the previous proposition, the left-hand equality gives $y=zx$,
for some $z\in\Cent_{1}(\tau)$. Inserting this in the right-hand
equality and using~(\ref{eq:varxz}), we obtain 
\[
\varD(x^{-1})=\varD(y^{-1})=\varD(x^{-1}z^{-1})=\varD(x^{-1})^{z^{-1}},
\]
which shows that $z$ belongs to the centralizer of $\varD(x^{-1})$
in $\Cent_{1}(\tau)$. We have thus proved the following result.
\begin{prop}
\label{prop:equalvars} Let $x,y\in\Cent_{2}(\tau)$. Then, the relations~(\ref{eq:equal-vars})
hold if and only if 
\[
y=c\,x,
\]
for some $c$ in the centralizer of $\varD(x^{-1})$.
\end{prop}

A simple computation shows that the centralizers of $\varD(x^{-1})$
and $\varD(x)$ are conjugate by $x$. Therefore, assuming (\ref{eq:equal-vars}),
we can equally write $y=x\,c$ for some $c$ in the centralizer of
$\varD(x)$.

\section{Derivations of the formal Dulac ring and normal forms}

\subsection{\protect\label{subsec:appB}Derivations}

We consider the set of differential operators 
\[
P\ee^{-\lambda z}\frac{\partial}{\partial z},
\]
for $(P,\lambda)\in\cC[z]\times\cR_{>0}$, seen as derivations in
the ring $\RingQSD$ of QSD-germs. This set is totally ordered with
respect to the $\succ$ relation defined in Section \ref{subsec:ramified-var}.
Furthermore, it is closed under the usual Lie bracket, since we have:
\begin{equation}
\left[P\ee^{-\lambda z}\frac{\partial}{\partial z},Q\ee^{-\mu z}\frac{\partial}{\partial z}\right]=\Big(P\partial_{\mu}(Q)-\partial_{\lambda}(P)Q\Big)\ee^{-(\lambda+\mu)z}\frac{\partial}{\partial z},\label{commutator}
\end{equation}
where $\partial_{\lambda}$ is the differential operator $\frac{\partial}{\partial z}-\lambda\id$.

A {\em nilpotent derivation} $X$ of the formal Dulac ring $\rDulacf$
is a formal sum of the above differential operators, namely
\[
X=\sum_{\lambda\in L}P_{\lambda}\ee^{-\lambda z}\frac{\partial}{\partial z},
\]
where $L\subset\cR_{>0}$ is a discrete subset. The set of such derivations
forms a Lie algebra $\DerDulacf$ under the bracket defined in \eqref{commutator}.

Notice that the condition on the exponents set $L$ guarantees that
the sequence of Dulac series $\left\{ z,X(z),X^{2}(z),..\right\} $
is formally convergent. More precisely, for each $\mu\in\cR_{>0}$,
there exists an index $k_{0}\in\cZ_{\ge0}$ such that $X^{k}(z)=\oo{\ee^{-\mu z}}$
for all $k\ge k_{0}$. In other words, the nilpotency condition guarantees
the the usual exponential series
\[
\Exp X=\idD+X+\frac{1}{2}X^{2}+\cdots
\]
is summable, and that $z\mapsto\left(\Exp X\right)(z)$ is an element
of the Dulac ring.

\smallskip{}

We now establish a correspondence between the group of formal Dulac
series with multiplier $1$ and the nilpotent derivations. We denote
by $\Dulacf_{>0}$ the normal subgroup of formal Dulac series having
an asymptotic expansion of the form 
\[
f=z+\oo 1.
\]
More specifically, given a $\lambda\in\cR_{>0}$, we say that a formal
Dulac series $f\in\Dulacf$ is {\em $\lambda$-tangent to identity}
if it has an expansion of the form: 
\[
f=z+P(z)\ee^{-\lambda z}+\oo{\ee^{-\lambda z}},
\]
for some polynomial $P$ (possibly zero). The set of $\lambda$-tangent
to identity formal series is a normal subgroup of $\Dulacf_{>0}$,
which we denote by $\Dulacf_{\ge\lambda}$. Similarly, we say that
a nilpotent derivation $X$ is {\em $\lambda$-flat} if it can
be written as: 
\[
X=P(z)\ee^{-\lambda z}\;\frac{\partial}{\partial z}+\oo{\ee^{-\lambda z}},
\]
for some polynomial $P$ (possibly zero). The sets of all such derivations
form a filtered collection of Lie ideals: 
\[
\Big[\DerDulacf_{\ge\lambda},\DerDulacf_{\ge\mu}\Big]\subset\DerDulacf_{\ge\lambda+\mu},
\]
as can be easily seen by the formula \eqref{commutator} for the Lie
bracket.

In what follows, we identify each formal Dulac series $f$ with the
corresponding {\em automorphism} of the formal Dulac ring defined
by
\[
p\in\rDulacf\mapsto p\circ f
\]
(which we will also denote by $f$). The following result has an immediate
proof.
\begin{lem}
\label{lem:exponentiallog} The formal exponential series establishes
a bijection 
\[
\Exp{}:\,\DerDulacf\longrightarrow\Dulacf_{>0},
\]
which maps $\DerDulacf_{\ge\lambda}$ to $\Dulacf_{\ge\lambda}$ for
each $\lambda>0$.
\end{lem}

The inverse map is given by the formal logarithmic series, which we
denote by $\Log{}$.

\subsection{BCH and variation operator}

By Lemma~\ref{lem:exponentiallog}, we can uniquely write an element
$f\in\Dulacf_{>0}$ as $f=\Exp X$ for some nilpotent derivation $X$.
The variation operator takes the form: 
\[
\varD(f)=\varD(\Exp X)=\tau^{-1}(\Exp{-X})\tau\Exp X=(\Exp{-X^{\tau}})(\Exp X),
\]
where $X^{\tau}=\tau^{-1}X\tau$ denotes the conjugation of $X$ by
the translation $\tau(z)=z+2\pi\ii$. In particular, if $X$ is one
of the basis elements $P(z)\ee^{-\lambda z}\frac{\partial}{\partial z}$,
\[
X^{\tau}=P(z-2\pi\ii)\ee^{-\lambda(z-2\pi\ii)}\frac{\partial}{\partial z}.
\]
If $X$ is $\lambda$-flat, then the same identity holds for $X^{\tau}$.
Let us denote by $\lvar(X)$ the unique derivation $Z$ such that
\[
\Exp Z=\varD(\Exp X).
\]
In other words, $\lvar=\Log{\circ\varD\circ\Exp{}}$. Using the usual
Baker-Campbell-Hausdorff expansion, we have: 
\begin{equation}
\lvar(X)=X-X^{\tau}+\frac{1}{2}[X,X^{\tau}]+\cdots,\label{eq:BCH-Z}
\end{equation}
where the omitted terms are higher order commutators that vanish if
$[X,X^{\tau}]=0$. Moreover, for $X\in\DerDulacf_{\ge\lambda}$, the
omitted terms belongs to $\DerDulacf_{\ge3\lambda}$. Finally, note
that $\varD(\Exp X)=\idD$ if and only if $\lvar(X)=0$.
\begin{prop}
\label{prop:ramifder} Let $f=\Exp X$, for some nilpotent derivation
\[
X=\sum_{\lambda\in L}P_{\lambda}\ee^{-\lambda z}\frac{\partial}{\partial z}.
\]
Then the following three statements are equivalent:
\begin{enumerate}
\item $\varD(f)=\idD$ (\emph{i.e.} $f$ is unramified);
\item $X=X^{\tau}$;
\item $L\subset\cZ_{\ge1}$ and all polynomials $P_{\lambda}$ are constants.
\end{enumerate}
\end{prop}

\begin{proof}
As remarked above, $\varD(f)=\idD$ if and only if $\lvar(X)=0$.
Using the asymptotic expansion given by the BCH formula~(\ref{eq:BCH-Z}),
this is equivalent to saying that $X=X^{\tau}$. Hence, 1. and 2.
are equivalent.

In order to prove the equivalence of 2. and 3., it suffices to observe
that the conjugation of $X$ by $\tau$ preserves the order of the
asymptotic terms in the expansion of $X$. In other words, the equality
$X=X^{\tau}$ is equivalent to requesting 
\[
\Big(P_{\lambda}\ee^{-\lambda z}\frac{\partial}{\partial z}\Big)^{\tau}=P_{\lambda}\ee^{-\lambda z}\frac{\partial}{\partial z}
\]
for each $\lambda\in L$. This holds if and only if $P_{\lambda}$
lies in the kernel of the linear operator $\cC[z]\rightarrow\cC[z]$
defined by 
\[
P(z)\longmapsto\ee^{2\pi\ii\lambda}P(z-2\pi\ii)-P\left(z\right)
\]
This is equivalent to saying that $\lambda$ is an integer and that
$P_{\lambda}$ is of degree 0.
\end{proof}
By analogy with the notations introduced in Section~\ref{subsec:ramifmildramif},
we say that a nilpotent derivation $X$ is:
\begin{itemize}
\item {\em unramified} if $\lvar(X)=0$;
\item {\em mildly ramified} if $\lvar(X)$ is unramified.
\end{itemize}
Notice that these conditions are equivalent to $\Exp X\in\U$ and
$\Exp X\in\MR$, respectively.

\subsection{Computing $\protect\lvar^{-1}$ and normal forms}

Let $Z$ be an unramified derivation. We are interested in determining
another nilpotent derivation $X$ such that 
\[
Z=\lvar(X).
\]
Consider the difference operator $\Differ:\cC[z]\rightarrow\cC[z]$
defined by 
\[
(\Differ P)(z)=P(z)-P(z-2\pi\ii)
\]
or equivalently, $\Delta P=P-P\tau^{-1}$. It is easy to see that
the sequence of polynomials 
\begin{equation}
\pP_{0}=1,\quad\pP_{1}=\frac{1}{2\pi\ii}\;z,\quad\pP_{k}=\frac{1}{(2\pi\ii)^{k}}\;z(z+2\pi\ii)\cdots(z+2\pi\ii(k-1))\label{basis-pol}
\end{equation}
satisfies $\Delta\pP_{k}=k\pP_{k-1}$. In particular, $\pP_{0}$ generates
the kernel of $\Differ$.
\begin{prop}
\label{prop:inverselvar} Let $Z$ be an unramified nilpotent derivation.
There exist a nilpotent derivation $X$ such that 
\[
Z=\lvar(X).
\]
Moreover, $X$ is uniquely determined \emph{modulo} the choice of
a section of the cokernel of $\Delta$.
\end{prop}

\begin{proof}
Let us write the expansion of $Z$ as 
\[
Z=\sum_{k\ge1}c_{k}\ee^{-kz}\frac{\partial}{\partial z},
\]
where $c_{k}$ are complex numbers. We now look for an expansion of
$X$ in the form 
\[
X=\sum_{k\ge1}P_{k}\ee^{-kz}\frac{\partial}{\partial z}.
\]
It follows from $Z=\lvar(X)$ and the BCH formula~(\ref{eq:BCH-Z})
that the polynomials $P_{k}$ can be recursively determined by the
formula 
\[
\Delta(P_{k})=c_{k}+R,
\]
where $R$ depends only on previously determined polynomials $P_{\ell}$
for $0<\ell<k$. In other words, $P_{k}$ is uniquely determined from
this formula \emph{modulo} an additive constant.
\end{proof}
\begin{rem}
This algorithm can also be used to compute $\lvar^{-1}(Z)$ for arbitrary
(\emph{i.e.} not necessarily unramified) nilpotent derivations
\[
Z=\sum_{L}Q_{\lambda}\ee^{-\lambda z}\frac{\partial}{\partial z}.
\]
In this case, we eventually need to consider the twisted difference
operators $\Delta^{\lambda}P=P(z)-\ee^{2\pi\ii\lambda}P(z-2\pi\ii)$,
which are {\em invertible} for $\lambda\notin\cZ$. In particular,
if the additive semi-group generated by the exponent set $L$ is disjoint
from $Z_{\ge0}$, then $\lvar^{-1}(Z)$ is uniquely determined.
\end{rem}

We reprove the following normal form result.
\begin{prop}
\label{prop:nfunramif} Let $Z$ be a nilpotent, unramified derivation
with an expansion 
\[
Z=c_{k}\ee^{-kz}\frac{\partial}{\partial z}+\oo{\ee^{-kz}},
\]
such that $c_{k}$ is nonzero. Then, there exists a formal unramified
series $g\in\Uf$ which conjugates $Z$ to a derivation of the form
\[
-2\pi\ii\frac{\ee^{-kz}}{1+\mu\ee^{-kz}}\frac{\partial}{\partial z}+\oo{\ee^{-3kz}},
\]
where $k\in\zz_{\geqslant1}$ and $\mu\in\cC$ (called the {\em
residue} of $Z$) are formal invariants.
\end{prop}

\begin{rem}
It is well-known that actually $g\in\mathcal{U}$. This stems from
the fact that only a finite number of reduction steps are applied.
\end{rem}

\begin{proof}
First of all, we conjugate $Z$ by a germ of the form $g_{0}(z)=z+b$,
$b\in\cc$, to obtain 
\[
Z^{g_{0}}=c_{k}\ee^{-kb}\ \ee^{-kz}\frac{\partial}{\partial z}+\oo{\ee^{-kz}}.
\]
Therefore, up to a convenient choice of $b$, we can assume that $c_{k}=-2\pi\ii$.
Then, considering the Lie bracket identity 
\[
\Big[\ee^{-jb}\frac{\partial}{\partial z},\ee^{-kz}\frac{\partial}{\partial z}\Big]=(j-k)\ee^{-(k+j)z}\frac{\partial}{\partial z},
\]
we observe that we can eliminate all terms $\{c_{k+j}\}_{1<j<k-1}$,
by successive conjugations by germs of the form 
\[
g_{j}=\Exp{\Big(b_{j}\ee^{-jz}\frac{\partial}{\partial z}\Big)},\qquad1\leqslant j<k-1.
\]
The remaining term $j=k$ corresponds to a resonance. It is an invariant
for $\Uf$-conjugation.
\end{proof}
\begin{rem}
Although not needed in the sequel, observe that we further reduce
the above expression by a formal series $g\in\Uf$ in order to eliminate
all $\oo{\ee^{-3kz}}$ terms and get simply 
\[
-2\pi\ii\frac{\ee^{-kz}}{1+\mu\ee^{-kz}}\frac{\partial}{\partial z}.
\]
In fact, each $\Uf$-conjugacy class contains a unique representative
of this form.
\end{rem}

\begin{thm}
\label{thm:solvevar} Let $X$ be a mildly ramified (but not unramified)
nilpotent derivation. Then there exists a formal unramified series
$g\in\U$ which conjugates $X$ to the form 
\[
-(z+a)\ee^{-kz}\frac{\partial}{\partial z}+\left(\Big(\mu-\frac{1}{2}\Big)z+b\right)\ee^{-2kz}\frac{\partial}{\partial z}+\oo{\ee^{-2kz}}
\]
for some $k\in\cZ_{\ge1}$ and $a,b,\mu\in\cC$.
\end{thm}

\begin{proof}
By the assumption, the derivation $Z=\lvar(X)$ is a nontrivial unramified
derivation. It follows from Proposition~\ref{prop:nfunramif} that
there exists a formal unramified series $g$ such that 
\[
Z^{g}=-2\pi\ii\frac{\ee^{-kz}}{1+\mu\ee^{-kz}}\frac{\partial}{\partial z}+\oo{\ee^{-3kz}}.
\]
It suffices to show that $X^{g}=\lvar^{-1}(Z^{g})$ has the desired
form.

In fact, it suffices to look at the initial steps of the algorithm
described in the proof of Proposition~\ref{prop:inverselvar}. Using
only the first two terms in the expansion 
\[
Z=X-X^{\tau}+\frac{1}{2}[X,X^{\tau}]+\cdots,
\]
it follows that $X^{g}$ has an expansion of the form 
\[
X^{g}=P\ee^{-kz}\frac{\partial}{\partial z}+Q\ee^{-2kz}\frac{\partial}{\partial z}+\oo{\ee^{-2kz}},
\]
where $P$ and $Q$ satisfy the following polynomial equations: 
\[
\Delta P=-2\pi\ii\quad\text{and}\quad\Delta Q+\frac{1}{2}(P\left(P\tau^{-1}\right)^{\prime}-P\tau^{-1}P^{\prime})=2\pi\ii\mu.
\]
The first equation has general solution $P=-(z+a)$ (with an arbitrary
$a\in\cC$). Then 
\[
\frac{1}{2}(P\left(P\tau^{-1}\right)^{\prime}-P\tau^{-1}P^{\prime})=\frac{1}{2}\Big((z+a)-(z+a-2\pi\ii)\Big)=\ii\pi,
\]
and we conclude that $Q=\left(\mu-\frac{1}{2}\right)z+b$, for an
arbitrary $b\in\cC$.
\end{proof}
\begin{cor}
\label{cor:finalformvars}Let $X$ be as in the enunciate of the Theorem.
Then, up to unramified conjugation, we have 
\begin{equation}
\lvar(X)=-2\pi\ii\frac{\ee^{-kz}}{1+\mu\,\ee^{-kz}}\frac{\partial}{\partial z}+\oo{\ee^{-2kz}}\label{first-vf}
\end{equation}
and 
\begin{equation}
\lvar(-X)=2\pi\ii\frac{\ee^{-kz}}{1+\Big(\mu-1\Big)\ee^{-kz}}\frac{\partial}{\partial z}+\oo{\ee^{-2kz}}.\label{second-vf}
\end{equation}
As a consequence, denoting $Z=\lvar(X)$ and $W=\lvar(-X)$, we get:
\[
[W,Z]=\left(4\pi^{2}k\ee^{-3kz}+\oo{\ee^{-3kz}}\right)\frac{\partial}{\partial z}.
\]
\end{cor}

\begin{proof}
This is a simple computation using the formulas given above.
\end{proof}

\section{Mildly ramified germs with rational multiplier}

Let $f\in\MR\backslash\mathcal{U}$ be a mildly ramified germ with
linear part $\frac{p}{q}z$ for some coprime positive integers $p,q\neq1$.
As usual, we denote by $\Var f<\diff$ the group generated by the
images of $\varD(f)$ and $\varD(f^{-1})$ under the morphism $\Pi_{*}:\mathcal{U}\rightarrow\diff$.

Although not needed for the proof of Theorem~D, the integrability
result of Section~\ref{sec:integrability} uses some finer properties
of $\varD(f)$ that may also be useful for proving topological rigidity
results, which we summarize in the next proposition.
\begin{prop}
\label{prop:solvable_pq}The group $\Var f$ is formally rigid. Moreover,
if it is solvable, then it is analytically conjugate to 
\[
\left\langle \ee^{2\pi\ii\nicefrac{q}{p}}x,\frac{\ee^{2\pi\ii\nicefrac{p}{q}}x}{(1+\varepsilon x^{k})^{1/k}}\right\rangle ,
\]
where $k$ is some positive integer not in $p\mathbb{Z}\cup q\mathbb{Z}$
and $\varepsilon\in\left\{ 0,1\right\} $.
\end{prop}

\begin{rem}
Both generators of the group are analytically linearizable, but in
general not in the same coordinates. The group itself is linearizable
if and only if it is abelian, which happens exactly when $\varepsilon=0$.
\end{rem}

\begin{proof}
Up to a conjugation inside $\mathcal{U}$, we can suppose that 
\[
f=\frac{p}{q}z+\oo 1
\]
and the variation $g=\varD(f)$ has the form 
\[
g=t_{2\pi\ii\left(\frac{q}{p}-1\right)}\tilde{g}
\]
where $t_{c}(z)=z+c$ is the translation map and $\tilde{g}$ is a
tangent to identity unramified germ. We split the proof into two cases.
\begin{lyxlist}{00.00.0000}
\item [{Case~1.}] Up to a conjugation inside $\widehat{\mathcal{U}}$,
we can write $g=\varD(f)$ in prepared form 
\[
g=t_{2\pi\ii\left(\frac{q}{p}-1\right)}\exp\left(-2\pi\ii\ee^{-kpz}\frac{\partial}{\partial z}+\oo{\ee^{-kpz}}\right)
\]
Using the fact that $h=\varD(f^{-1})=f\varD(f)^{-1}f^{-1}$, we obtain
\[
h=t_{2\pi\ii\left(\frac{p}{q}-1\right)}\exp\left(-2\pi\ii\ee^{-kqz}\frac{\partial}{\partial z}+\oo{\ee^{-kqz}}\right)
\]
Therefore, $G=\Pi_{*}g$ and $H=\Pi_{*}h$ are such that $G^{p}$
and $H^{q}$ do not commute. We conclude that the group $\langle G,H\rangle$
is not solvable.
\item [{Case~2.}] It remains to consider the case $g=\varD(f)$ when it
is $\mathcal{U}$-linearizable, that is (after convenient conjugation):
\[
g=t_{2\pi\ii\left(\frac{q}{p}-1\right)}
\]
and, according to Proposition~A.1, the Dulac germ has the form $f=rs_{p/q}$
for some tangent to identity unramified germ $r$. As a consequence,
$h=\varD(f^{-1})$ has the form 
\[
rt_{2\pi\ii\left(\frac{p}{q}-1\right)}r^{-1}.
\]
Therefore $f$ is the Poincaré map of a model $(\text{\ensuremath{\fol{p:q}}},R)$,
where $\fol{p:q}$ is the germ of a foliation generated by the linear
differential 1-form $\omega_{p:q}=\dd{\left(x^{q}y^{p}\right)}$,
and $R\in\tmop{Diff}(\mathbb{C},0)$ is an arbitrary holomorphic germ.
In particular, the variation group is generated by two germs of a
diffeomorphism $G,H$ with respective linear parts $\ee^{2\pi\ii q/p}x$
and $\ee^{2\pi\ii p/q}x$, and such that 
\[
G^{p}=H^{q}=\id.
\]
We now use a result of \cite{Loray:Meziani}, which establishes the
following trichotomy:
\begin{lyxlist}{00.00.0000}
\item [{Case~2.a.}] Either $\langle G,H\rangle$ is non-solvable, or
\item [{Case~2.b.}] $\langle G,H\rangle$ is abelian and analytically
linearizable, or
\item [{Case~2.c.}] $\langle G,H\rangle$ is metabelian and, in appropriate
formal coordinates, has the form 
\[
\langle G,H\rangle=\left\langle \alpha x,\frac{\beta x}{(1+x^{k})^{1/k}}\right\rangle 
\]
where we note note $\alpha=\ee^{2\pi\ii q/p},\beta=\ee^{2\pi\ii p/q}$
and $k$ is some positive integer in $\mathbb{Z}_{\geqslant1}\setminus\{p\mathbb{Z}\cup q\mathbb{Z}\}$
\end{lyxlist}
\begin{proof}[Proof based on \cite{Loray:Meziani}.]
 According to \cite[ Proposition I.1]{Loray:Meziani}, in the metabelian
case we can write in convenient formal coordinates 
\[
G=s_{\alpha}\exp\left(ax^{k}x\frac{\partial}{\partial x}\right),H=s_{\beta}\exp\left(bx^{k}x\frac{\partial}{\partial x}\right),
\]
where $s_{\lambda}(z)=\lambda z$ is the scaling map. We now impose
the conditions $H^{q}=G^{p}=\id$ and $[G,H]\neq\id$.

First of all, 
\begin{eqnarray*}
\id & = & H^{q}\\
 & = & \left(s_{\beta}\exp\left(bx^{k}x\frac{\partial}{\partial x}\right)\right)^{q}\\
 & = & s_{\beta}\exp\left(bx^{k}x\frac{\partial}{\partial x}\right)s_{\beta}\exp\left(bx^{k}x\frac{\partial}{\partial x}\right)\cdots s_{\beta}\exp\left(bx^{k}x\frac{\partial}{\partial x}\right)\\
 & = & \exp\left(\left(\beta^{k(1-q)}+\beta^{k(2-q)}+\cdots+1\right)bx^{k}x\frac{\partial}{\partial x}\right),
\end{eqnarray*}
which implies that $k$ is not a multiple of $q$. Similarly, from
$G^{p}=\id$ we get that $k\notin p\zz$. A simple computation finally
gives 
\begin{eqnarray*}
{}[G,H] & = & \exp\left(\left(a(\beta^{-k}-1)-b(\alpha^{-k}-1)\right)x^{k}x\frac{\partial}{\partial x}\right),
\end{eqnarray*}
which implies that $a(\beta^{-k}-1)+b(1-\alpha^{-k})\neq0$. Notice
that 
\[
\exp\left(-\lambda x^{k}x\frac{\partial}{\partial x}\right)g\exp\left(\lambda x^{k}x\frac{\partial}{\partial x}\right)=s_{\alpha}\exp\left(a+\lambda(1-\alpha^{-k})x^{k}x\frac{\partial}{\partial x}\right).
\]
Therefore, upon conjugating by $\exp\left(\lambda x^{k}x\frac{\partial}{\partial x}\right)$
with $\lambda$ such that 
\[
a+\lambda(1-\alpha^{-k})=0
\]
and by a further scaling, we obtain 
\[
G=\alpha x\quad\text{and}\quad H=\frac{\beta x}{(1+x^{k})^{1/k}}.
\]
\end{proof}
\end{lyxlist}
Finally, in order to prove the formal rigidity, it suffices to consider
the case where $\tmop{Var}(f)$ is metabelian, because non-solvable
groups are rigid, and we just proved that abelian groups are analytically
linearizable. According to \cite[Section 1, Remarque 1]{CerMou},
the group $\langle G,H\rangle$ is exceptional (\emph{i.e.} not rigid)
if and only if 
\[
\ee^{2\pi\ii\frac{q}{p}}=\ee^{i\pi\frac{n_{1}}{k}},\quad\ee^{2\pi\ii\frac{p}{q}}=\ee^{i\pi\frac{n_{1}}{k}}
\]
with $n_{1},n_{2}$ odd integers. In other words, 
\[
\frac{q}{p}\equiv\frac{n_{1}}{2k}\mod{\mathbb{Z}}~~\text{ and }~~\frac{p}{q}\equiv\frac{n_{2}}{2k}\mod{\mathbb{Z}}
\]
But this is impossible since it would imply that both $q$ and $p$
are even numbers. Therefore, $\langle G,H\rangle$ is always rigid,
and in particular is \emph{analytically} conjugate to $\left\langle \alpha x,\frac{\beta x}{(1+x^{k})^{1/k}}\right\rangle $.
\end{proof}
\bibliographystyle{bibliography/preprints}
\bibliography{bibliography/bibliography}

\end{document}